\documentclass[a4paper]{amsart}

\usepackage{amssymb,latexsym}

\author[W.\thinspace{}Kim]{Wansu Kim}
\address{Wansu Kim\\%
Centre for Mathematical Sciences\\%
University of Cambridge\\%
Cambridge, CB3 0WA\\%
United Kingdom}
\email{wk259@dpmms.cam.ac.uk}

\usepackage[active]{srcltx}
\usepackage{amsmath}
\usepackage{amsfonts}

\usepackage{amsthm}
\usepackage{amscd}
\usepackage{mathrsfs}

\usepackage{fancyhdr}
\usepackage{graphicx}
\usepackage{psfrag}
\usepackage{calc}
\usepackage{url}

\usepackage{pb-diagram,pb-xy}
\usepackage{stmaryrd}
\usepackage{times}
\usepackage{verbatim}
\usepackage[cmtip,arrow,matrix,curve,tips,frame]{xy}
\usepackage{hyperref}
\xyoption{matrix}

\usepackage{ifthen}
\usepackage{calc}
\usepackage{epsfig}
\usepackage{color}
\usepackage{amsxtra}
\usepackage{amstext}
\usepackage{amssymb}
\usepackage{stmaryrd}
\usepackage{mathrsfs}
\usepackage{pifont}
\usepackage{charter}
\usepackage{avant}
\usepackage{courier}
\usepackage[T1]{fontenc}
\usepackage{textcomp}

\numberwithin{equation}{subsection}

\theoremstyle{plain}
\newtheorem{thm}[subsection]{Theorem}
\newtheorem*{thm*}{Theorem}
\newtheorem{nothm}{Theorem}
\newtheorem{thmsub}[equation]{Theorem}

\newtheorem*{exthm*}{Expected Theorem}

\newtheorem{lem}[subsection]{Lemma}
\newtheorem{lemsub}[equation]{Lemma}
\newtheorem*{lem*}{Lemma}

\newtheorem{prop}[subsection]{Proposition}

\newtheorem{propsub}[equation]{Proposition}
\newtheorem*{prop*}{Proposition}
\newtheorem{noprop}[nothm]{Proposition}

\newtheorem{cor}[subsection]{Corollary}
\newtheorem{corsub}[equation]{Corollary}
\newtheorem*{cor*}{Corollary}
\newtheorem{nocor}[nothm]{Corollary}

\newtheorem{claimsub}[equation]{Claim}
\newtheorem*{claim*}{Claim}

\newtheorem*{conj*}{Conjecture}

\theoremstyle{definition}
\newtheorem{defn}[subsection]{Definition}
\newtheorem*{defn*}{Definition}
\newtheorem{defnsub}[equation]{Definition}

\newtheorem{exasub}[equation]{Example}
\newtheorem*{exa*}{Example}

\theoremstyle{remark}

\newtheorem{rmksub}[equation]{Remark}
\newtheorem*{rmk*}{Remark}

\numberwithin{figure}{subsection}
\numberwithin{table}{subsection}

\newcounter{listnum}

\newcommand{\eE}{\mathbf{E}}

\DeclareMathOperator{\im}{im}

\DeclareMathOperator{\coker}{coker}

\DeclareMathOperator{\pr}{pr}

\DeclareMathOperator{\rank}{rank}

\DeclareMathOperator{\id}{id}

\newcommand{\e}{\mathbf{e}}

\DeclareFontEncoding{OT2}{}{} 

\newcommand{\nf}[1]{\underline{#1}}
\newcommand{\ol}[1]{\overline{#1}}
\newcommand{\wt}[1]{\widetilde{#1}}

\newcommand{\tim}{\!\cdot\!}
\newcommand{\geqs}{\geqslant}
\newcommand{\leqs}{\leqslant}

\newcommand{\et}{\text{\rm\'et}}

\DeclareMathOperator{\red}{red}

\newcommand{\wh}[1]{\widehat{#1}}

\newcommand{\XX}{\mathfrak{X}}

\newcommand{\E}{\mathscr{E}}

\newcommand{\cO}{\mathcal{O}}

\newcommand{\cQ}{\mathcal{Q}}

\newcommand{\M}{\mathcal{M}}

\DeclareMathOperator{\Sym}{Sym}

\DeclareMathOperator{\Spec}{Spec}
\DeclareMathOperator{\MaxSpec}{MaxSpec}
\DeclareMathOperator{\Spf}{Spf}

\DeclareMathOperator{\perf}{perf}

\newcommand{\starr}{^\times}
\newcommand{\N}{\mathcal{N}}

\DeclareMathOperator{\Frac}{Frac}

\newcommand{\ra}{\rightarrow}
\newcommand{\xra}[1]{\xrightarrow{#1}}

\newcommand{\hra}{\hookrightarrow}

\newcommand{\thra}{\twoheadrightarrow}

\newcommand{\la}{\leftarrow}

\newcommand{\riso}{\xrightarrow{\sim}}
\newcommand{\liso}{\stackrel{\sim}{\la}}

\newcommand{\com}[1]{^{(#1)}}

\newcommand{\invlim}{\mathop{\varprojlim}\limits}

\newcommand{\Gmhat}{\widehat{\mathbb{G}}_m}

\newcommand{\set}[1]{\{#1\}}

\newcommand{\iv}{^{-1}}
\newcommand{\ivtd}[1]{\tfrac{1}{#1}}

\newcommand{\eps}{\varepsilon}
\newcommand{\Eps}{\mathcal{E}}
\newcommand{\vphi}{\varphi}

\newcommand{\Sig}{\mathfrak{S}}

\newcommand{\Q}{\mathbb{Q}}

\newcommand{\R}{\mathbb{R}}

\newcommand{\cD}{\mathcal{D}}

\newcommand{\PP}{E}

\newcommand{\C}{\mathbb{C}}

\newcommand{\CK}{{\mathbb C}_K}

\newcommand{\Qbar}{\overline{\Q}}
\newcommand{\kbar}{\bar{k}}

\newcommand{\cT}{\mathcal{T}}

\newcommand{\Z}{\mathbb{Z}}

\newcommand{\F}{\mathbb{F}}

\newcommand{\Qp}{\Q_p}

\newcommand{\Kbar}{\overline{K}}

\newcommand{\Zp}{\Z_p}

\newcommand{\Fp}{\F_p}

\DeclareMathOperator{\nilp}{nilp}

\newcommand{\fo}{\mathscr{O}}

\DeclareMathOperator{\Aut}{Aut}

\DeclareMathOperator{\ord}{ord}

\DeclareMathOperator{\Gal}{Gal}
\newcommand{\gal}{\boldsymbol{\mathcal{G}}}

\newcommand{\GRR}{\gal_R}

\newcommand{\GRinfty}{\gal_{R_\infty}}
\newcommand{\GRtinfty}{\gal_{\widetilde{R}_\infty}}

\newcommand{\p}{\mathfrak{p}}

\newcommand{\q}{\mathfrak{q}}
\newcommand{\m}{\mathfrak{m}}

\newcommand{\gF}{\mathfrak{F}}

\newcommand{\gM}{\mathfrak{M}}
\newcommand{\gN}{\mathfrak{N}}

\DeclareMathOperator{\ur}{ur}

\DeclareMathOperator{\Hom}{Hom}

\newcommand{\llie}{\mathscr{L}\mathit{ie}}

\DeclareMathOperator{\HH}{H}
\newcommand{\coh}[1]{\HH^{#1}}

\DeclareMathOperator{\Ext}{Ext}

\newcommand{\rad}{\mathtt{rad}}

\DeclareMathOperator{\Fil}{Fil}
\DeclareMathOperator{\gr}{gr}

\DeclareMathOperator{\cris}{cris}
\DeclareMathOperator{\CRIS}{CRIS}
\newcommand{\Acris}{A_{\cris}}
\newcommand{\Ainf}{A_{\inf}}

\newcommand{\Bcris}{B_{\cris}}
\newcommand{\Dcris}{D_{\cris}}
\newcommand{\DdR}{D_{\dR}}
\newcommand{\Vcris}{V_{\cris}}
\newcommand{\Tcris}{T_{\cris}}

\DeclareMathOperator{\dR}{dR}

\newcommand{\BdR}{B_{\dR}}

\DeclareMathOperator{\tor}{tor}
\DeclareMathOperator{\free}{free}

\newcommand{\MFF}{\mathbf{MF}}
\newcommand{\RMF}{\mathbf{MF}_{R/R_0}(\vphi,\nabla)}

\newcommand{\RMFa}{\mathbf{MF}^{\mathrm{a}}_{R/R_0}(\vphi,\nabla)}

\newcommand{\SMF}{\mathbf{MF}_{S}(\vphi,\nabla)}
\newcommand{\SMFw}{\mathbf{MF}^{\mathrm{w}}_{S}(\vphi,\nabla^u)}
\newcommand{\SMFq}{\mathbf{MF}_{S}(\vphi)}
\newcommand{\SMFqw}{\mathbf{MF}^{\mathrm{Br}}_{S}(\vphi,\nabla^0)}
\newcommand{\SMFtor}{\mathbf{MF'}_{S}(\vphi,\nabla)^{\tor}}

\newcommand{\SMFqtor}{\mathbf{MF'}_{S}(\vphi)^{\tor}}
\newcommand{\SMFqwtor}{\mathbf{MF}^{\prime\mathrm{Br}}_{S}(\vphi,\nabla^0)^{\tor}}

\newcommand{\wtSMF}{\mathbf{MF}_{\wt S}(\vphi,\nabla)}

\newcommand{\DMF}{\mathbf{MF}_{\wh D}(\vphi,\nabla)}
\newcommand{\DMFq}{\mathbf{MF}_{\wh D}(\vphi)}
\newcommand{\AMF}[1]{\mathbf{MF}_{\Acris(#1)}(\vphi,\nabla)}

\newcommand{\PMF}[1]{\mathbf{MF}_{#1}(\vphi,\nabla)}

\newcommand{\PMFq}[1]{\mathbf{MF}_{#1}(\vphi)}
\newcommand{\PMFqw}[1]{\mathbf{MF}^{\mathrm{Br}}_{#1}(\vphi,\nabla^0)}
\newcommand{\PMFtor}[1]{\mathbf{MF}_{#1}(\vphi,\nabla)^{\tor}}

\newcommand{\SM}{\mathbf{Mod}_{\Sig}(\vphi,\nabla)}
\newcommand{\SMq}{\mathbf{Mod}_{\Sig}(\vphi)}
\newcommand{\SMqw}{\mathbf{Mod}^{\mathrm{Ki}}_{\Sig}(\vphi,\nabla^0)}
\newcommand{\SMtor}{\mathbf{Mod}_{\Sig}(\vphi,\nabla)^{\tor}}
\newcommand{\SMqtor}{\mathbf{Mod}_{\Sig}(\vphi)^{\tor}}
\newcommand{\SMqwtor}{\mathbf{Mod}^{\mathrm{Ki}}_{\Sig}(\vphi,\nabla^0)^{\tor}}

\newcommand{\PMq}[1]{\mathbf{Mod}_{#1}(\vphi)}
\newcommand{\PMqw}[1]{\mathbf{Mod}^{\mathrm{Ki}}_{#1}(\vphi,\nabla^0)}
\newcommand{\PMtor}[1]{\mathbf{Mod}_{#1}(\vphi,\nabla)^{\tor}}
\newcommand{\PMqtor}[1]{\mathbf{Mod}_{#1}(\vphi)^{\tor}}

\newcommand{\SMFI}{\mathbf{(Mod\ FI)}_{\Sig}(\vphi,\nabla)}
\newcommand{\SMFIq}{\mathbf{(Mod\ FI)}_{\Sig}(\vphi)}
\newcommand{\SMFIqw}{\mathbf{(Mod\ FI)}^{\mathrm{Ki}}_{\Sig}(\vphi,\nabla^0)}

\newcommand{\EtPhiR}{\mathbf{Mod}^{\et}_{\fo_{\Eps}}(\vphi)}

\newcommand{\EtPhiPrR}{\mathbf{Mod}^{\et,\pr}_{\fo_{\Eps}}(\vphi)}

\newcommand{\EtPhiTorR}{\mathbf{Mod}^{\et,\tor}_{\fo_{\Eps}}(\vphi)}

\newcommand{\LL}{\mathbb{L}}

\DeclareMathOperator{\Rep}{Rep}

\newcommand{\prep}{\Rep_{\Zp}}

\newcommand{\freeprep}{\Rep_{\Zp}^{\free}}
\newcommand{\torprep}{\Rep_{\Zp}^{\tor}}

\newcommand{\Repcris}{\Rep^{\cris}}

\DeclareMathOperator{\sh}{sh}
\DeclareMathOperator{\Isog}{Isog}

\newcommand{\DD}{\mathbb{D}}

\title[Classification of $p$-divisible groups]{The relative Breuil-Kisin classification of $p$-divisible groups and finite flat group schemes}

\begin{document} 

\begin{abstract}
Assume that $p>2$, and let $\fo_K$ be a $p$-adic discrete valuation ring with residue field admitting a finite $p$-basis,  and let $R$ be a formally smooth formally finite-type $\fo_K$-algebra. (Indeed, we allow slightly more general rings $R$.) We construct an  anti-equivalence of categories between the categories of $p$-divisible groups over $R$ and certain semi-linear algebra objects which generalise $(\vphi,\Sig)$-modules of height $\leqs1$ (or Kisin modules). A similar classification result for $p$-power order finite flat group schemes is deduced from the classification of $p$-divisible groups. We also show compatibility of various construction of ($\Zp$-lattice or torsion) Galois representations, including the relative version of Faltings' integral comparison theorem for $p$-divisible groups. We obtain partial results when $p=2$.
\end{abstract}

\maketitle
\tableofcontents

\section{Introduction}\label{sec:intro}
One of the main motivations of $p$-adic Hodge theory, initiated by Fontaine, is to prove comparison isomorphisms between various $p$-adic cohomologies of varieties over $p$-adic fields  (such as \'etale, crystalline, and de~Rham cohomologies), and we now have very satisfying theory. Recently, the formalism of $p$-adic Hodge theory  (such as period rings and admissible representations) have been generalised  to the relative setting, most notably by Brinon. See \cite{Brinon:Habilitation} for a summary of the current status of ``relative $p$-adic Hodge theory'', and \cite{Scholze:CdR} for more recent developments built upon the theory of perfectoid spaces \cite{Scholze:Perfectoid}.

One of the most accessible ``test cases'' of (absolute) $p$-adic Hodge theory is $p$-divisible groups. This case is also important in its own right because it is closely related to the study of abelian varieties   with good reductions. Therefore, it is natural to ask what we can say about $p$-divisible groups over a more general base with the formalism of relative $p$-adic Hodge theory.

In this paper we generalise the classifications of $p$-divisible group by strongly divisible modules and $(\vphi,\Sig)$-modules of height $\leqs1$ over a $p$-adic affine formal base which is formally smooth over some $p$-adic discrete valuation ring (with some reasonable finiteness condition). 
We also recover, when $p>2$, the $p$-adic integral Tate module of a $p$-divisible group from the corresponding semi-linear algebraic objects. 

Let us describe our main  results in a simplified setting. (For the actual assumptions on the base ring $R$, we refer to \S\ref{subsec:setting}.) For simplicity, let $k$ be a perfect field of characteristic $p>2$. (When $p=2$ we obtain partial results.) Let $W$ denote its ring of Witt vectors. Let $K$ be a finite totally ramified extension of $\Frac W$, and we let $\fo_K$ denote its valuation ring. \emph{We fix a uniformiser $\varpi\in\fo_K$.} Let $R$ be a \emph{formally smooth} adic $\fo_K$-algebra such that $\wh\Omega_{R/\fo_K}$ is finitely generated over $R$, and $R_{\red}$ is finitely generated over $k$. We further assume that there exists a $W$-subalgebra $R_0\subseteq R$ such that $\fo_K\otimes_WR_0=R$. (The last assumption is automatic by Lemma~\ref{lem:sec}.) Note that all our main results in the ``integral theory'' depend upon the choice of $\varpi$ and $R_0$.

For a $p$-divisible group $G$ over $R$, let $\DD^*(G)$ denote the contravariant Dieudonn\'e crystal associated to $G$. We construct (in \S\ref{subsec:settingBr}) a divided power thickening  $S\thra R$ which generalises $S\thra \fo_K$.

\begin{nothm}[Theorem~\ref{thm:BreuilClassif}]\label{thm:intro:Breuil}
Let   $R$ be as above, and we use the notation from \S\ref{subsec:settingBr}. Assume that $p>2$. Then there is an exact anti-equivalence of categories, defined by $G\rightsquigarrow \DD^*(G)(S)$, from the category of $p$-divisible groups over $R$ to the category $\SMFqw$ of Breuil modules (\emph{cf.,} Definition~\ref{def:BrMod}). 
\end{nothm}

Note that when $R=\fo_K$, one can replace Breuil $S$-modules with certain $W[[u]]$-modules with simpler structure. In the relative setting, we have a subring $\Sig=R_0[[u]]\subseteq S$. We also define Kisin $\Sig$-modules (\emph{cf.,} Definition~\ref{def:KisMod}) in the manner analogous to the  case when $R=\fo_K$, except that we need to consider a connection in the relative setting in general.

\begin{noprop}[Proposition~\ref{prop:CL}]\label{prop:intro:CL}
Assume that $p>2$. 
The scalar extension $\gM\rightsquigarrow S\otimes_{\vphi,\Sig}\gM$ induces an exact equivalence of categories from the category of Kisin $\Sig$-modules to  the category of Breuil $S$-modules. 
\end{noprop}
Proposition~\ref{prop:intro:CL} was proved in \cite[\S2.2]{Caruso-Liu:qst} when $R=\fo_K$ and $p>2$. We modify this proof in the relative case (which also proves some weaker statement when $p=2$).
\begin{nocor}[Corollary~\ref{cor:RelKisin}]\label{cor:intro:BK}
Assume that $p>2$. 
There exists an exact anti-equivalence of categories from the category of $p$-divisible groups over $R$ to the category of Kisin $\Sig$-modules. 
\end{nocor}

We expect Corollary~\ref{cor:intro:BK}  to hold even when $p=2$, but at the moment we only obtain some partial results such as  Corollaries~\ref{cor:gMFF} and \ref{cor:BreuilClassif}.

Note that the equivalences of categories in Theorem~\ref{thm:intro:Breuil} and Corollary~\ref{cor:intro:BK} depend upon the choice of $\varpi$ and $R_0$, so they do not behave well under arbitrary base change, but they behave well under \'etale base change. (\emph{Cf.} Lemma~\ref{lem:etloc}.) 

Finally, the semi-linear algebra objects appearing in Theorem~\ref{thm:intro:Breuil} and Corollary~\ref{cor:intro:BK} carry connections. In some cases when the (anti-)equivalences of categories were constructed without connection, we can remove connections from the statement using Vasiu's construction of moduli of connections \cite[\S3]{Vasiu:ys}. (See Corollary~\ref{cor:BreuilClassif} for more details).

Let us now discuss the $p$-adic Tate module representation and integral $p$-adic  comparison isomorphism. Assume that $R$ is a domain, and let $\ol R$ denote the normalisation of $R$ in the affine ring of a pro-universal covering of $\Spec R[\ivtd p]$. Choose $\varpi^{(n)}\in\ol R$ for $n\geqs 0$ so that $\varpi^{(0)} = \varpi$ and $(\varpi^{(n+1)})^p=\varpi^{(n)}$.  
We construct a ``perfectoid subalgebra'' $\wt R_\infty\subset\wh{\ol R}$ (in the sense of Scholze \cite{Scholze:Perfectoid}), which is constructed, roughly speaking, by adjoining $\varpi\com i$ and compatible $p$-power roots of lifts of local $p$-basis of $R/(\varpi)$. (\emph{Cf.} \S\ref{subsec:perfectoid}.)
Set $\GRR:=\Gal(\ol R[\ivtd p]/R[\ivtd p])$ and $\GRtinfty:=\Gal(\wh{\ol R}[\ivtd p]/\wt R_\infty[\ivtd p])$.

We localise $R$ if necessary to ensure that a relative period rings have nice properties\footnote{Indeed, we will work with slightly more general rings than Brinon \cite{Brinon:imperfect, Brinon:CrisDR}. See \S\ref{subsec:BrinonSetting} for more details.}, and embed $S$ into $\Acris(R)$ using our choice of $\varpi^{(n)}$ in such a way that respects all the relevant structures (\emph{cf.,} \S\ref{subsec:Rinfty}). Using this we can define a $\Zp$-lattice $\GRR$-representation $\Tcris^*(\M)$ for any Breuil $S$-module $\M$ (\emph{cf.,} (\ref{eqn:Tcris}), \S\ref{subsec:GalViaN}). Also, by  the  theory of (relative) \'etale $\vphi$-modules, one obtains a natural $\Zp$-lattice $\GRinfty$-representation $\cT^*(\gM)$ for any Kisin $\Sig$-module $\gM$ (\emph{cf.,} Lemma~\ref{lem:FontaineB183}). 

The following theorem is obtained from the same argument as in \cite[\S6, Thm~7]{Faltings:IntegralCrysCohoVeryRamBase} and the proof of \cite[Theorem~2.2.7]{kisin:fcrys}, respectively.
\begin{nothm}[Theorem~\ref{thm:Faltings}, Corollary~\ref{cor:BrinonTrihan}, Proposition~\ref{prop:GRinftyRes}]\label{thm:intro:Faltings}
Suppose that $R$  satisfies the ``refined almost \'etaleness'' assumption~(\S\ref{cond:RAF}). Let $G$ be a $p$-divisible group over $R$, and  $\M:=\DD^*(G)(S)$. Let  $\gM$ be the  Kisin $\Sig$-module corresponding to $G$.
\begin{enumerate}
\item \label{thm:intro:Faltings:ratl}
The $\GRR$-representation $V_p(G)$ is crystalline in the sense of Brinon, and we have  $D_{\cris}^*(V_p(G))\cong D^*(G)$ as filtered $(\vphi,\nabla)$-modules (\emph{cf.,} Example~\ref{exa:BT}).
Furthermore, there is a natural injective $\GRR$-map $T_p(G)\ra \Tcris^*(\M)$, which is an isomorphism if $p>2$.
\item \label{thm:intro:Faltings:Kisin} For any $p$, there is a natural injective $\GRtinfty$-map $\cT^*(\gM)\ra \Tcris^*(\M)$ whose image is $T_p(G)$.
\end{enumerate}
\end{nothm}
Note that when $G=A[p^\infty]$ for an abelian scheme $A$ over $\Spec R$, Theorem~\ref{thm:intro:Faltings}(\ref{thm:intro:Faltings:ratl}) gives a comparison morphism between the first \'etale homology and the first crystalline homology with constant $\Qp$- and $\Zp$- coefficients.

Very recently, Peter~Scholze proved the comparison isomorphism between relative $p$-adic \'etale cohomology and relative de~Rham cohomology for proper smooth morphisms $f:X\ra Y$ of varieties over a $p$-adic field with perfect residue field \cite[Theorem~1.9]{Scholze:CdR}. Presumably, it should be possible, in the near future, to generalise the proof to handle (log-)crystalline cohomology in place of de Rham cohomology when $f$ has a (log-)smooth integral model, which can be thought of as a ``generalisation'' of (at least the $\Qp$-coefficient statement of) Theorem~\ref{thm:intro:Faltings}(\ref{thm:intro:Faltings:ratl}).

Finally, we state our main result for commutative  finite locally free group schemes\footnote{In this paper, we only consider \emph{commutative} finite locally free group schemes, so we will often suppress the adjective ``commutative''.}: 
\begin{nothm}[Theorem~\ref{thm:RelKisinFF}, Proposition~\ref{prop:GalRepFF}]\label{thm:intro:FF}
Assume that $p>2$. There exists a natural anti-equivalence of categories $H\leftrightsquigarrow\gM^*(H)$ between the category of $p$-power order finite flat group schemes $H$ over $R$ such that $H[p^i]$ is flat for each $i$ and the category $\SMFIqw$ of certain ``torsion Kisin $\Sig$-modules'' (Definition~\ref{def:torKisMod}). 
Furthermore, we have a natural $\GRtinfty$-equivariant isomorphism $H(\ol R)\cong \cT^*(\gM^*(H))$.
\end{nothm}

Let us review previous results when $p>2$. 
When $R=\fo_K$ with perfect residue field, all our main results  are known already. Theorem~\ref{thm:intro:Faltings}(\ref{thm:intro:Faltings:ratl}) is  proved by  Faltings \cite[\S6]{Faltings:IntegralCrysCohoVeryRamBase}, while the first assertion was already proved by Fontaine  \cite[Th\'eor\`eme~6.2]{Fontaine:BT}.   
Theorem~\ref{thm:intro:Breuil} was first proved by Breuil \cite{Breuil:GrPDivGrFiniModFil} and all the remaining statements were proved by Kisin \cite{kisin:fcrys}. 

 Brinon and Trihan \cite{BrinonTrihan:CrisImperf} generalised the results of Kisin \cite{kisin:fcrys} and Faltings \cite[\S6]{Faltings:IntegralCrysCohoVeryRamBase} to $p$-divisible groups over a $p$-adic discrete valuation ring with imperfect residue field admitting finite $p$-basis.

When $\fo_K=W(k)$ for some perfect  field $k$ and $R$ is of topologically finite type over $W(k)$,  Theorem~\ref{thm:intro:Breuil}, Theorem~\ref{thm:intro:Faltings}(\ref{thm:intro:Faltings:ratl}), and Theorem~\ref{thm:intro:FF} can be deduced from the result of  Faltings \cite[\S{VII}]{Faltings:Ccris}. Theorem~\ref{thm:intro:Breuil} can be deduced from an unpublished manuscript by Bloch and Kato \cite{BlockKato:Dieudonne} when $R$ is a formal power series over $W(k)$.

When  $R$ is a complete  regular local ring with perfect  residue field (but not necessarily formally smooth over $\fo_K$), Corollary~\ref{cor:intro:BK} was proved by Eike~Lau \cite{Lau:2010fk} generalising Vasiu and Zink \cite{ZinkVasiu:BreuiloverRegularLocal}. Indeed, they proved a stronger result as they classified $p$-divisible groups by $\Sig$-modules without connection. On the other hand, our result holds for not necessarily local base rings $R$ such as $p$-adic completion of a smooth $\fo_K$-algebra.

Let us first outline the proof our results when $p>2$.
The basic idea is to generalise Kisin \cite[\S{A}, \S2.2, \S2.3]{kisin:fcrys} and Faltings \cite[\S6]{Faltings:IntegralCrysCohoVeryRamBase} in a similar way to Brinon and Trihan \cite{BrinonTrihan:CrisImperf}; namely, we prove Theorem~\ref{thm:intro:Breuil} by ``lifting'' A.J. de\thinspace{}Jong's result on Dieudonn\'e crystals \cite{dejong:crysdieubyformalrigid} via Grothendieck-Messing deformation theory  \cite{messingthesis}, and apply the same argument as \cite[\S6]{Faltings:IntegralCrysCohoVeryRamBase}  to show Theorem~\ref{thm:intro:Faltings}(\ref{thm:intro:Faltings:ratl}). We then prove Proposition~\ref{prop:intro:CL} by generalising \cite[\S2.2]{Caruso-Liu:qst}. This allows us to  avoid the rigid analytic construction (\emph{cf.,} \cite[\S1]{kisin:fcrys} and \cite[\S4]{BrinonTrihan:CrisImperf}), which is hard to generalise to the relative setting. 
To prove Theorem~\ref{thm:intro:FF}, we generalise the strategy of \cite[\S2.3]{kisin:fcrys}. Due to the presence of connections, the actual argument is quite elaborate using the theory of moduli of connections \cite[\S3]{Vasiu:ys}.

We expect that by working with log Dieudonn\'e crystals we can generalise all the results in this paper to $p$-divisible groups over some semi-stable bases using an unpublished result of Bloch and Kato \cite{BlockKato:Dieudonne}, at least when $p>2$; more precisely, we allow $R$ such that $R/(\varpi)$ satisfies either (0.5.1) or (0.5.2) in \cite{BlockKato:Dieudonne}. This case will be studied in the forthcoming paper.

In \S\ref{sec:CotCplx} we prove some lifting result, which is needed for constructing $\Sig$ and $S$. In \S\ref{sec:classif} we introduce various semi-linear algebra objects and prove Theorem~\ref{thm:intro:Breuil} when $p>2$. In \S\ref{sec:Brinon} we recall (and slightly generalise) the construction and basic properties of relative $p$-adic period rings. In \S\ref{sec:Faltings} we prove integral comparison theorem (\emph{cf.,} Theorem~\ref{thm:intro:Faltings}(\ref{thm:intro:Faltings:ratl})). 

In \S\ref{sec:CL} we introduce generalisations of Kisin modules and prove Proposition~\ref{prop:intro:CL}. In \S\ref{sec:EtPhiMod}, we recall some basic results of perfectoid algebras \cite{Scholze:Perfectoid} and the theory of relative \'etale $\vphi$-modules. In \S\ref{sec:FaltingsKisin} we prove Theorem~\ref{thm:intro:Faltings}(\ref{thm:intro:Faltings:Kisin}). 
In \S\ref{sec:finfl} we state our main results on finite locally free group schemes of $p$-power order. The proof uses the theory of moduli of connection (in a slightly generalised form than the original version in \cite[\S3]{Vasiu:ys}), which will be studied in \S\ref{sec:Vasiu}.

\subsection*{Acknowledgement}
The author  appreciates   A. Johan de\thinspace{}Jong,  Mark Kisin, James Newton, Burt Totaro, Teruyoshi Yoshida for their helpful suggestions. He appreciates Brian Conrad and the anonymous referee for their careful reading of the earlier version of this paper as well as helpful suggestions,  and Adrian Vasiu for directing the author's attention to the theory of moduli of connections. He also appreciates  Eike~Lau for pointing out the mistake in \S6 in the earlier manuscriptThis work was supported by Herchel Smith Post-doctoral Research Fellowship.

\section{Settings, review, and some commutative algebra}\label{sec:CotCplx}
\subsection{General convention}
For any ring $A$, a ring endomorphism $\vphi:A\ra A$, and an $A$-module $M$, we write $\vphi^*M:=A\otimes_{\vphi,A}M$. We sometimes write $\vphi_A$ for $\vphi$ to indicate that it is an endomorphism of $A$. 

For any ideal $J\subset A$, we say that $A$ is \emph{$J$-adic} if $A$ is $J$-adically separated and complete. If $J=(x)$, then we say ``$x$-adic'' instead of ``$(x)$-adic''.

For any   $p$-adic ring $A$ we let $\wh\Omega_A(=\wh\Omega_{A/\Zp}):=\varprojlim_n\Omega_{\left(A/p^n\right)/\Zp}$ denote the module of $p$-adically continuous K\"ahler differentials. Note that if $k_0$ is a perfect subfield of $k$ then we have   $\wh\Omega_A \cong \wh\Omega_{A/W(k_0)}$. We will work with $\wh\Omega_A$ only when $A/(p)$ locally has a finite $p$-basis (\emph{cf.,} \cite[Lemmas~1.1.3, 1.3.3]{dejong:crysdieubyformalrigid}), so 
$\wh\Omega_A$ will always be finitely generated over $A$. 

For any noetherian ring $A$, we write $A_{\red}:=A/\rad(A)$, where $\rad(A)$ is the Jacobson radical. If $A$ is $J$-adic and $A/J$ is finitely generated over a field, then $A_{\red}$ is the maximal reduced quotient of $A/J$, which justifies the notation.

For any ring $A$ and a divided power ideal $I\subseteq A$, we denote by $s^{[n]}$ the $n$th divided power of $s\in I$.

\subsection{Assumptions on base rings}\label{subsec:setting}
Throughout the paper, we let $R$ be a $p$-adic flat $\Zp$-algebra with varying  technical assumptions depending on the situation. Here, we will list and motivate the assumptions on $R$ for readers' reference.

\subsubsection{$p$-basis assumption}\label{cond:Breuil}
We assume that $R \cong R_0[u]/E(u)$, where $R_0$ be a $p$-adic flat $\Zp$-algebra such that $R_0/(p)$ locally admits a finite $p$-basis, and 
\[E(u) = p + \sum_{i=1}^e a_i u^i \]
for some integer $e>0$, with $a_i\in R_0$ and $a_e\in R_0\starr$. Note that such $R$ is a finite $R_0$-algebra which is  free of rank~$e$ as an $R_0$-module. 
We will let $\varpi\in R$ denote the image of $u\in R_0[u]$.

If there is a Cohen subring $W\subset R_0$ such that $E(u)\in W[u]$, then we set $\fo_K:=W[u]/E(u)$, which is a finite totally ramified extension of $W$. In this case, $R:=R_0[u]/E(u) = \fo_K\otimes_W R_0$. By Lemma~\ref{lem:sec}  any $p$-adic flat $\fo_K$-algebra $R$ such that $R/(\varpi)$ locally admits a finite $p$-basis can be written in this form. Although it is the prototypical case, we do not restrict to the case when $E(u)\in W[u]$ for some Cohen subring $W$ for the following reason, except in the theory of \'etale $\vphi$ modules in \S\ref{sec:EtPhiMod} and \S\ref{sec:FaltingsKisin}. (See Remark~\ref{rmk:PerfectoidCond} for more discussions.) 
 
If $R$ satisfies  the $p$-basis assumption~(\S\ref{cond:Breuil}), then we can define relative Breuil modules in \S\ref{subsec:settingBr}. In order to relate relative Breuil modules to $p$-divisible groups, one needs more assumptions on $R$. 

\subsubsection{Formally finite-type assumption}\label{cond:dJ}
In addition to the $p$-basis assumption~(\S\ref{cond:Breuil}), we assume that $R$ is $J_R$-adically separated and complete for some finitely generated ideal $J_R$ containing $\varpi$ (defined in \S\ref{cond:Breuil}), and $R/J_R$ is finitely generated over some field $k$. Then $R_0/(p)=R/(\varpi)$ is $J_R/(\varpi)$-adically separated and complete. Note that  $R_0$ is $J_{R_0}$-adically separated and complete, where $J_{R_0}\subset R_0$ denote the kernel of $R_0\thra R/J_R$.

Since $R_0/(p)=R/(\varpi)$ surjects onto $R/J_R$,  it follows that $k$ as above admits a finite $p$-basis by  the $p$-basis assumption~(\S\ref{cond:Breuil}). By cotangent complex consideration it also follows that $R/(\varpi)$ is formally smooth $\Fp$-algebra (\emph{cf.} \cite[Lemma~1.1.2]{dejong:crysdieubyformalrigid} and  \cite[Ch.III, Corollaire~2.1.3.3]{Illusie:CplxCotI}). Therefore, $R/(\varpi)$ satisfies the assumption \cite[(1.3.1.1)]{dejong:crysdieubyformalrigid}, and de\thinspace{}Jong's theorem asserts that the crystalline Dieudonn\'e functor over  $\Spec R/(\varpi)$ is an equivalence of categories; \emph{cf.} \cite[Main~Theorem~1]{dejong:crysdieubyformalrigid}. 
%
 
Clearly, $R_0$ is formally smooth over $\Zp$ with respect to the $J_{R_0}$-adic topology, and the same holds if $\Zp$ is replaced by any Cohen subring $W\subset R_0$. If $E(u)\in W[u]$ and $\fo_K = W[u]/E(u)$, then $R$ is formally smooth over $\fo_K$. Conversely, any formally smooth formally finite-type $\fo_K$-algebra satisfies the formally finite-type assumption~(\S\ref{cond:dJ}) by Lemma~\ref{lem:sec}.

\subsubsection{``Refined almost \'etaleness'' assumption}\label{cond:RAF}
Here, we assume that $R$ is a domain which satisfies the formally finite-type assumption~(\S\ref{cond:dJ}), such that $R[\ivtd p]$ is finite \'etale over $R_0[\ivtd p]$ and we have $\wh\Omega_{R_0} = \bigoplus_{i=1}^d R_0 d T_i$  for some finitely many units $T_i\in R_0\starr$. Here, $\wh\Omega_{R_0}$ is the module of $p$-adically continuous K\"ahler differentials. 
Under this condition, we obtain a  version of refined almost \'etaleness (Theorem~\ref{thm:RAF}), which is slightly more general than \cite[\S5]{Andreatta:GenNormRings}. Note that refined almost \'etaleness is used to show that relative period rings have nice properties (namely, Proposition~\ref{prop:PeriodRings}).

For example, let $W$ be some Cohen ring with residue field $k$ admitting a finite $p$-basis, and set $R_0=W\langle T_1^{\pm1},\cdots, T_{d'}^{\pm1}\rangle$ or $R_0=W[[X_1,\cdots,X_{d'}]]$. Then the normal extension $R:=R_0[[u]/E(u)$ satisfies the ``refined almost \'etaleness'' assumption~(\S\ref{cond:RAF}), where $E(u)$ is as in \S\ref{cond:Breuil} with the additional assumption that $\frac{\partial}{\partial u} E(\varpi)\in R\starr$. (In the first case, $T_i$'s together with a lift of a $p$-basis of $k$ do the job, and in the second case we use $T_i=1+X_i$ instead.) 

\subsubsection{Normality assumption}\label{cond:BM}
In addition to the $p$-basis assumption~(\S\ref{cond:Breuil}), we assume that $R/(\varpi)$ is a finite product of normal domains. In this case, the crystalline Dieudonn\'e functor over  $\Spec R/(\varpi)$ is fully faithful by \cite[Th\'eor\`eme~4.1.1]{Berthelot-Messing:DieudonneIII}.

Let us give some examples. Assume that  $R$ that satisfies the formally finite-type assumption~(\S\ref{cond:dJ}). Then $R$ satisfies the normality assumption~(\S\ref{cond:BM}), which follows from Lemmas~1.1.3,~1.3.3(d) in \cite{dejong:crysdieubyformalrigid}. Moduli of  connections (\emph{cf.} Definition~\ref{def:ModConn}) provides examples of $R$ satisfying the normality assumption~(\S\ref{cond:BM}) but not necessarily the formally finite-type assumption~(\S\ref{cond:dJ}). See Lemma~\ref{lem:Artin-Schreier} for the precise statement. This example plays an important role in the proof of the classification of finite locally free group schemes (Theorem~\ref{thm:RelKisinFF}).
%
%
%
%
%

\subsubsection{lci assumption}\label{cond:dJM}
The definition of relative Breuil modules can be modified for base rings $R$ satisfying a certain lci assumption. (\emph{Cf.} Remarks~\ref{rmk:Brlci} and \ref{rmk:Slci}.) Under additional excellence assumption, we can obtain some full faithfulness result as well. (\emph{Cf.} Remarks~\ref{rmk:lciBreuilClassif}, and \ref{rmk:lciRelKisin}.) In order not to over-complicate our notation, we mainly focus on the setting with the $p$-basis assumption~(\S\ref{cond:Breuil}), and record the results in the lci case in remarks marked with ``(lci case)''.

\subsection{Formally smooth subalgebra}\label{subsec:CommAlg}
In this section, we prove some deformation result, which produces many classes of examples of $R$ which satisfy  the $p$-basis assumption~(\S\ref{cond:Breuil}).

Let $k$ be a field of characteristic $p>0$ which admits a finite $p$-basis (i.e., $[k:k^p]<\infty$), and we choose a Cohen ring $W$ with residue field $k$; i.e., a $p$-adic discrete valuation ring with fixed isomorphism $W/(p) \cong k$. (Note that $W$ is not a ring of Witt vectors unless $k$ is perfect.) Let us fix a ring endomorphism $\vphi_W:W\ra W$ which lifts the $p$th power map $\vphi_k:k \ra k$. 

\begin{lemsub}[\emph{Cf.} {\cite[\S1]{Berthelot-Messing:DieudonneIII}, \cite[Lemma~1.3.3]{dejong:crysdieubyformalrigid}}]\label{lem:lifting}
Let $\bar A$ be a $k$-algebra which locally admits a finite $p$-basis. 
Then there exists a $p$-adic flat $W$-algebra  $A$ which lifts $\bar A$, equipped with a lift of Frobenius $\vphi:A\ra A$ over $\vphi_W:W\ra W$. Such $A$ is formally smooth over $W$ under the $p$-adic topology. If furthermore $\bar A$ is formally of finite type over  $k$ with ideal of definition $\bar J$ (for example, if $\bar A = R/(\varpi)$ where $R$ satisfies the formally finite-type assumption~(\S\ref{cond:dJ})), then any such $A$ is $J$-adically complete where $J$ is the preimage of $\bar J$, and under this topology $A$ is formally smooth.
\end{lemsub}
\begin{proof}
Let us first assume that $\bar A$ locally admits a finite $p$-basis.  By \cite[Lemma~1.1.2]{dejong:crysdieubyformalrigid}, the augmentation map $L_{\bar A/k} \ra \Omega_{\bar A/k}$ is a quasi-isomorphism where $L_{\bar A/k}$ is a cotangent complex.  (Indeed, there is a natural distinguished triangle $\bar A\stackrel{\LL}{\otimes}_{k} L_{k/\Fp} \ra L_{\bar A/\Fp} \ra L_{\bar A/k} \ra [+1]$, and now we apply \cite[Lemma~1.1.2]{dejong:crysdieubyformalrigid} to the first two terms.) Note also that  $ \Omega_{\bar A/k}$ is finite projective over $\bar A$ by the existence of local finite $p$-basis. 

Now if $A_{W_i}$ is a lift of $\bar A$ over $W_i:=W/(p^{i+1})$, then the obstruction class for lifting $A_{W_i}$ over  $W_{i+1}$ lies in $\Ext^2_{\bar A}(L_{\bar A/k}, \bar A) = \Ext^2_{\bar A}(\Omega_{\bar A/k}, \bar A) $, which is zero. (\emph{Cf.} \cite[Ch.III, Corollaire~2.1.3.3]{Illusie:CplxCotI}.) This shows the existence of a $p$-adic flat $W$-lift $A$ of $\bar A$.

Using the same notation as above, let $\tilde\vphi_i:\vphi_W^*A_{W_i} \ra  A_{W_i}$ be a lift of the relative Frobenius map $\vphi:\vphi_k^* \bar A\ra \bar A$ over $k$. The obstruction class for lifting $\tilde\vphi_i$ to $\tilde\vphi_{i+1}:  \vphi_W^*A_{W_{i+1}} \ra A_{W_{i+1}}$ lies in $\Ext^1_{\bar A}(L_{\bar A/k}, \bar A) = \Ext^1_{\bar A}(\Omega_{\bar A/k}, \bar A) $, which is zero. (\emph{Cf.} \cite[Ch.III, Proposition~2.2.2]{Illusie:CplxCotI}.) This shows the existence of  a $W$-lift $\vphi_W^*A \ra A$ of the relative Frobenius map. Now, we obtain $\vphi: A\ra A$ over $\vphi_W$ by precomposing this lift with the natural inclusion $A\ra \vphi_W^*A$ (defined by $a\mapsto 1\otimes a$ for any $a\in A$).

To show that $A$ is formally smooth over $W$ (for the $p$-adic topology), it suffices to show that $A_{W_i}:=A/(p^{i+1})$ is formally smooth over $W_i:=W/(p^{i+1})$ for each $i$.  Indeed, by \cite[Remark~1.3.4(b)]{dejong:crysdieubyformalrigid}, $ \Omega_{A_{W_i}/W_i}$ is  finite projective over $A_{W_i}$. Now one can show inductively that the augmentation map $L_{A_{W_i}/W_i}\ra \Omega_{A_{W_i}/W_i}$ is a quasi-isomorphism, which implies by \cite[Ch.~III, Proposition~3.1.2]{Illusie:CplxCotI} that $A_{W_i}$ is formally smooth over $W_i$. 

Now, assume that $\bar A$ is formally of finite type over some field $k$ with ideal of definition $\bar J$. To show that $A$ is $J$-adically complete (where $J
$ is as in the statement), we let $A'$ be the $J$-adic completion of $A$. The natural map $A\ra A'$ is clearly injective, and the surjectivity can be checked mod~$p$ as both $A$ and $A'$ are $p$-adic. But $A/(p)$ is already $\bar J$-adically complete. The formal smoothness under the $J$-adic topology follows from the formal smoothness under the $p$-adic topology.
\end{proof}

\begin{rmksub}\label{rmk:NonuniqueLift} 
The choice of $A$ is unique up to non-unique isomorphism, while the choice of $\vphi$ is far from unique in general. Therefore, given a continuous map $A\ra A'$  of formally smooth $p$-adic $\Zp$-algebras, it may not be possible in general to choose lifts of Frobenii for $A$ and $A'$ so that $f$ commutes with them. We will see later (Lemma~\ref{lem:etloc}), however, that this is possible (in a unique way) if $f$ is \'etale.
\end{rmksub}

\begin{rmksub}\label{rmk:Cohen}
Let $\bar A$ be an $\Fp$-algebra locally admitting a finite $p$-basis, and $A$ its lift over $\Zp$. Assume that $\bar A$ contains a field $k$, and we choose a Cohen ring $W$ with residue field $k$. Then since $A$ is formally smooth over $\Zp$ by Lemma~\ref{lem:lifting}, there is a lift $W\hra A$ of $k\hra \bar A$.
If we choose a lift of Frobenius $\vphi_W:W\ra W$, then we may choose $\vphi_A:A\ra A$ to be compatible with $\vphi_W$.
\end{rmksub}

For the following lemma, set $\fo_K=W[u]/E(u)$ where $E(u)\in W[u]$ is an Eisenstein polynomial (\emph{cf.}  \S\ref{cond:Breuil}), and let $\varpi\in\fo_K$ be the image of $u$. Let $R$ be a $p$-adic flat $\fo_K$-algebra such that $R/(\varpi)$ locally admits a finite $p$-basis. Let $R_0$ denote a flat $W$-lift of $R/\varpi$ (obtained from applying Lemma~\ref{lem:lifting} to $\bar A = R/(\varpi)$). 
\begin{lemsub}\label{lem:sec}
In the above setting, there exists a $W$-algebra map $\iota:R_0 \ra R$ which lifts the  $k$-isomorphism $R_0/(p)\riso R/(\varpi)$. Furthermore, the induced $\fo_K$-algebra map $\iota_{\fo_K}:R_0\otimes_{W}\fo_K\ra R$ is an isomorphism of topological $\fo_K$-algebras.

If furthermore $R$ is formally of finite type over $\fo_K$ with radical $J_R$ (in particular, $R$ satisfies assumption \S\ref{cond:dJ}), then we may additionally require $\iota$ to be continuous for the $J_{R_0}$-adic topology, where $J_{R_0}\subset R_0$ is the preimage of $J_R/(\varpi)$.
\end{lemsub} 
When $R/(\varpi)$ is a finitely generated algebra over $k$, this lemma is a consequence of \cite[1, Exp.~III, Corollaire~6.8]{SGA}. Note that the map $R_0\ra R$ in the statement is far from unique in general.
\begin{proof}
Since $R$ is $\fo_K$-flat, the existence of the isomorphism $\iota_{\fo_K}$ follows from the uniqueness of $\fo_K$-lifts of $R/(\varpi)$ (\emph{cf.,} the proof of Lemma~\ref{lem:lifting}).

Now assume that $R$ satisfies the formally finite-type assumption~(\S\ref{cond:dJ}). It remains to construct a $J_{R_0}$-adically continuous map $\iota$.  By formal smoothness of $R_0$, the natural isomorphism $R_0/J_{R_0}\riso R/J_R$ has a continuous lift to $R_0 \ra R/J_R^{i+1}$ for any $i>0$. By taking limit we get a continuous morphism $\iota':R_0\ra R$. Let $\bar\iota':R_0/(p)\ra R/(\varpi)$ denote the map induced by $\iota'$, and let $\bar\iota:R_0/(p)\ra R/(\varpi)$ be the isomorphism given by the construction of $R_0$ in Lemma~\ref{lem:lifting}. 
By construction,  we have $\bar\iota\iv\circ\bar\iota' \equiv \id_{R_0/(p)}\bmod{J_{R_0}/(p)}$. We obtain $\iota:R_0\ra R$ by modifying $\iota '$ by an automorphism of $R_0$ which lifts $(\bar\iota\iv\circ\bar\iota')\iv$.
\end{proof}

Assume that  $R$ satisfies the formally finite-type assumption~(\S\ref{cond:dJ}), and let $(R_0,J_{R_0})$ and $E(u)\in R_0[u]$  be as in \S\ref{cond:dJ}. Choose a lift of Frobenius $\vphi:R_0\ra R_0$ as in Lemma~\ref{lem:lifting}. Let $I\subset J_{R_0}$ be a closed ideal containing $p$. (Often, we will take either $I=J_{R_0}$ or $I=(p)$.) Set $\ol R:= R_0/I$, and let $\ol R'$ be a finitely generated \'etale  $\ol R$-algebra.
\begin{lemsub}\label{lem:etloc}
With the above setting, there exists a  $I$-adic formally \'etale $R_0$-algebra $R_0'$ such that $R_0'/IR_0' \cong \ol R'$ as $\ol R$-algebras. Such an $R_0'$ is unique up to unique isomorphism, and $\vphi_{R_0}$ uniquely extends to a lift of Frobenius $\vphi_{R_0'}: R_0' \ra R_0'$.

Let $R'$ be an $I$-adic formally \'etale $R$-algebra such that $R'/IR'$ is finitely generated over $R/IR$. Let $R_0'$ be a lift of $\ol R':=R'/(\varpi,I)$. Then there exists a unique $R_0$-homomorphism $R_0'\hra R'$ which induces $R_0'[u]/E(u)\cong R'$ as $R$-algebras. 
\end{lemsub}
It is clear that  $R_0'$ is formally \'etale over $R_0$, so $R_0'$ is necessarily formally smooth over $W$. Since $\vphi_{R_0'}$ is uniquely determined by $\vphi_{R_0}$, there will be no confusion in denoting both by $\vphi$ in this case. 

\begin{proof}
The existence and uniqueness of $R_0'$ follows from standard deformation theory (\emph{cf.,} Proposition~6.1 and Th\'eor\`eme~6.3 in \cite[1]{SGA}). 
By Proposition~5.8 in \cite[1]{SGA}, there exists a unique map $\vphi_{R_0}^*R_0' \ra R_0'$ which lifts the relative Frobenius $\vphi_{\ol R}^*\ol R' \ra  \ol R'$.
We set $\vphi_{R_0'}:R_0'\ra R_0'$ to be the composition  $R_0'\ra \vphi_{R_0}^*R_0' \ra R_0'$, where the first map is defined by $a\mapsto 1\otimes a$ for any $a\in R_0'$, and the second map is the unique lift of the relative Frobenius. 

The last assertion follows from the uniqueness up to unique isomorphism of $R$-lift of $\ol R$, which provides a unique $R$-isomorphism  $R_0'[u]/E(u)\cong R'$ that lifts the identity map on $\ol R$. 
\end{proof}

\section{Classification of $p$-divisible groups by Dieudonn\'e crystals}\label{sec:classif}
In this section, we introduce a ``relative'' version of strongly divisible modules, and prove a generalisation of the Breuil classification of  $p$-divisible group over a base ring $R$ which satisfies the formally finite-type assumption~(\S\ref{cond:dJ}) (and some partial result for more general kind of $R$). We follow Kisin's strategy \cite[\S{A}]{kisin:fcrys} using the main results of \cite{dejong:crysdieubyformalrigid, Berthelot-Messing:DieudonneIII} and Grothendieck-Messing deformation theory \cite{messingthesis}. Along the way, we give a very brief review of crystalline Dieudonn\'e theory.

\subsection{Review: crystalline Dieudonn\'e theory}
We follow the notation and convention from  \cite{Berthelot-Breen-Messing:DieudonneII, dejong:crysdieubyformalrigid}. Let $\XX$ be a formal scheme over $\Spf \Zp$, and let $\ol\XX:=\XX\times_{\Spf \Zp}\Spec \Fp$ denote  the closed formal subscheme of $\XX$ defined by the ideal $(p)$. For example, if $\XX = \Spf (R,(p))$ then $\ol\XX = \Spec R/(p)$. Let $(\XX/\Zp)_{\CRIS}$ and $(\ol\XX/\Zp)_{\CRIS}$ denote the big fppf-crystalline topoi. 

Let $G$ be a $p$-divisible group over $\XX$ and $\ol G:=G\times_\XX \ol\XX$. One contravariantly associates to $G$ a crystal $\DD^*(G)$ of finite locally free $\cO_{\XX/\Zp}$-module in such a way that commutes with base change. See \cite{messingthesis}, \cite{Mazur-Messing}, or \cite{Berthelot-Breen-Messing:DieudonneII} for the construction (\emph{cf.,} \cite[Definition~2.4.2(b)]{dejong:crysdieubyformalrigid}). In particular, since $p=0$ in $\ol\XX$  one obtains $F:\vphi^*\DD^*(\ol G)\ra\DD^*(\ol G)$ and $V:\DD^*(\ol G)\ra\vphi^*\DD^*(\ol G)$ from the relative Frobenius and Verschiebung of $\ol G$, and we have $F\circ V = p$ and $V\circ F = p$.

Let $i_{\CRIS}:=(i_{\CRIS,*},i_{\CRIS}^*):(\ol\XX/\Zp)_{\CRIS}\ra(\XX/\Zp)_{\CRIS}$ be the morphism of topoi induced from the closed immersion $\ol\XX\hra \XX$. Then $i_{\CRIS,*}$ and $i_{\CRIS}^*$ induce quasi-inverse exact equivalences of categories between the categories of crystals of finitely presented (respectively, finite locally free) $\cO_{\ol\XX/\Zp}$-modules and $\cO_{\XX/\Zp}$-modules. (This follows from  \cite[Lemma~2.1.4]{dejong:crysdieubyformalrigid}, noting that there is a natural isomorphism $i_{\CRIS,*}\cO_{\ol\XX/\Zp} \cong \cO_{\XX/\Zp}$ as in \cite[\S5.17.3]{Berthelot-Ogus}.) Since the formation of $\DD^*$ commutes with base change, we obtain a natural isomorphism $\DD^*(G)\cong i_{\CRIS,*}(\DD^*(\ol G))$. 

Let $\DD^*(G)_{\XX}$ denote the locally free $\cO_\XX$-module obtained from the push-forward of $\DD^*(G)$ to the Zariski topos. The construction of $\DD^*(G)$ also provides the following functorial exact sequence of vector bundles which commutes with base change:
\begin{equation}\label{eqn:HodgeFiltr}
0\ra \llie(G)^*\ra \DD^*(G)_{\XX} \ra  \llie(G^\vee) \ra 0, 
\end{equation}
where $G^\vee$ is the dual $p$-divisible group.  This exact sequence defines the \emph{Hodge filtration} $\llie(G)^*\subseteq\DD^*(G)_{\XX}$ for $G$. (There are two possible ways to define the Hodge filtration -- one via universal vector extension in \cite{messingthesis}, and the other via \cite[Corollaire~3.3.5]{Berthelot-Breen-Messing:DieudonneII} -- and they coincide by \cite[Th\'eor\`eme~3.1.7]{Berthelot-Messing:DieudonneIII}.)

\begin{defnsub}\label{def:DieudonneCrys}
A \emph{Dieudonn\'e crystal} over $\XX$ is a quadruple $(\E, F, V,\Fil^1\E_\XX)$, where
\begin{enumerate}
\item\label{def:DieudonneCrys:locfree} $\E$ is a crystal of finite locally free $\cO_{\XX/\Zp}$-module.
\item\label{def:DieudonneCrys:FV} Let $\ol\E:=\E|_{\ol\XX} (=i_{\CRIS}^*\E)$, and let $\vphi:\ol\XX\ra\ol\XX$ denote the absolute Frobenius morphism. Then $F:\vphi^*\ol\E\ra\ol\E$ and $V:\ol\E\ra\vphi^*\ol\E$ are morphisms which satisfy  $F\circ V = p$ and $V\circ F = p$.
\item\label{def:DieudonneCrys:HodgeFiltr} Let $\E_\XX$  denote the  locally free $\cO_\XX$-module obtained from the push-forward of $\E$ to the Zariski topos. 
Then $\Fil^1\E_\XX \subseteq \E_\XX$ is a direct factor as an $\cO_\XX$-submodule such that $\vphi^*(\Fil^1\E_{\XX}|_{\ol\XX}) = \ker [F:\vphi^*(\E_\XX|_{\ol\XX}) \ra \E_\XX|_{\ol\XX}]$. We call $\Fil^1\E_\XX \subset \E_\XX$ the \emph{Hodge filtration}.
\end{enumerate}
\end{defnsub}
If $\XX=\ol\XX$ (i.e., $p\cO_{\XX}=0$) then (\ref{def:DieudonneCrys:HodgeFiltr}) in Definition~\ref{def:DieudonneCrys} is automatic from the other conditions (\emph{cf.,} \cite[Proposition~2.5.2]{dejong:crysdieubyformalrigid}). In particular, our definition is compatible with \cite[Definitions~2.3.4, 2.4.2(b)]{dejong:crysdieubyformalrigid} in this case. 

\begin{lemsub}\label{lem:DC}
 For any $p$-divisible group $G$ over $\XX$, the  crystal $\DD^*(G)$ associated to $G$ has a natural structure of Dieudonn\'e crystal over $\XX$.
\end{lemsub}
\begin{proof} It suffices to show that $\Fil^1\DD^*(\ol G)_{\ol\XX}$ (the kernel of $F$ as in Definition~\ref{def:DieudonneCrys}(\ref{def:DieudonneCrys:HodgeFiltr})) coincides with the Hodge filtration for $\ol G$, which follows from  \cite[Proposition~4.3.10]{Berthelot-Breen-Messing:DieudonneII}. 
\end{proof}

\begin{rmksub}
There is an obvious notion of short exact sequences of Dieudonn\'e crystals. It follows from  \cite[Proposition~4.3.1]{Berthelot-Breen-Messing:DieudonneII} that the crystalline Dieudonn\'e functor $G\mapsto \DD^*(G)$ is exact over a base where $p$ is locally topologically nilpotent.

Also, one can describe the effect of duality on $\DD^*(G)$ \cite[\S5.3]{Berthelot-Breen-Messing:DieudonneII}; namely, $\DD^*(G^\vee)$ is the $\cO_{\XX/\Zp}$-linear dual of $\DD^*(G)$, $F$ (respectively, $V$) on $\DD^*(G^\vee)$ is induced from $V$ (respectively, $F$) on $\DD^*(G)$, and $\Fil^1\DD^*(G^\vee)_\XX$ is the annihilator of $\Fil^1\DD^*(G)_\XX$ under the natural duality pairing.
\end{rmksub}

\begin{defnsub}\label{def:FormalUnip}
We say that a Dieudonn\'e crystal $\E$ is \emph{$F$-nilpotent} if $F^n:\vphi^{n*}\E\ra\E$ factors through $p\E$. We similarly define $V$-nilpotence.
\end{defnsub}
Recall that a $p$-divisible group $G$ over $R$ is called \emph{formal} if $G[p^n]$ are infinitesimal thickening of $\Spec R$ as a scheme. If $pR=0$, then $G$ is formal if and only if $G = \varinjlim G[F^n]$ where $G[F^n]$ is the kernel of the $n$th iterated relative Frobenius morphism $F^n:G\ra \vphi^{n*}G$. A $p$-divisible group $G$ is called \emph{unipotent} if its dual is formal. Note that a $p$-divisible group is formal or unipotent if and only if all of its geometric fibres are formal or unipotent, respectively. The following lemma is now straightforward by classical Dieudonn\'e theory:
\begin{lemsub}\label{lem:FormalUnip}
A $p$-divisible group $G$ over $\XX$ is formal (respectively, unipotent) if and only if $\DD^*(G)$ is $F$-nilpotent (respectively, $V$-nilpotent).
\end{lemsub}

\subsection{Filtered Frobenius modules}\label{subsec:connection}
One can understand a crystal (over an affine formal scheme) as a suitable module with connection, which allows us to describe Dieudonn\'e crystals by some  concrete objects. 

\begin{defnsub}[\emph{Cf.} {\cite[\S2.1]{Lau:2010fk}}]\label{def:Frame}
A  \emph{frame} is a tuple $(\wh D, \Fil^1\wh D, R, \vphi, \vphi_1)$ where $R=\wh D/\Fil^1\wh D$, $\vphi:\wh D\ra\wh D$ is a lift of Frobenius,  $\vphi_1 = \vphi/p:\Fil^1\wh D\ra \wh D$ is such that the ideal generated by $\vphi_1(\Fil^1\wh D)$ is the unit ideal, and $p\wh D + \Fil^1\wh D$ is contained in the Jacobson radical of $\wh D$. In this case, we will focus on the case when $\wh D$ is a torsion-free $p$-adic ring and $\Fil^1\wh D$ has a ($p$-adically continuous) divided power structure.

A morphism of frames are morphism of rings that respects all the structures. (There is more general notion of morphisms that allow $\vphi_1$'s to differ by a certain unit multiple; \emph{cf.} \cite[\S2.1]{Lau:2010fk}, we will not use this more general notion.)
 \end{defnsub}
 
\begin{rmksub}\label{rmk:whD}
Let $R$ be a $p$-adic ring, and choose a $p$-adic $\Zp$-flat 
algebra $A$, such that  $A/(p)$ locally admits a finite $p$-basis and $R\cong A/I$ for some ideal $I$. We also choose a lift of Frobenius $\vphi:A\ra A$. Let $\wh D$ denote  the $p$-adically completed divided power envelope of $A$ with respect to $I$.  Let $\Fil^1 \wh D$ denote the ideal topologically generated by the divided powers of elements in $I$. Clearly, $p\wh D+\Fil^1\wh D$ is contained in the Jacobson radical, as $\wh D$ is $p$-adic and the image of $\Fil^1\wh D$ in $\wh D/(p)$ is a nil-ideal (being a divided power ideal).

We \emph{assume} that $\wh D$ is $\Zp$-flat. This is satisfied if $\wh D = A$ (for example, when $R=A/(p)$). Another example will be given in \S\ref{subsec:settingBr}.
Note that $\wh D$ may fail to be $\Zp$-flat -- such an example is given in the proof of \cite[Proposition~A.2]{MR700767}.  

It easily follows that $\vphi:A\ra A$ uniquely extends to $\vphi:\wh D\ra \wh D$, and $\vphi(\Fil^1\wh D)\subset p\wh D$. (To see that $\vphi:\wh D\ra\wh D$ is well defined, observe that $\wh D$ is also a $p$-adically completed divided power envelope for $(p,I)\subset A$, and $\vphi(p,I)\subset (p,I)$. To see $\vphi(\Fil^1\wh D)\subset p\wh D$, just note that $\vphi:\wh D/(p)\ra \wh D/(p)$ is the $p$-th power map so it annihilates $\Fil^1\wh D/(p)$.) We set $\vphi_1:=\frac{\vphi}{p}:\Fil^1\wh D\ra \wh D$, which is well-defined as $\wh D$ is $\Zp$-flat. Although  $\vphi_1(\Fil^1\wh D)$ may not generate the unit ideal in general, this is often satisfied in practice -- indeed, one can often find an element $x\in\Fil^1\wh D$ such that $\vphi_1(x)$ is a unit. If all of these are satisfied, we obtain a frame $(\wh D, \Fil^1\wh D, R, \vphi, \vphi_1)$.

For $\wh D$ constructed from $A$ as above, let $d_{\wh D}:\wh D \ra \wh D\otimes_{A}\wh\Omega_{A}$ denote the $p$-adically continuous derivation obtained by naturally extending the universal continuous derivation $d_{A}:A\ra\wh\Omega_{A}$ by setting $d_{\wh D}(s^{[n]}):= s^{[n-1]}d_{A}(s)$ for any $s\in I$ and $n$. (Here, $s^{[n]}:=\frac{s^n}{n!}$.) 
\end{rmksub}

\begin{defnsub}\label{def:DieudonneModwHF}
For a frame $\wh D$, let $\DMFq$ denote the category of tuples $(\M, \Fil^1\M, \vphi_\M, \vphi_1)$ where 
\begin{enumerate}
\item $\M$ is a finite projective $\wh D$-module; 
\item $\Fil^1\M\subset \M$ is a $\wh D$-submodule containing $(\Fil^1 \wh D)\M$, and such that $\M/\Fil^1\M$ is projective over $R = \wh D/\Fil^1\wh D$;
\item $\vphi_\M$ is a $\vphi$-linear endomorphism of $\M$ such that $(1\otimes\vphi_\M)(\vphi^*\Fil^1\M) = p\M$;
\item $\vphi_1 = \vphi_\M/p:\Fil^1\M\ra\M$;
\end{enumerate}
If $\wh D$, constructed as in Remark~\ref{rmk:whD}, is $\Zp$-flat, 
then we let  $\DMF$ denote the category of tuples $(\M, \Fil^1\M, \vphi_\M, \vphi_1,\nabla_\M)$ where $(\M, \Fil^1\M, \vphi_\M, \vphi_1)\in\DMFq$ and $\nabla_\M:\M\ra\M\otimes_{A} \wh\Omega_{A}$ is a topologically quasi-nilpotent integrable connection over $d_{\wh D}$ which commutes with $\vphi_\M$. (\emph{Cf.} \cite[Remarks~2.2.4]{dejong:crysdieubyformalrigid}.)
\end{defnsub}

\begin{rmksub}\label{rmk:psi}
Since $\wh D$ is $\Zp$-flat, one can find a unique injective morphism $\psi_\M:\M\ra\vphi^*\M$ such that $(1\otimes\vphi_\M)\circ\psi_\M = p\id_\M$ and $\psi_\M\circ(1\otimes\vphi_\M) = p\id_{\vphi^*\M}$.
\end{rmksub}

Let $\E$ be a Dieudonn\'e crystal over $\Spf (R,(p))$, and set $\M:=\E(\wh D)$. We define:
\begin{enumerate}
\item The  linearisation  $1\otimes\vphi_{\M}$ of $\vphi_{\M}$ is induced from  $F:\vphi^*(\ol\E) \ra \ol\E$, where $\ol\E$ is the pull-back of $\E$ over $\Spec R/(p)$;
\item $\Fil^1\M \subset \M$ is the preimage of the Hodge filtration $\Fil^1 \E(R) \subset \E(R)$ by the natural projection $\M \thra \M/(\Fil^1\wh D)\M \cong \E(R)$;
\end{enumerate}

\begin{propsub}\label{prop:FiltMod}
For $\M$ associated to a Dieudonn\'e crystal $\E$ as above, we have $(\M,\Fil^1\M, \vphi_\M,\vphi_\M/p) \in \DMFq$.
If furthermore $\wh D$ is constructed as in Remark~\ref{rmk:whD} and is $\Zp$-flat,  then there exists a natural connection $\nabla_{\M}:\M\ra \M\otimes_{A} \wh\Omega_{A}$, so that $\E\rightsquigarrow(\M,\Fil^1\M, \vphi_\M,\vphi_\M/p,\nabla_\M)$ induces an exact equivalence of categories from the category of Dieudonn\'e crystals over $\Spf (R,(p))$ to $\DMF$.
\end{propsub}
\begin{proof}
Let us first settle the proposition when $p$ is  nilpotent in $R$ (i.e., when $\Spf (R,(p)) = \Spec R$). Note that $V:\ol\E\ra\vphi^*\ol\E$ induces a $\wh D$-linear map $\psi_\M:\M\ra\vphi^*\M$ such that $(1\otimes\vphi_\M)\circ\psi_\M = p\id_\M$ and $\psi_\M\circ(1\otimes\vphi_\M) = p\id_{\vphi^*\M}$.
By Definition~\ref{def:DieudonneCrys}(\ref{def:DieudonneCrys:HodgeFiltr}) we have $\vphi^*\Fil^1\E(R/p) =\ker (F) = \im(V)$, so it follows that $(1\otimes\vphi_\M)(\vphi^*\Fil^1\M) \subseteq p\M$. In particular, the image of $\psi_\M$ is contained in the image of $\vphi^*\Fil^1\M$ in $\vphi^*\M$. So from $(1\otimes\vphi_\M)\circ\psi_\M = p\id_\M$, it follows that $(1\otimes\vphi_\M)(\vphi^*\Fil^1\M) = p\M$. The rest of the conditions are clear.

Now let $\wh D$ be the $p$-adically completed PD envelope of $A\thra R$, where $A/(p)$ locally has a finite $p$-basis. 
Then, $\E\rightsquigarrow\E(\wh D)$ induces a natural equivalence of categories between crystals and certain modules with connection. (Indeed, by the proof of  \cite[Proposition~2.2.2]{dejong:crysdieubyformalrigid} it suffices to handle the case when $R=A/(p)$, in which case we may apply \cite[Proposition~1.3.3]{Berthelot-Messing:DieudonneIII} for some Zariski cover of $\Spec A/(p)$.) We also assume that $\wh D$ is $\Zp$-flat. Then one can naturally view $\M:=\E(\wh D)$ as an object of $\DMF$ for a Dieudonn\'e crystal $\E$ over $R$. Conversely, by inverting the above equivalence of categories $\E\rightsquigarrow \E(\wh D)$, given $\M\in\DMF$ one obtains a crystal $\E_\M$ over $R$ with a horizontal isomorphism $\M\cong \E_\M(\wh D)$, and we have  $F:\vphi^*\ol\E_\M\ra\ol\E_\M$ and $V:\ol\E_\M\ra\vphi^*\ol\E_\M$ which are induced from $1\otimes\vphi_\M$ and $\psi_\M$. (Here $\psi_\M$ is defined in Remark~\ref{rmk:psi}.) 
The Hodge filtration is given by $\Fil^1\E_\M(R) = \Fil^1\M / (\Fil^1\wh D)\M$, which satisfies Definition~\ref{def:DieudonneCrys}(\ref{def:DieudonneCrys:HodgeFiltr}). So we obtain a quasi-inverse $\M\rightsquigarrow\E_\M$ of the functor $\E\rightsquigarrow \E(\wh D)$. This proves the proposition when $p$ is locally nilpotent in $R$.

Now, let $R$ be any $p$-adic ring, and set $\M:=\E(\wh D)$ for a Dieudonn\'e crystal $\E$ over $\Spf (R,(p))$. Note that all the extra structures on $\M$ except $\Fil^1\M$ depend only on $\ol\E$, which is a Dieudonn\'e crystal over $\Spec R/(p)$. Set $\M_n:=\E|_{R/(p^n)}(\wh D)$ equipped with all the extra structure as in the proposition. Then the natural reduction map $\M \ra \M_n$ is an $\wh D$-isomorphism compatible with Frobenius structures and connections if they are defined, and identifies $\Fil^1\M_n$ with $p^n\M+ \Fil^1\M$. Therefore, we have $\M\riso \varprojlim_n\M_n$ respecting all the extra structures, and the proposition for $R$ is deduced from the proposition for $R/(p^n)$.
\end{proof}

\begin{defnsub}\label{def:vphiNilp}
Let $\wh D$ be a  frame (Definition~\ref{def:Frame}), and let $\M$ be an object of either $\DMFq$ or $\DMF$.
Then $\M$ is called \emph{$\vphi$-nilpotent} if $\vphi_{\M}^n(\M)\subset p\M$ for some $n\gg0$. Similarly, $\M$ is called \emph{$\psi$-nilpotent} if $\psi_{\M}^n(\M)\subset p(\vphi^{n*}\M)$ for $n\gg0$ where $\psi_{\M} = p(1\otimes\vphi_{\M})\iv:\M\ra\vphi^*\M$. Let $\DMFq^{\vphi\nilp}$ and $\DMFq^{\psi\nilp}$ respectively denote the full subcategories of $\vphi$-nilpotent and $\psi$-nilpotent objects. We similarly define $\DMF^{\vphi\nilp}$ and $\DMF^{\psi\nilp}$.
\end{defnsub}

Note that if $\E$ is an $F$-nilpotent (respectively, $V$-nilpotent) Dieudonn\'e crystal, then $\E(\wh D)$ is $\vphi$-nilpotent (respectively, $\psi$-nilpotent). The converse holds if $\wh D$ is constructed as in Remark~\ref{rmk:whD}. 

\begin{exasub}\label{exa:S} We give an example where we can make ``nice'' choices of $A$ and $\wh D$.
Let $\bar A$ be a $\Fp$-algebra which locally admits a finite $p$-basis. Then by Lemma~\ref{lem:lifting} there is a $p$-adic $\Zp$-flat lift $A$ with a lift of Frobenius $\vphi$. In this case, we may set $\wh D = A $ and $\Fil^1\wh D = (p)$. Then it follows from \cite[Proposition~2.5.2]{dejong:crysdieubyformalrigid} that the objects in $\DMF$ are precisely \emph{crystalline Dieudonn\'e modules} defined in \cite[Definition~2.3.4]{dejong:crysdieubyformalrigid}. 
Note that  $p\in\Fil^1\wh D$ and $\vphi_1(p)=1$, so $(R_0, pR_0, R_0/(p), \vphi,\vphi_1)$ is a frame and $\PMF{A}$ makes sense.
\end{exasub}

\begin{rmksub}[lci case]\label{rmk:Slci} We present more examples when $\wh D$ is a $p$-adic flat $\Zp$-algebra. Assume that $R$ is a noetherian $\Fp$-algebra, and all the local rings are complete intersection (i.e., $R$ is locally complete intersection). Assume also that there exists a $p$-adic $\Zp$-flat noetherian ring $A$ such that $A/(p)$ locally admits a finite $p$-basis and there is a surjective map $A\thra R$. Let $\wh D$ be the $p$-adic PD envelop of $A\thra R$ as constructed in Remark~\ref{rmk:whD}. 

We now claim that such $\wh D$ is $\Zp$-flat. For any maximal ideal $\m\subset R$, let $\wh A_\m$ denote the completion of $A$ with respect to the kernel of $A\thra R/\m$. Since $A$ is noetherian, $\wh A_\m$ is a flat $A$-algebra.
 It then follows from \cite[Proposition~3.21]{Berthelot-Ogus} that $\wh D_\m:= \wh A_\m\wh\otimes_A\wh D$ is the $p$-adically completed PD envelop of $\wh A_\m \thra \wh R_\m$, where $\wh\otimes_A$ is the $p$-adic completion of the usual tensor product. 
 
By the $p$-basis assumption on $A/(p)$, the completion $\wh A_\m/(p)$ is a complete local noetherian ring with finite $p$-basis by \cite[Lemma~1.1.2]{dejong:crysdieubyformalrigid}. By cotangent complex consideration via \cite[Lemma~1.1.1]{dejong:crysdieubyformalrigid}, $\wh A_\m/(p)$ is formally smooth over $\Fp$ so it is regular by \cite[Theorems~28.7]{matsumura:crt}. As $\wh R_\m$ is assumed to be complete intersection, the kernel of $\wh A_\m\thra \wh R_\m$ is generated by a regular sequence. In such a case, it is known that the $p$-adically completed PD envelop $\wh D_\m$ is $\Zp$-flat. (\emph{Cf.} the proof of \cite[Lemma~4.7]{deJong-Messing}.) 

Note that we have $\MaxSpec R \riso \MaxSpec \wh D$ because  $\Fil^1\wh D$  is contained in the Jacobson radical (\emph{cf.} Remark~\ref{rmk:whD}). This shows that $\set{\Spf(\wh D_\m,(p))}_{\m\in\MaxSpec R}$ is an fpqc covering of $\Spf (\wh D, (p))$. As $\wh D_\m$ is $\Zp$-flat, we have $\Zp$-flatness of $\wh D$.
\end{rmksub}

\begin{rmksub}\label{rmk:BC}
Since the definition of $\DMFq$ depends upon non-canonical choices, one cannot expect to have a notion of base change for arbitrary map $f:R\ra R'$ of $p$-adic rings. If, on the other hand, one can find $(A,\vphi)$ and $(A',\vphi)$ which induce ``frames'' for $R$ and $R'$, respectively (\emph{cf.,} \S\ref{subsec:connection}), and there is a \emph{$\vphi$-compatible} map $\wt f:A\ra A'$ which reduces to $f$, then $\wt f$ extends to a map of $p$-adically completed divided power envelopes $\wh D\ra \wh D'$, which respects all the extra structures. Note that Lemma~\ref{lem:etloc} provides some examples of $f$ where there is $\wt f$ with desired properties. See \S\ref{subsec:BC} (Ex1) for more details.

Now the scalar extension $\M\rightsquigarrow \M':=\wh D'\otimes_{\wh D} \M$ induces a functor $\DMFq\ra \PMFq{\wh D'}$   as follows: we set $\vphi_{\M'}:=\vphi_{\wh D'}\otimes\vphi_\M$, and $\Fil^1\M':=\ker[\M'\thra R'\otimes_R (\M/\Fil^1\M)]$, and 
\begin{equation}
\nabla_{\M'}(s'\otimes m) = m\otimes d_{\wh D'}(s') + s'\tim \nabla_\M(m)
\end{equation}
for any $s'\in \wh D'$ and $m\in\M$.
\end{rmksub}

\subsection{Ring $S$}\label{subsec:settingBr}
We use the notation from \S\ref{cond:Breuil}, such as $R$, $R_0$, $E(u)$, $\varpi\in R$. We also fix a lift of Frobenius $\vphi:R_0\ra R_0$, which exists by Lemma~\ref{lem:lifting}. We set $\Sig:=R_0[[u]]$ and extend $\vphi$ by $\vphi(u)=u^p$. Now, let $S$ be the $p$-adically completed divided power envelope of $\Sig$ with respect to the kernel of $\Sig\thra R$, and define $\Fil^1S$, $\vphi$, $d_S$ as in Remark~\ref{rmk:whD} (for $A=\Sig$ and $\wh D = S$). We give the $p$-adic topology to all of these rings.

Recall that $R=\Sig/E(u)$. From the assumption on $E(u)$ as in  \S\ref{cond:Breuil}, it is straightforward to check the following (using binomial coefficients):
\begin{equation}\label{eqn:S}
S=\left\{
\sum_{n\geqs0} f_n  \frac{u^n}{\lfloor n/e\rfloor!}; \text { where }\ f_n\in R_0 \text{ and } f_n \to 0\ p\text{-adically}
\right\},
\end{equation}
where  $e$ is the degree of $\PP(u)$. The equality takes place in $R_0[\ivtd p][[u]]$. It  follows that $S$ is $\Zp$-flat and $\Fil^1S$ is topologically generated by  the divided powers of $E(u)$. 

Let $c:=\vphi(\PP(u))/p$. One can check by direct computation that $c \in S\starr$, so we have $\langle\frac{\vphi}{p}(\Fil^1S)\rangle = S$.  (Indeed, the map $S\thra R_0$, induced from $u\mapsto 0$, maps $c$ to $1$.)
This shows that  $(S, \Fil^1S,R, \vphi,\vphi_1)$ is a frame in the sense of Definition~\ref{def:Frame}, and  we can apply the discussions in \S\ref{subsec:connection}. In particular, $\SMF$ is in equivalence with the category of Dieudonn\'e crystals over $\Spf (R,(\varpi))$ by Proposition~\ref{prop:FiltMod}.

In addition to $d_S$, we  define another connection $d_{S}^u:S\ra S\otimes_{R_0}\wh\Omega_{R_0}$ by $d^u_{S}(s):= d_{S}(s) \bmod{du}$ for any $s\in S$; more concretely, we have $d^u_{S}(\sum_n  \frac{u^n}{\lfloor n/e\rfloor!}f_n) = \sum_n  \frac{u^n}{\lfloor n/e\rfloor!} d_{R_0}(f_n)$  for  $f_n\in R_0$ which $p$-adically tends to $0$.

We define a derivation $N:S\ra S$ by $N:=-u\frac{\partial}{\partial u}$. For any $s\in S$ we have $d_{S}(s) = (-1/u)N(s)du + d^u_{S}(s)$, and $N\circ\vphi = p\vphi\circ N$.

In addition to $\SMF$ and $\SMFq$ (\emph{cf.,} Definition~\ref{def:DieudonneModwHF}), we make the following definitions:
\begin{defnsub}\label{def:BrMod}
Let $\SMFw$ be the category of  $(\M, \Fil^1\M, \vphi_1)\in\SMFq$ equipped with  a  topologically quasi-nilpotent integrable  connection $\nabla^u_{\M}:\M\ra\M\otimes_{R_0}\wh\Omega_{R_0}$ over $d^u_{S}$ which  commutes with $\vphi_\M$.

Let $\SMFqw$ be the category of  $(\M,\Fil^1\M, \vphi_1)\in\SMFq$ equipped with a topologically quasi-nilpotent integrable connection $\nabla_{\M_0}$ on $\M_0:=R_0\otimes_{S}\M$ which commutes with $\vphi_{\M_0} = \vphi_{R_0}\otimes\vphi_\M$. We call an object in $\SMFqw$ a \emph{Breuil $S$-module}.
\end{defnsub}

\begin{rmksub}\label{rmk:BrModAlt}
For $\M\in\SMFq$, the giving of $\nabla_\M:\M\ra \M\otimes_\Sig\wh\Omega_\Sig$ that makes $\M$ an object in $\SMF$ is equivalent to the giving of  $\nabla^u_\M:\M\ra \M\otimes_{R_0}\wh\Omega_{R_0}$ and a differential operator $N_\M:\M\ra\M$ over $N$ (i.e.,  $N_\M(sm) = N(s)m + sN_\M(m)$ for any $s\in S$ and $m\in\M$), such that
\begin{enumerate}
\item $(\M,\nabla^u_\M)\in \SMFw$;
\item $N_\M(\M)\subset u\M$;  and $-\nabla^u_{\M}\circ (u\iv N_\M) = (u\iv N_\M\otimes1_{\wh\Omega_{R_0}})\circ\nabla^u_{\M}$;
\item We have $N_{\M}\circ \vphi_\M = p\cdot\vphi_\M\circ N_\M$. 
\end{enumerate}
Given $N_\M$ and $\nabla^u_{\M}$ as above, let  $\nabla_\M:\M\ra  \M \otimes_{\Sig} \wh\Omega_{\Sig}$  be the connection defined by $\nabla_\M(m) = -u\iv N_\M(m)du+\nabla^u_{\M}(m)$ for any $m\in\M$. One can check that $(\M, \nabla_\M)\in\SMF$  if and only if  $N_\M$ and $\nabla^u_{\M}$ satisfy the above conditions. (Note that the connection $\nabla_\M$ is integrable if and only if $\nabla^u_{\M}$ is integrable and anti-commutes with $u\iv N_\M$.)
Using this description, one can define a ``forgetful functor'' $\SMF\ra\SMFw$ by forgetting $N_\M$. 

Let $I_0\subset S $ be the ideal topologically generated by $u$ and $\frac{u^{ej}}{j!}$ for $j>0$. Then we have $S/I_0 \cong R_0$, and the natural projection $S\thra R_0$ commutes with $\vphi$'s. Now one can define functors $\SMF\ra\SMFqw$ and $\SMFw \ra \SMFqw$ by ``reducing the connection modulo $I_0$. (Note that these functors are compatible with the forgetful functor $\SMF\ra\SMFw$  that we constructed in the previous paragraph.)

If $\M=\E(S)\in\SMF$ for some Dieudonn\'e crystal $\E$ over $\Spf (R,(\varpi))$, then we have $R_0\otimes_S\M\cong \E_{R/(\varpi)}(R_0)$ as objects in $\PMF{R_0}$. (To see the isomorphism is horizontal, interpret a connection as a HPD stratification in the sense of \cite[Theorem~6.6]{Berthelot-Ogus}, and work out the effect of base change. We leave the details to readers.)

We can construct the functors from $\SMF$, $\SMFw$, and $\SMFqw$ into $\PMF{R_0}$ by
 $\M\rightsquigarrow\M_0:=R_0\otimes_{S}\M$ and reducing the extra structures modulo $I_0$. (\emph{Cf.} Remark~\ref{rmk:BC}.) These mod-$I_0$ reduction functors are compatible with the functors among $\SMF$, $\SMFw$, and $\SMFqw$  defined above.
 
 Finally, by working with $\Fil^rS$ for $r\leqs p-1$ we can study  ``higher weight'' cases. Also by allowing $\nabla_\M$ to have  a ``logarithmic pole'' at $u=0$ and $\N_\M\bmod{I_0}$ to be not necessarily zero, we can handle  ``log crystals''. 
 \end{rmksub}
\begin{lemsub}\label{lem:forgetfulN}
The functors $\SMF\ra\SMFw$ and $\SMFw\ra\SMFqw$, defined in Remark~\ref{rmk:BrModAlt}, are fully faithful.
\end{lemsub}
\begin{proof}
The full faithfulness of the functor $\SMFw\ra\SMFqw$ follows from Lemma~\ref{lem:Reduction}.
To show the full faithfulness of $\SMF\ra\SMFw$, we want to show that any morphism  $f:\M\ra\M'$ of objects in $\SMFw$ commutes with the differential operators $N_\M$ and $N_{\M'}$.

Consider $\delta_f:= N_{\M'}\circ f - f\circ N_\M : \M \ra \M'$. By assumption we have $\delta_f(\M)\subseteq I_0\tim \M'$ where $I_0:=\ker(S\ra R_0)$. But since $\delta_f$ also commutes with Frobenius structures, for any $m\in\Fil^1\M$ we have $\delta_f(\vphi_1(m)) = c\iv\vphi_1(\PP(u)\delta_f(m))$ where $c:=\frac{\vphi}{p}(\PP(u))$. In particular, we have $\delta_f(\M) \subseteq\vphi(I_0)\tim \M'$ as $\vphi_1(\Fil^1\M)$ generates $\M$. By repeating this, we obtain $\delta_f(\M)\subseteq \bigcap_n \vphi^n(I_0) \tim\M' = \set0$.
\end{proof}

\begin{lemsub}\label{lem:Reduction}
Let $\M\in\SMFw$, and set $\M_0:=R_0\otimes_S\M\in\PMF{R_0}$. Then there exists a unique $\vphi$-compatible section $s:\M_0[\ivtd p]\ra \M[\ivtd p]$. Furthermore, $s$ is horizontal and the map $1\otimes s:S[\ivtd p]\otimes_{R_0}\M_0\riso \M[\ivtd p]$ is an isomorphism.
\end{lemsub}
\begin{proof}
Let us choose an arbitrary section $s_0:\M_0\ra \M$, and we would like to show that the following ``formula'' gives a well-defined morphism $s:\M_0[\ivtd p]\ra\M[\ivtd p]$:
\begin{align*}
s:= &\text{``}\lim_{n\ra\infty} (1\otimes\vphi_\M^n) \circ (\vphi^{*n} s_0)\circ(1\otimes\vphi_{\M_0}^{-n})\text{''}\\ 
= & s_0+\sum_{n=0}^\infty \bigg[(1\otimes\vphi_\M^{n+1}) \circ (\vphi^{*n+1} s_0)\circ(1\otimes\vphi_{\M_0}^{-n-1})  - (1\otimes\vphi_\M^n) \circ (\vphi^{*n} s_0)\circ(1\otimes\vphi_{\M_0}^{-n}) \bigg].
\end{align*}
Note that $(1\otimes\vphi_{\M_0})[\ivtd p]:\vphi^*\M_0[\ivtd p]\ra\M_0[\ivtd p]$ is an isomorphism, and the preimage of $\M_0$ is contained in $ p\iv(\vphi^*\M_0)$.

Set $s_{n+1}:=(1\otimes\vphi_\M^{n+1}) \circ (\vphi^{*n+1} s_0)\circ(1\otimes\vphi_{\M_0}^{-n-1})  - (1\otimes\vphi_\M^n) \circ (\vphi^{*n} s_0)\circ(1\otimes\vphi_{\M_0}^{-n})$. Clearly $s_1(\M_0) \subseteq p\iv u \M$, because $s_0$ is a section. By iterating this, one obtains that $s_{n+1}(\M_0) \subseteq p^{-n-1} u^{p^n} \M$ for any $n\geqs0$. Note that $p^{-n-1} u^{p^n}= p^{-n-1}(q_n!) u^{p^n-q_ne}(u^e)^{[q_n]}$ where $q_n:=\lfloor \frac{p^n}{e}\rfloor$, so we have 
\begin{equation*}
\ord_p\big(p^{-n-1}(q_n!) \big) \geqs -n-1 +p^{n-1}/e -1 \stackrel{n\ra\infty}{\to}\infty
\end{equation*}
This proves that $s:\M_0[\ivtd p]\ra\M[\ivtd p]$ is well-defined, and by construction $s$ commutes with $\vphi$'s. The map $1\otimes s:S[\ivtd p]\otimes_{R_0}\M_0\riso \M[\ivtd p]$ is a surjective map of projective $S[\ivtd p]$-modules of same rank (by Nakayama lemma), so it is an isomorphism.

Now, if there are two sections $s$ and $s'$, then $(s-s')(\M_0[\ivtd p])\subseteq I_0\M[\ivtd p]$. If both $s$ and $s'$ commute with $\vphi$'s then we obtain
\begin{multline*}
(s-s')(\M_0[1/p]) \subseteq (s-s')\big((1\otimes\vphi_{\M_0})^n(\vphi^{*n}\M_0[1/p])\big)\\
\subseteq (1\otimes\vphi_\M^n)\bigg(\vphi^{*n}\big((s-s')(\M_0[1/ p])\big)\bigg) \subseteq \vphi^n(I_0)\M[1/p].
\end{multline*}
Therefore, we have $\im(s-s')\subseteq\bigcap_n \vphi^n(I_0)\M[\ivtd p] = \set0$, which establishes the uniqueness of $s$.

It is left to show that $s$ is horizontal. Consider $\delta_s:=(s\otimes1)\circ\nabla_{\M_0} - \nabla^u_{\M}\circ s:\M_0[\ivtd p]\ra\M[\ivtd p]\otimes_{R_0}\wh\Omega_{R_0}$. By construction, $\delta_s$ commutes with $\vphi$'s and its image is divisible by $u$. So by the same argument that shows the uniqueness of $s$, we conclude that $\delta_s=0$.
\end{proof}

\begin{rmksub}[lci case]\label{rmk:Brlci}
We can extend the construction of $S$ and Lemma~\ref{lem:forgetfulN} in the following variant as well. Let $R$ be a $p$-adic flat $\Zp$-algebra which is locally complete intersection. Assume that there exists a $p$-adic noetherian flat $\Zp$-algebra $A$ such that
\begin{itemize}
\item
$A/(p)$ locally admits a finite $p$-basis;
\item
there exists a surjective map $A[[u]]\thra R$, where $u$ maps to a regular element $\varpi\in R$;
\item the kernel of $A[[u]]\thra R$ includes an element $E(u)= p+\sum_{i=1}^e a_i u^i$  with $a_i\in A$ and $a_e\in A\starr$.
\end{itemize}
Let $S_0$ be the $p$-adically completed PD envelop of $A\thra R/(\varpi)$, and $S$ be the $p$-adically completed PD envelop of $A[[u]]\thra R$. By Remark~\ref{rmk:Slci}, $S_0$ defines a frame. We also have the description of $S$ similar to (\ref{eqn:S}) with $R_0$ replaced by $S_0$, so  $S$ is also  $\Zp$-flat and $\vphi_1(E(u))\in S\starr$; in other words, $S$  defines a frame.

Now, by simply replacing $R_0$ with $S_0$, we can define $\SMF$, $\SMFw$, and $\SMFqw$, and the proof of Lemma~\ref{lem:forgetfulN} generalises to this setting.
\end{rmksub}

\subsection{Review of results by Berthelot-Messing and de\thinspace{}Jong}\label{subset:BMdJ}
Let $\bar A$ be an $\Fp$-algebra locally admitting a finite $p$-basis, and choose $(A,\vphi)$ lifting $\bar A$ and the Frobenius morphism (which exists by Lemma~\ref{lem:lifting}).
For any $p$-divisible group $G_0$ over $\bar A$ we  set   $\M_0^*(G_0):=\DD^*(G_0)(A)$, which belongs to   $\PMF{A}$ by Example~\ref{exa:S}. We recall the following results on the functors $\DD^*$ and $\M_0^*$:
\begin{enumerate}
\item\label{subset:BMdJ:dJ} Suppose that $\bar A$ is formally smooth and formally of finite type over some field $k$ which admits a finite $p$-basis. (For example, $\bar A=R/(\varpi)$ where $R$ satisfies the formally finite-type assumption~(\S\ref{cond:dJ}).) Then $\M_0^*$ is an anti-equivalence of categories. If $\ol \XX$ is a scheme that has a Zariski covering by such $\Spec  \bar A$'s, then $\DD^*$ over $\ol\XX$ is an anti-equivalence of categories. (\emph{Cf.} \cite[Main~Theorem~1]{dejong:crysdieubyformalrigid}.)
\item\label{subset:BMdJ:BM} Assume that $\ol\XX$ is  a normal $\Fp$-scheme which is locally irreducible and locally admits a finite $p$-basis. Then $\DD^*$ over $\ol\XX$ is fully faithful. (\emph{Cf.} \cite[Th\'eor\`eme~4.1.1]{Berthelot-Messing:DieudonneIII}.) The same statement holds for $\M^*_0$ if $\ol\XX$ is affine (for example, $\ol \XX = \Spec R/(\varpi)$ where $R$ satisfies the normality assumption~(\S\ref{cond:BM})). 
\end{enumerate}

Let us set up the notation for the the main theorem of this section below. Assume that $\XX$ is a formal scheme which has a Zariski covering $\{\Spf (R_\alpha, (\varpi))\}$ where $R_\alpha$ satisfies the condition in \S\ref{subsec:settingBr}. Consider the exact  functor  
\begin{equation*}
\big\{p\text{-divisible groups over }R \big\} \xra{\DD^*} 
\big\{\text{Dieudonn\'e crystals over }\XX \text{ (Definition~\ref{def:DieudonneCrys})} \big\}.
\end{equation*}
If $\XX = \Spf (R,(p))$ where $R$ satisfies the condition in \S\ref{subsec:settingBr}, then consider the exact functor 
\begin{equation*}
\M^*:\big\{p\text{-divisible groups over }R \big\} \ra \SMFqw
\end{equation*}
defined by $\M^*(G):=\DD^*(G)(S)$ for $S$ as chosen in \S\ref{subsec:settingBr}. Note that $\M^*(G)$ is \emph{a priori} an object in $\SMF$  by Proposition~\ref{prop:FiltMod}, but we view it as an object in $\SMFqw$ via the fully faithful functor $\SMF \ra \SMFqw$ constructed in Remark~\ref{rmk:BrModAlt}. (\emph{Cf.} Lemma~\ref{lem:forgetfulN}.)

\begin{thm}\label{thm:BreuilClassif}
Assume that $p>2$.

If $\XX$ is Zariski locally covered by $\Spf (R,(\varpi))$ where each $R$ satisfies the formally finite-type assumption~(\S\ref{cond:dJ}), then $\DD^*$ and $\M^*$ (if the latter is defined) are anti-equivalences of categories.

If $\XX$ is Zariski locally covered by $\Spf (R,(\varpi))$ where each $R$ satisfies the normality assumption~(\S\ref{cond:BM}), then $\DD^*$ and $\M^*$ (if the latter is defined) are fully faithful.

If $p=2$, then in both cases above the functors $\DD^*$ and $\M^*$ (if the latter is defined) are fully faithful up to isogeny. 
\end{thm}
When $p=2$, we can obtain a better result when restricted to formal or unipotent $p$-divisible groups (\emph{cf.} Theorem~\ref{thm:FormalBreuil}, Corollary~\ref{cor:BreuilClassif}). 
\begin{proof}
To prove the theorem it suffices to prove the statement for $\M^*$ when $\XX = \Spf (R,(\varpi))$ (by Proposition~\ref{prop:FiltMod} and Lemma~\ref{lem:forgetfulN}), so we assume this from now on.

Let us outline the idea. First, assume that $p>2$. Then we define a functor $G^*$  on $\SMFqw^\DD$ and show that it is a quasi-inverse of $\M^*$  if $\M_0^*$ (as defined in \S\ref{subset:BMdJ}) is fully faithful (respectively, quasi-inverse up to isogeny if $\M^*_0$ is fully faithful up to isogeny). This strategy can be modified to show our claim when $p=2$.  The construction of the functor $G^*$ is motivated by the proof of  \cite[Proposition~A.6]{kisin:fcrys}, while we start with the results of de\thinspace{}Jong and Berthelot-Messing (which is recalled in \S\ref{subset:BMdJ}) instead of classical Dieudonn\'e theory.

Let $\SMFqw^\DD$ be the full subcategory of $\M\in\SMFqw$ such that $R_0\otimes_S\M$ is in the essential image of $\M^*_0$ (defined in \S\ref{subset:BMdJ}). Clearly, $\SMFqw^\DD$ contains the essential image of $\M^*$. (Note that if $R$ satisfies the formally finite-type assumption~(\S\ref{cond:dJ})  then $\SMFqw^\DD = \SMFqw$ by de\thinspace{}Jong's theorem.)

Suppose $(\M,\Fil^1\M, \vphi_1,\nabla_{\M_0})\in\SMFqw^\DD$, and set $\M_0:=R_0\otimes_{S}\M \in \PMF{R_0}$. By assumption,  there exists a $p$-divisible group $G_0$ over $R/(\varpi)$  such that $\M_0\cong \DD^*(G_0)(R_0)$. 
To construct $G$ corresponding to $\M$, we will lift $G_0$ over $R$ in a functorial way using Grothendieck-Messing deformation theory \cite{messingthesis}. 

For any non-negative integer $i$, let $I_i\subset S$ be the ideal topologically generated by $u^{i+1}$ and $\frac{u^{ej}}{j!}$ for any $ej>i$. By applying the discussion of Remark~\ref{rmk:BC} to $f:R \thra R/(\varpi^{i+1})$, $A' = \Sig/(u^{i+1})$ and $\wh D'= S_{i}:= S/I_i$, it follows that $\M/I_i\M$ has a natural structure of $\PMFq{S_i}$. Recall that for any $p$-divisible group $G_i$ over $R/(\varpi^{i+1})$, we have $\DD^*(G_i)(S_i)\in \PMFq{S_i}$ by Proposition~\ref{prop:FiltMod},

\begin{claimsub}\label{clm:BreuilClassif:Modp}
For any $0\leqs i < e$, there exist a $p$-divisible group $G_i$ over $R/(\varpi^{i+1})$ and an isomorphism  $\DD^*(G_i)(S_{i})\cong S_{i}\otimes_{S}\M$ in $\PMFq{S_i}$ which lifts the isomorphism $\DD^*(G_{i})(R_0) \cong R_0\otimes_{S}\M$. This construction $\M\rightsquigarrow G_i$ is functorial in $G_0$ and $\M$.
\end{claimsub}
The case when $i=0$ is obvious. Now assuming the claim for $i-1$ for some $0<i< e$, and we will prove the claim for $i$. Consider $\Sig_i:=\Sig/(u^{i+1}) \thra R/(\varpi^i)$, and let $\wh D_i$  be the $p$-adically completed divided power envelope of $\Sig_i$ with respect to the kernel. Then we have $\wh D_i \cong S_i$ with $\Fil^1\wh D_i = I_{i-1}+\Fil^1 S_i$. The natural projection $\Sig\thra \Sig/(u^{i+1})=\Sig_i$ induces a map $S\thra \wh D_i$ respecting all the extra structure, so we obtain the scalar extension functor $\SMFq \ra \PMFq{\wh D_i}$. Note that $\Fil^1(\wh D_i\otimes_S\M)$ is generated by $I_{i-1}\M/I_i\M$ and the image of $\Fil^1\M$. Similarly, we obtain the scalar extension functor $\PMFq{\wh D_i} \ra \PMFq{S_{i-1}}$.

By Proposition~\ref{prop:FiltMod} we have $\DD^*(G_{i-1})(\wh D_i)\in\PMFq{\wh D_i}$. Applying Lemma~\ref{lem:BreuilClassifLifting} to $\wh D = \wh D_i$ and $\wh J:=I_{i-1}\wh D_i$, we deduce that there is a unique $\vphi$-compatible isomorphism $\DD^*(G_{i-1})(\wh D_i)\cong  \wh D_i\otimes_{S}\M$ which lifts the isomorphism $\DD^*(G_{i-1})(S_{i-1})\cong S_{i-1}\otimes_{S}\M$ given by the induction hypothesis. 

Consider the divided power thickening $R/(\varpi^{i+1})\thra R/(\varpi^i)$ with the ``trivial'' divided power structure (i.e., $(\varpi^i)^{[j]}=0$ for any $j>1$). This is compatible with the canonical divided power structures on $p \Zp$. Choose a lift $G_i$ of $G_{i-1}$ which corresponds to the filtration defined by the image of  $\Fil^1\M$ in $R/(\varpi^{i+1})\otimes_{S}\M$ via Grothendieck-Messing deformation theory. Then by construction we have natural $\vphi$-compatible isomorphisms $\DD^*(G_i)(S_i)\cong\DD^*(G_{i-1})(\wh D_i)\cong S_i\otimes_S\M$, which takes $\Fil^1\DD^*(G_i)(S_i)$ to the image of $\Fil^1\M$ in $S_i\otimes_S\M$.
Note also that  the formation of $G_i$ is functorial in $G_{i-1}$ and the filtration. (See \cite[Ch.V, Theorem~1.6]{messingthesis} for the precise statement.) This proves Claim~\ref{clm:BreuilClassif:Modp}.

Consider $\Sig \thra R/(p)=R/(\varpi^{e})$ and let $\wh D_{e}$ be the $p$-adically completed divided power envelope of $\Sig$ with respect to the kernel. Then we have $\wh D_{e}\cong S$ with $\Fil^1\wh D_{e} = I_{e-1 }+ \Fil^1S$. As in the proof of Claim~\ref{clm:BreuilClassif:Modp}, we have scalar extension functors $\SMFq \ra \PMFq{\wh D_{e}}$ which sends $(\M,\Fil^1\M)$ to $\wh D_{e}\otimes_S\M\cong (\M, I_{e-1}\M+\Fil^1\M)$. Applying Lemma~\ref{lem:BreuilClassifLifting} to $\wh D=\wh D_{e}$ and $\wh J = I_{e-1}$ and proceeding similarly to the proof of Claim~\ref{clm:BreuilClassif:Modp}, we deduce that there is a unique $\vphi$-compatible isomorphism $\DD^*(G_{e-1})(\wh D_{e})\cong \M$ which lifts the isomorphism \[\DD^*(G_{e-1})(S_{e-1})\cong S_{R,e-1}\otimes_{S}\M\] in $\PMFq{S_{e-1}}$.

Now assume that $p>2$, in which case  $pR$ is a topologically nilpotent divided power ideal. Let $G_\M$ be the lift of $G_{e-1}$ which corresponds to the filtration \[\Fil^1\M/(\Fil^1S)\M\subset \M/(\Fil^1S)\M.\] (Note that $G_\M$ functorially depends on $\M$ and $G_0$.) As in the proof of Claim~\ref{clm:BreuilClassif:Modp}, we obtain a natural isomorphism $\M^*(G_\M)\cong \M$ in $\SMFqw$, which proves essential subjectivity of $\M^*$ for $p>2$.

If $\M^*_0$, defined in \S\ref{subset:BMdJ}, is fully faithful (which is the case if $R$ satisfies the normality assumption~(\S\ref{cond:BM})), then  we  define a functor $G^*$ from $\SMFqw^\DD$ to the category of $p$-divisible groups over $R$ by choosing a quasi-inverse $R_0\otimes_S\M\rightsquigarrow G_0$ of $\M_0^*$ and setting $G^*(\M):=G_\M$ to be the lift of $G_0$ produced by the procedure described above. Then by construction we have a natural isomorphism $(\M^*\circ G^*)(\M)\cong\M$ for any $\M\in\SMFqw^\DD$. We also get a natural isomorphism $(G^*\circ \M^*)(G)\cong G$ for any $p$-divisible group $G$ over $R$  from the uniqueness of each deformation step. This settles the theorem when $p>2$.

Assume that $p=2$ and we show that the functor $\M^*$ is fully faithful up to isogeny; indeed, we show that for any $p$-divisible groups  $G$ and $G'$ over $R$ and a morphism $\alpha:\M^*(G') \ra\M^*(G)$ there exists a unique morphism $f':G\ra G'$ such that $\M^*(f') = p^2\alpha$. As in the case when $p>2$, $\alpha$ gives rise to a morphism $f_{R/p}:G_{R/p} \ra G'_{R/p}$. By  \cite[Lemma~1.1.3]{Katz:SerreTate}, it follows that $p^2 f_{R/p}$ lifts to a unique morphism $f'_{R/p^2}:G_{R/p^2} \ra G_{R/p^2}'$. Applying the Grothendieck-Messing deformation theory to the (topologically nilpotent!) divided power thickening $R\thra R/p^2$, there exists a unique morphism $f':G\ra G'$ which lifts $f'_{R/p^2}$ and satisfies $\M^*(f') = p^2\alpha$.
\end{proof}

\begin{lemsub}\label{lem:BreuilClassifLifting}
Let $\wh D$ be a frame (\emph{cf.,} Definition~\ref{def:Frame}) such that $\Fil^1\wh D$ is a divided power ideal, and $\wh J\subseteq \Fil^1\wh D$ be a divided power sub-ideal such that such that $\vphi^n(\wh J) \subset p^{n+j_n}\wh J$ for some $\set{j_n}$ with $j_n\to\infty$ as $n\to\infty$. (In particular, $\wh D/\wh J$ is automatically $\Zp$-flat so $\wh D/\wh J$ has a natural frame structure induced from $\wh D$, and the reduction modulo $\wh J$ induces a functor $\DMFq \ra \PMFq{\wh D/\wh J}$.) 

Then for  $\M_1, \M_2\in\DMFq$, any isomorphism $\bar\theta:\M_1/\wh J\M_1 \riso \M_2/\wh J\M_2$ in $\PMFq{\wh D/\wh J}$ uniquely lifts to a $\vphi$-compatible $\wh D$-linear isomorphism $\theta:\M_1\riso\M_2$.
\end{lemsub}
\begin{proof}
The proof is identical to the proof of \cite[Lemma~A.4]{kisin:fcrys}. Let us first shows the uniqueness. Let $\theta$ and $\theta'$ be two lifts as in  the statement. Then since $\theta-\theta'$ commutes with $\vphi$'s, we have $(\theta-\theta')(\M_1)\subseteq\bigcap_{n\geqs0}\vphi^n(\wh J)(\M_2)=\set0$.

Let us show the existence. Note first that $1\otimes\vphi_{\M_i}$ is injective; indeed, $\wh D$ is $p$-torsion free by assumption, and $(1\otimes\vphi_{\M_i})[\ivtd p]$ is an isomorphism because it is a surjective map of finite projective $\wh D[\ivtd p]$-modules of the same rank. Therefore $(1\otimes\vphi_{\M_1})/p:\Fil^1(\vphi^*\M_i)\ra \M_i$ is an isomorphism, where $\Fil^1(\vphi^*\M_i)\subset \vphi^*\M_i$ is the image of $\vphi^*(\Fil^1\M_i)$. (Note that $\vphi$ is not necessarily flat.)

Pick an arbitrary lift $\theta_0:\M_1\ra \M_2$ of $\bar\theta$. As $\theta$ in the statement is not required to respect $\Fil^1$'s, we may replace $\Fil^1\M_i$ by $\wh J\M_i+\Fil^1\M_i$ to assume that $\theta_0(\Fil^1\M_1)\subset \Fil^1\M_2$. Let us recursively define lifts $\theta_n$ of $\bar\theta$, such that $\theta_0$ is the chosen lift (when $n=0$), and for any  $n\in\Z_{\geqs0}$ we have the following diagram:
\begin{equation}\label{eqn:BreuilClassifLifting}
\xymatrix{
\Fil^1(\vphi^*\M_1) \ar[r]^{\vphi^*\theta_n} \ar[d]^\cong_{(1\otimes\vphi)/p} & \Fil^1(\vphi^*\M_2) \ar[d]^{(1\otimes\vphi)/p}_\cong \\
\M_1 \ar[r]_{\theta_{n+1}} & \M_2}
\end{equation}

Now we  show by induction that $(\theta_{n+1}-\theta_n)(\M_1)\subseteq (\vphi/p)^n(\wh J)\M_2$ as follows. For any $x\in \M_1$, we can find $y\in\Fil^1(\vphi^*\M_1)$ such that $(1\otimes\vphi_{\M_1})(y)=x$. When $n=0$, we have by construction:
\[
(\theta_{1}-\theta_0)(x) = \left[(1\otimes\vphi_{\M_2}/p)\circ(\vphi^*\theta_0) - \theta_0\circ(1\otimes\vphi_{\M_1}/p)\right](y) \in\wh J\M_2.
\]
Now assume that $(\theta_{n}-\theta_{n-1})(\M_1)\in (\vphi/p)^{n-1}(\wh J)\M_2$ for some $n>0$. Then we have
\[
(\theta_{n+1}-\theta_n)(x)=  \left[(1\otimes\vphi_{\M_2}/p)\circ\vphi^*(\theta_{n}-\theta_{n-1}) \right](y)\in (\vphi/p)^n(\wh J)\M_2.
\]
By the assumption on $\wh J$, the series $\theta:=$``$\lim_{n\to\infty}\theta_n$''$ = \theta_0+\sum_{n=0}^\infty(\theta_{n+1}-\theta_n)$ converges, and clearly $\theta$ satisfies all the requirements in the statement. 
\end{proof}

The following theorem can be obtained from the same proof as Theorem~\ref{thm:BreuilClassif}:
\begin{thmsub}\label{thm:FormalBreuil}
Let $p$ be any prime, including $p=2$, and assume that $R$ satisfies the formally finite-type assumption~(\S\ref{cond:dJ}). Then $\M^*$ induces an anti-equivalence of categories from the category of formal (respectively, unipotent) $p$-divisible groups over $R$ to $\SMFqw^{\vphi\nilp}$ (respectively, $\SMFqw^{\psi\nilp}$). A similar result holds for $\DD^*$ over a formal scheme base $\XX$ which can be covered by $\Spf(R,(\varpi))$ for such $R$.
\end{thmsub}
\begin{proof}
For formal or unipotent $p$-divisible groups, the Grothendieck-Messing deformation theory holds for  $R\thra R/(p)$ (\emph{cf.} \cite[\S3.3, Corollary~97]{Zink:DisplayFormalGpAsterisq278}), hence the proof of Theorem~\ref{thm:BreuilClassif} works for any $p$.  Lemma~\ref{lem:FormalUnip} shows that the essential image is as desired; \emph{cf.} the discussion following Definition~\ref{def:vphiNilp}.
\end{proof}

\begin{rmksub}
In general, we cannot replace $\SMFqw$ in Theorem~\ref{thm:BreuilClassif} with $\SMFq$. On the other hand, there are cases when we can ``forget the connection'' -- see Corollary~\ref{cor:BreuilClassif} for more details.
\end{rmksub}

\begin{rmksub}[lci case]\label{rmk:lciBreuilClassif}
Let $R$, $\varpi$, $S_0$,  $S$, and $E(u)$ be as in Remark~\ref{rmk:Brlci}. We additionally assume that $R$ is \emph{excellent} (as well as locally complete intersection); excellence is satisfied, for example, if $A/(p)$ is formally finitely generated over some field, by \cite[Theorem~4]{Valabrega:ExcellencePS}. Then for a $p$-divisible group $G_0$ over $R/(\varpi)$, let $\M_0^*(G_0):=\DD^*(G_0)(S_0) \in \PMF{S_0}$. Then the functor $\M_0^*$ is fully faithful by  \cite[Theorem~4.6]{deJong-Messing}. 

For a $p$-divisible group $G$ over $R$, let $\M^*(G):=\DD^*(G)(S)\in\SMFqw$. By precisely the same proof of Theorem~\ref{thm:BreuilClassif}, we can obtain that $\M^*$ is fully faithful if $p>2$, and it is fully faithful up to isogeny if $p=2$. 
\end{rmksub}

\subsection{Base Change}\label{subsec:BC}
Let $f:R\ra R'$ be a map where both $R$ and $R'$ satisfying  the $p$-basis assumption~(\S\ref{cond:Breuil}). We choose $(\Sig,\vphi)$ and $(\Sig',\vphi)$ for $R$ and $R'$ as in \S\ref{subsec:settingBr}, respectively.
As we have  observed in Remark~\ref{rmk:BC}, the ``base change'' of the functor $\M^*$ under $f$ can be defined if there exists a $\vphi$-compatible morphism $\Sig\ra \Sig'$ which reduces to $f$.

We have already observed (in Remark~\ref{rmk:BC}), the assumption is satisfied in the following cases:
\begin{enumerate}
\item[(Ex1)] Assume that $R$ is $J$-adic and $p\in J$. Let $R'$ be a $J$-adic formally \'etale algebra such that $R/J\ra R'/JR'$ is finitely generated (i.e., the morphism $\Spf(R',JR')\ra \Spf(R,J)$ is \'etale). Set $\wt J:=\ker (R_0\thra R/J)$ and assume that $R_0$ is $J_0$-adically complete (by replacing $R_0$ with the $J_0$-adic completion if necessary). Then by Lemma~\ref{lem:etloc} one can find a unique pair $(R_0',\vphi)$ so that $R\ra R'$ lifts to $\Sig\ra\Sig':=R_0'[[u]]$ respecting $\vphi$.
\end{enumerate}

We list some other cases where this assumption is satisfied:
\begin{enumerate}
\item[(Ex2)] $R$ is a discrete valuation ring with residue field $k$ (with fintie $p$-basis) and $R'$ is a $p$-adic flat $R$-algebra such that $R'/(\varpi)$ locally admits a finite $p$-basis (\emph{cf.} Lemma~\ref{lem:sec});
\item[(Ex3)] if $R'=R[[x_i]]\langle y_j\rangle$, then we may take $R'_0=R_0[[x_i]]\langle y_i\rangle$;
\item[(Ex4)] Let $W\ra W'$ be a map of Cohen rings with residue fields admitting finite $p$-bases. By Lemma~\ref{lem:lifting}, there is a lift of Frobenius $\vphi':W'\ra W'$ which leaves $W$ invariant. Assume that $R$ and $R_0$ are $W$-algebras, and choose $\vphi:R_0\ra R_0$ over $\vphi:W\ra W$. 
Now, we set $R':=R\wh\otimes_W W'$, $R_0':= R_0\wh\otimes_W W'$, and $\vphi_{R_0'} = \vphi_{R_0}\otimes \vphi$.
\item[(Ex5)] Let $\p\subset R$ be a prime ideal containing $\varpi$,  and set $\p_0:=\ker (R_0\thra R/\p)$. (For example, if $\p=(\varpi)$ then $\p_0=(p)$.) If $R'$ is the $\varpi$-adic completion of the localisation $R_\p$ then we may take $R'_0$ to be the $\varpi$-adic completion of $(R_0)_{\p_0}$. Note that $R_0'$ and $R'$ are actually $p$-adic without noetherian-ness assumption on $R$ (\emph{cf.,} Corollaire~2 in \cite[Ch. III, \S2, no.~12, page~228]{Bourbaki:AlgComm1-4}), and $\vphi_{R_0}$ extends to $R_0'$ because $\vphi\iv(\p_0) = \p_0$. (Note that $\vphi$ induces the identity map on the underlying topological space of $\Spec R/(\varpi)$.)

The same construction works if $\p$ and $\p_0$ are finitely generated prime ideals, and $R'$ is the $\p$-adic completion of $R_\p$ and $R_0'$ is the $\p_0$-adic completion of $(R_0)_{\p_0}$.
\item[(Ex6)] Let $R'_0$ be the $p$-adic completion of $\varinjlim_\vphi (R_0)_{(p)}$, and $R':=R'_0\otimes_{R_0}R$ .
\end{enumerate}
Note that in (Ex6) case we have a $\vphi$-compatible isomorphism $R_0'\cong W(k')$ by the universal property of Witt vectors over perfect rings, where $k':=\varinjlim_\vphi\Frac(R/(\varpi))$ is the perfect closure of $\Frac(R/(\varpi))$.

\section{Review: relative $p$-adic Hodge theory}\label{sec:Brinon}
In this section, we recall (and slightly generalise) the construction and basic properties of relative $p$-adic period rings. We refer to \cite{Brinon:Habilitation} for a brief, but more complete, overview of relative $p$-adic Hodge theory.

\subsection{Period rings}\label{subsec:Acris}
Let $R$ be a \emph{normal domain} which satisfies the $p$-basis assumption~(\S\ref{cond:Breuil}). (We may soon assume further that $R$ satisfies the ``refined almost \'etaleness'' assumption~(\S\ref{cond:RAF}) later.) 
Choose a separable closure $E$ of $\Frac(R)$, and define $\ol R$ to be the union of normal $R$-subalgebras $R'\subset E$ such that $R'[\ivtd p]$ is  finite \'etale over $R[\ivtd p]$. Note that $\Spec \ol R[\ivtd p]$ is a pro-universal covering of $\Spec R[\ivtd p]$ and $\ol R$ is an integral closure of $R$ in $\ol R[\ivtd p]$. Set $\wh{\ol R} :=\varprojlim_n\ol R/(p^n)$ and $\gal_R:=\Gal(\ol R[\ivtd p]/R[\ivtd p]) = \pi_1^{\et}(\Spec R[\ivtd p], \bar\eta)$ where $\bar\eta:\Spec R[\ivtd p]\ra\Spec E$. (The ``correct'' notation for $\GRR$ should be $\gal_{R[\frac{1}{ p}]}$, but we suppress this for the typographical reason. When $R=\fo_K$ is a discrete valuation ring we allow both notations $\gal_{\fo_K} = \gal_K$.) When $R=\fo_K$, we have $\ol R = \fo_{\ol K}$, $\wh{\ol R} = \fo_{\C_K}$, and $\gal_R = \Gal(\Kbar/K)$.

Let $\ol R^\flat:=\invlim_{x\mapsto x^p} \ol R/(p)$, which is a perfect ring equipped with a natural $\gal_R$-action. (We follow the notation of Scholze \cite[Lemma~6.2]{Scholze:Perfectoid}; perhaps, $\wh{\ol R}^\flat$ would be a more precise notation as $(\wh{\ol R}[\ivtd p], \wh{\ol R})$ is a perfectoid affinoid $\fo_{\CK}$-algebra, but there is no danger of confusion for using $\ol R^\flat$. See \S\ref{subsec:perfectoid} for (slightly) more discussions on perfectoid algebras.) As in the classical case (i.e., $R=\fo_K$ with perfect residue field), there is a natural multiplicative bijection $\ol R^\flat\riso \invlim_{x\mapsto x^p} \wh{\ol R}$, defined by the component-wise reduction modulo~$p$. To see that it is an isomorphism, we construct its inverse as follows. For any $(x_n)_{n\in \Z_{\geqs0}}\in \ol R^\flat$, define
\begin{equation}
x^{(n)}:=\lim_{m\to\infty}(\tilde x_{m+n})^{p^m}
\end{equation}
for any lift $\tilde x_{m+n}\in \wh{\ol R}$ of $x_{m+n}\in \ol R/(p)$. Note that $x^{(n)}$ is independent of all the choices, and $(x_n)_{n\in \Z_{\geqs0}} \mapsto (x^{(n)})_{n\in \Z_{\geqs0}}$ is the desired inverse.

For any  $a\in \wh{\ol R}$,  there exists an element $\wt a:=(a^{(n)})\in \ol R^\flat$ with $a^{(0)} = a$. (Note that $\wt a$  is not uniquely determined by $a$.) Another useful element is $\epsilon = (\epsilon^{(n)})$ where $\epsilon^{(0)} = 1$ and $\epsilon^{(1)}\ne 1$. The choice of $\epsilon$ is equivalent to the choice of a $\Zp$-basis of $\Zp(1)$.

Consider the following map:
\begin{equation}\label{eqn:theta}
\theta:W(\ol R^\flat)\ra \wh{\ol R}\ ,\quad \theta(a_0,a_1,\cdots) := \sum_{n=0}^\infty p^n a_n^{(n)}.
\end{equation}  
The same proof as \cite[Propositions~5.1.1, 5.1.2]{Brinon:CrisDR} shows that $\theta$ is a surjective ring homomorphism with  kernel  principally generated by $\xi:=p-[\wt p]$, where $[\cdot]$ denotes the Teichm\"uller lift. 

Let $\BdR^{\nabla,+}(R)$ be the $\ker(\theta)$-adic completion of $W(\ol R^\flat)[\ivtd p]$. Then $t:=\log[\epsilon]$ makes sense as an element of $\BdR^{\nabla,+}(R)$ (and indeed, it actually lies in the classical de Rham period ring $\BdR^+(\Zp)$). Set $\BdR^\nabla(R):=\BdR^{\nabla,+}(R)[\ivtd t]$. These rings carry natural $\gal_R$-actions and filtrations $\Fil^r\BdR^\nabla(R)=t^r\BdR^{\nabla,+}(R)$, which coincides with the $\ker(\theta)$-adic filtration.

Consider the $R$-linear extension $\theta_{R}:R\otimes_{\Zp}W(\ol R^\flat)\ra \wh{\ol R}$  of $\theta$, and set $\Ainf(R):=\varprojlim_n \big(R\otimes_{\Zp}W(\ol R^\flat)\big)/\big(\theta_R\iv(p\wh{\ol R})\big)^n$. Define $\BdR^+(R)$ to be the $\ker(\theta_R)$-adic completion of $\Ainf(R)[\ivtd p]$, and $\BdR(R):=\BdR^+(R)[\ivtd t]$, where $t:=\log[\epsilon]$, which makes sense in $\BdR^+(R)$. These rings carry  natural $\gal_R$-actions. For filtration, we set $\Fil^r\BdR^+(R):=(\ker\theta_R)^r\BdR^+(R)$ for $r\geqs0$ and extend it to $\Fil^r\BdR(R):=\sum_{n\geqs-r}\ivtd{t^n}\Fil^{n+r}\BdR^+(R)$.  When $R=\fo_K$ with perfect residue field, $\BdR^{\nabla}(\fo_K) = \BdR(\fo_K)$ is the usual de~Rham period ring constructed by Fontaine \cite{fontaine:Asterisque223ExpII}. If it is possible to define a connection on $\BdR(R)$ (\emph{cf.,} Proposition~\ref{prop:PeriodRings}), then we have $\BdR^\nabla(R) = (\BdR(R))^{\nabla=0}$, which explains the notation.

We define $\Acris^\nabla(R)$ to be the $p$-adically completed divided power envelope of $W(\ol R^\flat)$ with respect to $\ker(\theta)$. The Witt vector Frobenius extends to $\vphi$ on $\Acris^\nabla(R)$. We let $\Fil^1\Acris^\nabla(R)$ denote the  ideal topologically generated by the divided powers of $\xi:=p-[\wt p]$, and $\Fil^r\Acris^\nabla(R)$ its $r$th divided power ideal. Note that $\Acris^\nabla(R)$ only depends on $\ol R$. 

Consider the $R_0$-linear extension $\theta_{R_0}:R_0\otimes_{\Zp}W(\ol R^\flat)\ra \wh{\ol R}$  of $\theta$, and define $\Acris(R)$ to be the $p$-adically completed divided power envelope of $R_0\otimes_{\Zp}W(\ol R^\flat)$ with respect to $\ker (\theta_{R_0})$.  (Note that $\Acris(R)$ depends on the choice of $\ol R$ and $R_0$, which we suppress from the notation.)

As before, $\Acris(R)$ is equipped with a Frobenius endomorphism $\vphi$ which extends $\vphi$ on $R_0$ and the Witt vector Frobenius on $W(\ol R^\flat)$. Let $\Fil^1\Acris(R)$ be the ideal topologically generated by the divided powers of $\ker(\theta_{R_0})$, and  $\Fil^r\Acris(R)$ its $r$th divided power ideal. In addition, we define a connection $\nabla:\Acris(R)\ra\Acris(R)\wh\otimes_{R_0}\wh\Omega_{R_0}$ by $W(\ol R^\flat)$-linearly extending the universal continuous derivation of $R_0$ so that $\nabla(f^{[j]}):=f^{[j-1]}\nabla(f)$ for any $f\in\Fil^1\Acris(R)$ and $j>0$. One can directly check that $\vphi$ and $\gal_R$-action on $\Acris(R)$ is horizontal. Also note that $\Acris^\nabla(R)$ is naturally embedded in $\Acris(R)$, and coincides with $\Acris(R)^{\nabla=0}$, hence the notation.

The elements $[\epsilon]\in W(\ol R^\flat)$  and the formal power series $t=\log[\epsilon]$ can be viewed in $\Acris^\nabla(R)$ and $\Acris(R)$, and have all the expected properties. (Indeed, they all lie in the ``classical period ring'' $\Acris(\Zp)$, constructed by Fontaine \cite{fontaine:Asterisque223ExpII}.) We define $\Bcris^\nabla(R):=\Acris^\nabla(R)[\ivtd t]$ and $\Bcris(R):=\Acris(R)[\ivtd t]$. Note that $p$ is invertible in these rings since $p$ divides $t^{p-1}$. (Note that $p$ divides  $t^{p-1}$ in $\Acris(\Zp)$, which is well-known.) For any $r\in\Z$, we define $\Fil^r\Bcris(R) = \sum_{n\geqs -r}\ivtd{t^{n}}\Fil^{n+r}\Acris(R)[\ivtd p]$ and similarly define $\Fil^r\Bcris^\nabla(R)$.

 \begin{defn}
A \emph{filtered $(\vphi,\nabla)$ module over $R$ (relative to $R_0$)} is defined to be a quadruple $(D, \vphi_D, \nabla_D, \Fil^\bullet D_R)$, where 
\begin{enumerate}
\item $D$ is a finite projective $R_0[\ivtd p]$-module;
\item $\vphi_D:D\ra D$ is a $\vphi$-linear endomorphism such that $1\otimes\vphi_D$ is an isomorphism;
\item $\nabla_D:D\ra D\otimes_{R_0}\wh\Omega_{R_0}$ is an integrable topologically quasi-nilpotent connection (i.e., there exists a $R_0$-lattice $\M_0\subset D$ on which $\nabla_D$ induces a topologically quasi-nilpotent connection);
\item $\Fil^\bullet D_R$ is a decreasing separated exhaustive $R$-filtration on $D_R:=R\otimes_{R_0}D$, which satisfies Griffiths transversality with respect to $\nabla_D$.
\end{enumerate}
We denote by $\MFF_{R/R_0}(\vphi,\nabla)$ the category of filtered $(\vphi,\nabla)$-modules over $R$ relative to $R_0$. For any $D\in\MFF_{R/R_0}(\vphi,\nabla)$, the \emph{Hodge-Tate weights} of $D$ are $w\in\Z$ such that $\gr^wD_R\ne 0$.
\end{defn}
One can naturally define  short exact sequences, direct sums, $\otimes$-products, duals, etc. 

\subsection{Crystalline representations}
For any $p$-adic $\GRR$-representation $V$ and $D\in\MFF_{R/R_0}(\vphi,\nabla)$, we define:
\begin{align} 
\label{eqn:Dcris} \Dcris^*(V)&:=\Hom_{\gal_R}(V,\Bcris(R))\\
\label{eqn:DdR} \DdR^*(V)&:=\Hom_{\gal_R}(V,\BdR(R))\\
\label{eqn:Vcris} \Vcris^*(D)&=\Hom_{R_0[1/p],\vphi, \nabla,\Fil^\bullet}(D,\Bcris(R)).
\end{align}
Note that $\Dcris^*(V)$ is an $R_0[\ivtd p]$-module equipped with a $\vphi$-semilinear endomorphism and a connection coming from $\Bcris(R)$,  $\DdR^*(V)$ is an  $R[\ivtd p]$-module equipped with a  filtration coming from $\BdR(R)$, and $\Vcris^*(D)$ is a $\Qp$-vector space equipped with a continuous $\GRR$-action.

Note that in order to show that these constructions yield finitely generated objects (over some suitable rings), let alone prove some natural properties, one needs to make some extra assumptions on $R$, which we explain now.

\subsection{Remark on Faltings' purity and refined almost \'etaleness}\label{subsec:BrinonSetting}
For the rest of the section, we assume that $R$ satisfies  the ``refined almost \'etaleness'' assumption~(\S\ref{cond:RAF}), which we recall now: 
\begin{description}
\item[Assumption~\S\ref{cond:RAF}]
The (irreducible)  normal extension $R$ of $R_0$ satisfies the formally finite-type assumption~(\S\ref{cond:dJ}), $R[\ivtd p]$ is finite \'etale over $R_0[\ivtd p]$,  and we have $\wh\Omega_{R_0} = \bigoplus_{i=1}^d R_0 dT_i$  for some finitely many units $T_i\in R_0\starr$. 
\end{description}
See \S\ref{cond:RAF} for examples.

This condition is slightly more general than the conditions on $R$ under which the relative crystalline period rings $\Acris^\nabla(R)$ and $\Acris(R)$ are studied in \cite{Brinon:imperfect} and \cite{Brinon:CrisDR} (\emph{cf.,} \cite[\S1, \S3]{Brinon:Habilitation}); the cases covered in the literature are when $R$ is a discrete valuation ring with the residue field $k$ admitting a finite $p$-basis \cite{Brinon:imperfect}, or when the residue field $k$ is perfect and $R$ is obtained by some combination of finite \'etale extension, localisation, and completion starting from $\fo_K[T_1^{\pm1},\cdots,T_d^{\pm1}]$  \cite{Brinon:CrisDR}.

The point of  the ``refined almost \'etaleness'' assumption~(\S\ref{cond:RAF}) is, as the name suggests, the property called  refined almost \'etaleness (i.e., the condition (RAE) in \cite[\S5]{Andreatta:GenNormRings}), which plays the key technical role in ensuring that  relative period rings are well-behaved. Suppose that the assumption \S\ref{cond:RAF} is satisfied, and  let $A$ be a normal $R$-algebra such that $A[\ivtd p]$ is finite \'etale over $R[\ivtd p]$. Set $R_{(n)}:= R[\zeta_{p^n}; T_1^{\pm 1/p^n},\cdots, T_d^{\pm 1/p^n}]$, where $\zeta_{p^n}$ is a primitive $p^n$th root of unity and $T_i$'s are as in \S\ref{cond:RAF}. We similarly define $R_{0,(n)}$. (Note that $R_{(n)}$ is a finite normal extension of $R$, and is defined over $R_0$ in the sense that $R_{(n)} = R_{0,(n)}\otimes_{R_0}R$.) We let $A_{(n)}:= R_{(n)}\otimes_R A$, which is normal $R_{(n)}$-algebra which becomes \'etale after inverting $p$. Therefore, there exists an idempotent $\mathfrak{e}_{A,n}\in (A_{(n)}\otimes_{R_{(n)}}A_{(n)})[\ivtd p]$ which corresponds to the splitting of the natural map $ (A_{(n)}\otimes_{R_{(n)}}A_{(n)})[\ivtd p] \ra A_{(n)}[\ivtd p]$. For $\epsilon\in \ivtd{p-1}\Z[\ivtd p]$, let $p^\epsilon \in \bigcup_n K(\zeta_{p^n})$ denote any element with $\ord_p(p^\epsilon) = \epsilon$, where $\ord_p$ is the valuation normalised so that $\ord_p(p) = 1$.
\begin{thmsub}\label{thm:RAF}
Under the above setting, there exists an integer $l$ (only depending on $A$) such that $p^{lp^{-n}}\mathfrak{e}_{A,n} \in  A_{(n)}\otimes_{R_{(n)}}A_{(n)}$ for any $n$. 
\end{thmsub}
\begin{proof}
This theorem is essentially due to Faltings \cite[Theorem~4, \S2b]{Faltings:AlmostEtaleExt} and Andreatta \cite[Theorem~5.1]{Andreatta:GenNormRings}, and we indicate how to deduce this theorem from the aforementioned results.

It suffices to handle the case when $R=R_0$ by viewing $A$ as a normal extension of $R_0$; indeed, via the natural projection $A_{(n)}\otimes_{R_{0,(n)}}A_{(n)}[\ivtd p] \thra A_{(n)}[\ivtd p]\otimes_{R_{(n)}}A_{(n)}[\ivtd p]$, the claim for an idempotent $\mathfrak{e}_{A,n}\in A_{(n)}\otimes_{R_{0,(n)}}A_{(n)}[\ivtd p] $ implies the claim for its image in $A_{(n)}\otimes_{R_{(n)}}A_{(n)}[\ivtd p] $. From now on, we assume that $R$ is formally finitely generated over a Cohen ring $W$ with residue field $k$, and $R/(p)$ locally admits a finite $p$-basis.

Let us first assume that $k$ is algebraically closed. Choose a $W$-algebra map $\imath:W[X_1^{\pm1},\cdots, X_r^{\pm1}] \ra R_0$ with $r\geqs d$, such that $\imath(X_i) = T_i$ for $i=1,\cdots, d$, and the composition $k [X_1^{\pm1},\cdots, X_r^{\pm1}] \xra{k\otimes\imath} R/(\varpi)\thra R/J_R$ is surjective. This is possible by Bertini's theorem (as $k$ is assumed to be algebraically closed). We set $R'_0:=\varprojlim_s W [X_1^{\pm1},\cdots, X_r^{\pm1}]/(J')^s$ where $J':= \ker(W[X_i^{\pm1}]\thra R/J_R)$, and let $R':=\fo_K\otimes_W R'_0$. 

The map $\imath$ extends to a quotient map $R'\thra R$, and it admits a continuous section since $R$ is formally smooth over $W$. (Recall that we assumed $R=R_0$.) In particular, there exists a finitely generated projective $R$-module $M$ such that $R'\cong R[[M]]:=\prod_{i\geqs0}\Sym^i_{R}M$ by  \cite[Lemma~1.3.3]{dejong:crysdieubyformalrigid}. We set $R'_{(n)}:= R'[\zeta_{p^n}; X_1^{\pm 1/p^n},\cdots, X_r^{\pm 1/p^n}]$, then the same argument as above shows that $R'_{(n)}\cong R_{(n)} [[M_n]]$ for some finitely generated projective $R_{(n)}$-module $M_n$.

Recall that the lemma is known for any finite normal $R'$-algebra $A'$ such that $A'[\ivtd p]$ is \'etale over $R'[\ivtd p]$ by   Faltings'  purity theorem \cite[Theorem~4]{Faltings:AlmostEtaleExt} (\emph{cf.,} \cite[Theorem~5.11]{Andreatta:GenNormRings}) and \cite[Theorem~5.1]{Andreatta:GenNormRings}. Now, for any normal  $R$-algebra $A$ such that $A[\ivtd p]$ is finite \'etale over $R[\ivtd p]$, we can apply the result to $A':=R'\otimes_R A$ and obtain an integer $l$ such that 
\begin{equation*}
p^{lp^{-n}}\mathfrak{e}_{A',n} \in  A'_{(n)}\otimes_{R'_{(n)}}A'_{(n)} \cong R_{(n)}[[M_n]]\otimes_{R_{(n)}}(A_{(n)}\otimes_{R_{(n)}}A_{(n)})
\end{equation*}
 for any $n\geqs0$. The image of $p^{lp^{-n}}\mathfrak{e}_{A',n}$ under the projection $R_{(n)}[[M_n]]\otimes_{R_{(n)}}(A_{(n)}\otimes_{R_{(n)}}A_{(n)}) \thra A_{(n)}\otimes_{R_{(n)}}A_{(n)}$ 
(defined by quotienting out $M_n$) is exactly $p^{lp^{-n}}\mathfrak{e}_{A,n}$, so the lemma follows when $k$ is algebraically closed.

When $k$ is not algebraically closed, we choose a map $W\ra W(\kbar)$, which always exists. Since the lemma is known for $W(\kbar)\otimes_W R$, we can deduce the lemma for $R$ by repeating the argument in the previous paragraph. 
\end{proof}

Let us list a few useful properties of these relative period rings. Although the statement is slightly more general than those found in the literature \cite{Brinon:imperfect,Brinon:CrisDR}, it is not hart to extend it to our setting.
\begin{prop}\label{prop:PeriodRings}
Assume that $R$ satisfies the ``refined almost \'etaleness'' assumption~(\S\ref{cond:RAF}). For $i=1,\cdots,d$, let $u_i\in\Acris(R)\subset\BdR(R)$ denote the image of $T_i-[\wt T_i]$, where $T_i$'s are chosen as in \S\ref{subsec:BrinonSetting}. 
\begin{enumerate}
\item\label{prop:PeriodRings:BdR} We have $\BdR^+(R) = \BdR^{\nabla,+}(R)[[u_1,\cdots,u_d]]$, and the associated graded rings are $\gr^\bullet\BdR^{\nabla,+} \cong \wh{\ol R}[\ivtd p][t]$ and $\gr^\bullet\BdR^+(R) \cong \wh{\ol{R}}[\ivtd p][t,u_1,\cdots,u_d]$.

\item\label{prop:PeriodRings:Griff}
The $\BdR^\nabla(R)$-linear connection defined by $\nabla(u_i) = 1\otimes dT_i$ satisfies the Griffiths transversality (i.e., $\nabla(\Fil^r\BdR(R)) \subseteq \Fil^{r-1}\BdR(R)\otimes_{R_0}\wh\Omega_{R_0}$).

\item\label{prop:PeriodRings:Acris} We have $\Acris(R) = \Acris^\nabla(R)\langle u_1,\cdots,u_d\rangle^{PD}$; i.e., the $p$-adically completed divided power polynomial in $u_i$'s over $\Acris^\nabla(R)$. 

\item\label{prop:PeriodRings:AcrisFil} The ideal $\Fil^r\Acris(R)$ is topologically generated by $(p-[\tilde p])^{[j_0]}u_1^{[j_1]}\cdots u_d^{[j_d]}$ for $\sum_{n=0}^d j_n\geqs r$, and we have $\Fil^r\Acris^\nabla(R) = \Fil^r\Acris(R)\cap \Acris^\nabla(R)$. In particular, $t\in\Fil^1\Acris^\nabla(R)\subset \Fil^1\Acris(R)$. 

\item\label{prop:PeriodRings:nabla} 
The connection $\nabla$ on $\Acris(R)$ is the unique $\Acris^\nabla(R)$-linear connection such that $\nabla(u_i^{[n]}) = u_i^{[n-1]}\otimes dT_i$, and $\vphi:\Acris(R)\ra\Acris(R)$ is horizontal.

\item\label{prop:PeriodRings:flatness}  Both $\BdR^+(R)$  and $\Acris(R)$ have no non-zero $t$-torsion. (In particular, $\Acris(R)$ has no non-zero $p$-torsion.) The rings $\wh{\ol{R}}[\ivtd p]$, $\BdR^+(R)$ and $\BdR(R)$ are faithfully flat over $R[\ivtd p]$, and $\Bcris(R)$ is faithfully flat over $R_0[\ivtd p]$.

\item\label{prop:PeriodRings:Inj} The natural map $R_0\otimes_{\Zp} W(\ol R^\flat) \hra R\otimes_{\Zp} W(\ol R^\flat)$ uniquely extends to $\Acris(R)\ra\BdR(R)$, which is injective, filtered, horizontal, and $\gal_R$-equivariant. Furthermore,  the natural map $R\otimes_{R_0}\Acris(R) \ra \BdR(R)$ is injective.  

\item\label{prop:PeriodRings:GalCohom} As an $R[\ivtd p]$-algebra $\BdR^+(R)$ contains $\ol R[\ivtd p]$ as a $\GRR$-stable subring, and  we have $\BdR(R)^{\gal_R} = R[\ivtd p]$. Let $\wh R_0^{\ur}$ denote the closure of maximal ind-\'etale $R_0$-subalgebra in $\wh{\ol R}$. Then $\Acris(R)$ contains $\wh R_0^{\ur}$ as a $\GRR$-stable subring, and $\Bcris(R)^{\gal_R} = R_0[\ivtd p]$.

\item\label{prop:PeriodRings:FundSeq} 
The  sequence 
$0\ra\Qp\ra\Bcris^\nabla(R)^{\vphi = 1} \ra\BdR^{\nabla}(R)/\BdR^{\nabla,+}(R) \ra 0$ is exact.
(This sequence is called the \emph{fundamental exact sequence}).
\end{enumerate}
\end{prop}
\begin{proof}
The proofs in \cite{Brinon:imperfect} and \cite{Brinon:CrisDR} work in our setting (since we have all the ingredients for the proof; especially,  Theorem~\ref{thm:RAF}), so we content with giving references for the proofs.

For (\ref{prop:PeriodRings:BdR}) and (\ref{prop:PeriodRings:Acris}), the same proof for the case when $R$ is a discrete valuation ring carries over if we work with $R_0$ instead of $\cO_{K_0}$  (\emph{cf.,} Propositions~2.9, 2.19, and 2.39 in \cite{Brinon:imperfect}); note that $R_0$ is formally smooth over $\Zp$, which suffices for the proof to work. A direct computation using (\ref{prop:PeriodRings:BdR}) shows (\ref{prop:PeriodRings:Griff}). (\emph{Cf.} \cite[Proposition~2.23]{Brinon:imperfect}, \cite[Proposition~5.3.9]{Brinon:CrisDR}.)

The statement (\ref{prop:PeriodRings:AcrisFil}) follows from (\ref{prop:PeriodRings:Acris}) and \cite[Proposition~5.1.2]{Brinon:CrisDR}, which asserts that the kernel of $\theta:W(\ol R^\flat)\thra\wh{\ol R}$ is principally generated by $p-[\wt p]$. The statement (\ref{prop:PeriodRings:nabla}) follows from a direct computation using (\ref{prop:PeriodRings:Acris})--(\ref{prop:PeriodRings:AcrisFil}) (\emph{cf.,} Proposition 6.2.5 of  \cite{Brinon:CrisDR}).

The $p$- and $t$- torsion statement in (\ref{prop:PeriodRings:flatness})  follows from  the  proofs of  Propositions~5.1.5 and 6.1.10 of \cite{Brinon:CrisDR}.
The faithful flatness statement in (\ref{prop:PeriodRings:flatness}) follows from the proofs of Th\'eor\`emes~3.2.3, 5.4.1, and  6.3.8 of \cite{Brinon:CrisDR}, which uses  refined almost \'etaleness (Theorem~\ref{thm:RAF}). 
The proofs of Propositions~6.2.1 and 6.2.7 of \cite{Brinon:CrisDR} show the injectivity statements of (\ref{prop:PeriodRings:Inj}) while the rest of (\ref{prop:PeriodRings:Inj})  can be directly checked.

To see that there is a natural embedding $\wh R_0^{\ur}\hra \Acris(R)$, note that  $\theta_{R_0}$ induces a nilpotent thickening $\Acris(R)/(p^n) \thra \ol R/(p^n)$, so by ind-\'etaleness there is a unique map $ R_0^{\ur}\ra \Acris(R)/(p^n)$ lifting the natural map $ R_0^{\ur}\ra \ol R/(p^n)$ for any $n$. The same argument shows the embedding $\ol R[\ivtd p]\hra \BdR^+(R)$ (\emph{cf.,}  \cite[Proposition~5.2.3]{Brinon:CrisDR}). 

The proofs of Propositions~5.2.12 and 6.2.9 of  \cite{Brinon:CrisDR} shows (\ref{prop:PeriodRings:GalCohom}). 
For (\ref{prop:PeriodRings:FundSeq}) see the proof of  \cite[Proposition~6.2.23]{Brinon:CrisDR}.
\end{proof}

\subsection{(Relative) crystalline $\GRR$-representations}
Assume   that  $R$ satisfies the ``refined almost \'etaleness'' assumption~(\S\ref{cond:RAF}).
One can deduce (rather formally from Proposition~\ref{prop:PeriodRings}) the following properties for  $\Dcris$, $\DdR$, and $\Vcris$ (\ref{eqn:Dcris}--\ref{eqn:Vcris}) in the same manner as  \cite[\S8]{Brinon:CrisDR}:
\begin{enumerate}
\item For any $p$-adic $\GRR$-representation $V$, the following natural maps (which respects all the structure)
\begin{align*}
\alpha_{\cris}&:\Bcris(R)  \otimes_{R_0[\ivtd p]}\Dcris^*(V) \ra \Bcris(R)\otimes_{\Qp}V\\
\alpha_{\dR}&:\BdR(R)  \otimes_{R[\ivtd p]}\DdR^*(V) \ra \BdR(R)\otimes_{\Qp}V
\end{align*}
are injective (\emph{cf.,}  \cite[Propositions~8.2.4 and 8.2.6]{Brinon:CrisDR}). If  $\alpha_{\cris}$ is an isomorphism then we say that $V$ is \emph{crystalline}. In this case, $\Dcris^*(V)$ is finitely generated projective over $R_0[\ivtd p]$ by Proposition~\ref{prop:PeriodRings}(\ref{prop:PeriodRings:flatness}), and the natural map $R\otimes_{R_0}\Dcris^*(V) \ra \DdR^*(V)$ is an isomorphism by Proposition~\ref{prop:PeriodRings}(\ref{prop:PeriodRings:Inj}). 
In particular, we may naturally view $\Dcris^*(V)\in\RMF$ (\emph{cf.,} \cite[\S8.3]{Brinon:CrisDR}). We call $D\in\RMF$  \emph{admissible} if there exists $V$ such that $D\cong\Dcris^*(V)$.
\item Let $\Repcris_{\Qp}(\gal_R)$ denote the category of crystalline $\gal_R$-representations, and let $\RMFa\subset \RMF$ denote the full subcategory of admissible objects. Then $\Dcris^*$ and $\Vcris^*$ are quasi-inverse anti-equivalences of Tannakian categories between $\Repcris_{\Qp}(\gal_R)$ and $\RMFa$ (\emph{cf.,} Th\'eor\`emes 8.4.2 and 8.5.1 of \cite{Brinon:CrisDR}).
\end{enumerate}

\begin{exasub}[Case of $p$-divisible groups]\label{exa:BT}
For any $p$-divisible group $R$, one can define $D^*(G):=(\DD^*(G)(R_0)[\ivtd p], \Fil^1\DD^*(G)(R)[\ivtd p])$, and this is clearly a  filtered $(\vphi,\nabla)$-module with Hodge-Tate weights in $\set{0,1}$. (Note that Griffiths transversality imposes no condition.) We will show later (Corollary~\ref{cor:BrinonTrihan}) that under the ``refined almost \'etaleness'' assumption~(\S\ref{cond:RAF}) we have $\Vcris^*(D^*(G)) \cong V_p(G)$ (i.e., $V_p(G)$ is crystalline and $D^*(G)$ is admissible). We deduce this from a finer statement about the integral lattice $T_p(G)\subset V_p(G)$ and $\DD^*(G)$.
\end{exasub}

\section{Relative integral $p$-adic comparison isomorphism }\label{sec:Faltings}
In this section, we prove the integral $p$-adic comparison theorem for $p$-divisible groups over $R$, which directly follows from the proof of Faltings \cite[\S6]{Faltings:IntegralCrysCohoVeryRamBase}. A similar approach to ours can be found in  \cite{BrinonTrihan:CrisImperf} over a $p$-adic discrete valuation ring with residue field admitting a finite $p$-basis, and perhaps the main result of this section was already well known to experts, but we include this section for completeness.

We continue to assume that $R$ is a normal domain which satisfies the $p$-basis assumption~(\S\ref{cond:Breuil}), and will specify when we make a stronger assumption on $R$ (namely, the ``refined almost \'etaleness'' assumption~(\S\ref{cond:RAF})). 
\subsection{The rings $S$ and $\Acris(R)$}\label{subsec:Rinfty}
Choose $\varpi^{(n)}\in\ol R$ for $n\geqs 0$ so that $\varpi^{(0)} = \varpi$ and $(\varpi^{(n+1)})^p=\varpi^{(n)}$. Note that $\wt\varpi:=(\varpi^{(n)}) $ defines an element in $\ol R^\flat$. Set $R_\infty:=\bigcup_n R[\varpi^{(n)}] \subset \ol R$, and $\GRinfty:=\Gal(\ol R[\ivtd p]/R_\infty[\ivtd p])$.

We define an $R_0$-algebra map $\Sig \ra R_0\otimes_{\Zp} W(\ol R^\flat)$ by sending $u\mapsto [\wt\varpi]$, and this naturally extends to a $\GRinfty$-invariant (but \emph{not} $\GRR$-invariant) injective map $S\ra\Acris(R)$ which respects $\vphi$, $\nabla$, filtrations, and divided power structure. 

\subsection{Integral comparison isomorphism}
Recall that for a $p$-divisible group $G$ over $R$, the (contravariant) Dieudonn\'e crystal $\DD^*(G)$ is obtained from the Lie algebra of the universal vector extension for (lifts of) $G^\vee$ (\emph{cf.,} \cite[Ch.VI]{messingthesis}). Using this definition, one can easily obtain $\DD^*(\Qp/\Zp)(S)  \cong (S,\Fil^1S, \vphi/p, d_{S})$  and $\DD^*(\Gmhat)(S)  \cong(S,S, \vphi, d_{S})$.

As in the proof of Theorem~\ref{thm:BreuilClassif}, let us denote $\M^*(G) := \DD^*(G)(S)$ for any $p$-divisible group $G$ over $R$. Then we have $\DD^*(G_{\ol R})(\Acris(R)) \cong \Acris(R)\otimes_{S} \M^*(G)$ which respects all the extra structures, except the $\GRR$-action. (Indeed the isomorphism is only $\GRinfty$-equivariant as $S\subset\Acris(R)$ is not $\GRR$-invariant but $\GRinfty$-invariant.) Let us write $\AMF{R}$ for the category of $\Acris(R)$-module equipped with $\Fil^1$, $\vphi_1$ and $\nabla$ in the exactly same way as $\SMFw$. 

As $\GRR$-modules we may naturally identify $T_p(G)\cong \Hom_{\ol R}(\Qp/\Zp, G_{\ol R})$. So we obtain a pairing
\begin{equation*}
T_p(G)\times \left(\Acris(R)\otimes_{ S}\M^*(G) \right)\ra \Acris(R)\otimes_{S}\M^*(\Qp/\Zp) = \Acris(R)
\end{equation*}
by $(x,m)\mapsto x^*m$ for any $x:\Qp/\Zp \ra G_{\ol R}$ and $m\in \DD^*(G_{\ol R})(\Acris(R))$.
Therefore we obtain the following integral comparison morphism
\begin{equation}\label{eqn:Faltings}
\rho_G:\Acris(R)\otimes_{S}\M^*(G) \ra \Acris(R)\otimes_{\Zp}T_p(G)^*.
\end{equation}
With the naturally defined extra structures on the both sides, $\rho_G$ can be naturally viewed as a morphism in $\AMF{R}$. 

\begin{rmksub}\label{rmk:rhoG}
There is a natural  $\GRR$-action on $\DD^*(G_{\ol R})(\Acris(R))$ induced from the natural $\GRR$-action on $\Acris(R)$, so we have a natural $\GRR$-action on $\Acris(R)\otimes_{S}\M^*(G)$. With this $\GRR$-action on the target, $\rho_G$ is $\GRR$-equivariant. Note that the $\GRR$-action on $\Acris(R)\otimes_{S}\M^*(G)$ does \emph{not} fix $\M^*(G)$; indeed, only $\GRinfty$ fixes $\M^*(G)$.
\end{rmksub}

\begin{thm}\label{thm:Faltings}
The map $\rho_G$ is injective with cokernel annihilated by $t$. 
In particular, each of the morphisms below
\begin{equation*}
\Bcris(R)\otimes_{R_0[\frac{1}{p}]}D^*(G) \ra \Bcris(R)\otimes_{S}\M^*(G) \xra{\rho_G[\frac{1}{p}]}  \Bcris(R)\otimes_{\Qp} V_p(G)^*
\end{equation*}
is an isomorphism, where $D^*(G)$ is defined in Example~\ref{exa:BT} and the first map is induced by the unique section $D^*(G)\ra \M^*(G) [\ivtd p]$ as in Lemma~\ref{lem:Reduction}. Furthermore, the composition $\Bcris(R)\otimes_{R_0[\frac{1}{p}]}D^*(G) \riso  \Bcris(R)\otimes_{\Qp} V_p(G)^*$ is $\GRR$-equivariant.
\end{thm}
\begin{proof}
The theorem follows if we show that $\rho_G$ (\ref{eqn:Faltings}) is injective with cokernel killed by $t$; indeed, the last assertion on $\GRR$-equivariance follows from the $\GRR$-equivariance of $\rho_G$ (Remark~\ref{rmk:rhoG}) and the $\GRR$-invariance of the section $s_{\Acris(R)}:D^*(G)\ra \Acris(R)\otimes_{S}\M^*(G)$.

The proof is exactly the same as \cite[\S6]{Faltings:IntegralCrysCohoVeryRamBase}. Let us first make $\rho_{\Gmhat}$ explicit when $G=\Gmhat$. Let $\beta:\Zp(1)\ra\Fil^1\Acris(R)$ be the map that sends $\eps=(\eps^{(n)})\in \varprojlim_n\mu_{p^n}(\ol R)\subset \ol R^\flat$ to $\log[\eps]$. Then, one can verify that the morphism
\begin{equation*}
\rho_{\Gmhat}:(\Acris(R), \Acris(R), \vphi, \nabla) \ra \Hom_{\Zp}(\Zp(1),\Acris(R))
\end{equation*} 
sends $1$ to $ \beta$, as explained in \cite[\S6]{Faltings:IntegralCrysCohoVeryRamBase}. (One way to see this is  by applying \cite[Ch. VI, Theorem~2.2]{messingthesis} to the sections over the PD completion of $\Acris(R)$. See \cite[Ch. VI, \S2.5]{messingthesis} for the construction of the morphism of Dieudonn\'e crystals corresponding to a morphism of $p$-divisible groups.) Therefore if we naturally identify  $\Hom_{\Zp}(\Z_1(1),\Acris(R))$ with $t\iv\Acris(R)$, then $\rho_{\Gmhat}$ can be identified with the natural inclusion $\Acris(R)\hra t\iv\Acris(R)$. This shows that $\rho_{\Gmhat}$ is injective and and its cokernel is killed by $t$. 

Let us now handle the general case. 
For any $y\in \Hom_{\ol R}(G_{\ol R},\wh{\mathbb G}_{m,\ol R})$ , one can check that the following  diagram commutes:
\begin{equation}\label{eqn:FaltingsBrinonTrihan}
\xymatrix{
\Acris(R)\otimes_{S}\M^*(G) \ar[r]^{\rho_G} & \Acris(R)\otimes_{\Zp}T_p(G)^* \\
\Acris(R)\otimes_{S}\M^*(\Gmhat) \ar[u]^{y^*} \ar[r]^{\rho_{\Gmhat}}& \Acris(R)\otimes_{\Zp} T_p(\Gmhat)^* \ar[u]_{\id_{\Acris(R)}\otimes T_p(y)^*}
}\end{equation}

Recall that we have a natural $\GRR$-equivariant isomorphism $\Hom_{\ol R}(G_{\ol R}, \wh{\mathbb G}_{m,\ol R})\riso T_p(G)^*(1)$ defined by sending $y:G_{\ol R}\ra \wh{\mathbb G}_{m,\ol R}$ to $x\mapsto y\circ x$ for any $x\in T_p(G)\cong \Hom_{\ol R}(\Qp/\Zp, G_{\ol R})$.
Now choose a $\Zp$-basis $\eps \in T_p(\Gmhat)$ so that $t=\beta(\eps)$, and let $\eta\in T_p(G)^*$ be such that $y$ corresponds to $\eta\otimes\eps\in T_p(G)^*(1)$ under the natural isomorphism (i.e., for any $x\in T_p(G)$, we have $y\circ x = \eta(x)\eps$). Recall that $\rho_{\Gmhat}(1\otimes1) = \beta$, and one can compute 
\begin{equation*}
\left(\id\otimes T_p(y)^*\right)(\beta):x \mapsto \beta(y\circ x) = \beta(\eps)\eta(x) = t\eta(x);
\end{equation*}
i.e., $\left(\id\otimes T_p(y)^*\right)(\beta) = t\otimes\eta$. By the commutative diagram (\ref{eqn:FaltingsBrinonTrihan}), it follows that $t\otimes\eta$ is in the image of $\rho_G$. Since $\eta\in T_p(G)^*$ can be arbitrary as we vary $y$, the theorem follows.
\end{proof}

\subsection{Galois-stable lattices}
If $R$ satisfies the ``refined almost \'etaleness'' assumption~(\S\ref{cond:RAF}) (so that we have refined almost \'etaleness and the period rings have nice properties: Proposition~\ref{prop:PeriodRings}) then one can define $\Tcris^*(\M)$ for any  $\M\in \SMFw$ as follows:
\begin{equation}\label{eqn:Tcris}
\Tcris^*(\M):= \Hom_{S,\Fil^1, \vphi^1,\nabla}(\M, \Acris(R)),
\end{equation}
where $\Acris(R)$ is viewed as an $S$-algebra  as in \S\ref{subsec:Rinfty}, and is given $\Fil^1\Acris(R)$, $\vphi_1:=\frac{\vphi}{p}:\Fil^1\Acris(R)\ra\Acris(R)$, and $\nabla:\Acris(R)\ra\Acris(R)\otimes_{R_0}\wh\Omega_{R_0}$. 

Clearly, $\Tcris^*(\M)$ is $p$-adic, and has a natural continuous $\GRinfty$-action induced from the $\GRinfty$-action on $\Acris(R)$. It is not \emph{a priori} obvious if  $\Tcris^*(\M)$ is finite free over $\Zp$, but this follows from Corollary~\ref{cor:BrinonTrihan} below. (Indeed, we also show that $\rank_{\Zp}\Tcris^*(\M) = \rank_{S}\M$.) 

Note that the map $S\ra\Acris(R)$ is only $\GRinfty$-invariant, so we only obtain $\GRinfty$-action on $\Tcris^*(\M)$. Using the differential operator $N_\M$, however, one can define a $\GRR$-action on $\Tcris^*(\M)$ that extends its natural $\GRinfty$-action. See \S\ref{subsec:GalViaN} for more details.

\begin{corsub}\label{cor:BrinonTrihan}
Suppose that $R$ satisfies the ``refined almost \'etaleness'' assumption~(\S\ref{cond:RAF}). Then  $\rho_G$ (as in (\ref{eqn:Faltings})) induces 
a $\GRinfty$-equivariant injective morphism 
\begin{equation*}
T_p(G) \ra \Hom_{\AMF{R}}(\Acris(R)\otimes_{S}\M^*(G), \Acris(R))\cong T^*_{\cris}(\nf \M^*(G)), 
\end{equation*}
which is isomorphism if $p>2$, and has cokernel annihilated by $p$ if $p=2$. In particular, it induces an isomorphism $D^*(G)\riso D^*_{\cris}(V_p(G))$ as filtered isocrystals.
\end{corsub}
\begin{proof} 
It is straightforward from Theorem~\ref{thm:Faltings}, noting that $p$ divides $t$ if and only if $p=2$, in which case $p=2$ divides $t$ exactly once. 
\end{proof}
This corollary in particular proves that when $R$ satisfies the ``refined almost \'etaleness'' assumption~(\S\ref{cond:RAF}),   the $\GRR$-representation $V_p(G)$ is crystalline and $D^*(G)$ is admissible for any $p$-divisible group $G$ over $R$.

\begin{rmksub}\label{rmk:p2} 
It follows from the proof of Theorem~\ref{thm:Faltings} that $T_p(\Gmhat) = 2\tim\Tcris^*(\Gmhat)$ in $V_p(\Gmhat)$ (if $p=2$). 
\end{rmksub}

\subsection{Galois Action on the $\Zp$-lattice}\label{subsec:GalViaN}
Suppose that $R$ satisfies the ``refined almost \'etaleness'' assumption~(\S\ref{cond:RAF}). It follows from Theorem~\ref{thm:Faltings} that  $T^*_{\cris}(\M^*(G))$ can be viewed as a $\GRinfty$-stable $\Zp$-lattice in $V_p(G)$, and is $\GRR$-stable if $p>2$.  
In this section, we define a natural $\GRR$-action directly on $T^*_{\cris}(\M^*(G))$ for any $p$, in such a way that the natural $\GRR$-equivariant map  $T^*_{\cris}(\M^*(G))\hra V_p(G)$ is $\GRR$-stable. We generalise the construction in  \cite[\S5]{Liu:StronglyDivLattice} (\emph{cf.,} \cite[\S2.2]{Breuil:IntegralPAdicHodgeThy}).

Let us fix some notation. For any $n\geqs 0$ define a  cocycle $\epsilon^{(n)}:\GRR\ra \wh{\ol R}\starr$  as follows:
\begin{equation}
\epsilon^{(n)}(g):=  g\tim \varpi^{(n)}/\varpi^{(n)}, \textrm{ for any }g\in\GRR.  
\end{equation}
Set $\epsilon(g):=(\epsilon^{(n)}(g)) \in \ol R^\flat$, and  $t_g:=\log[\epsilon(g)] \in \Acris^\nabla(R)$. Note that for any $g\in\GRR$, $t_g$ is  a $\Zp$-multiple of $t\in\Fil^1\Acris(R)$ (where $t$ is  as in \S\ref{subsec:Acris}), and $t_g=0$ if and only if $g\in\GRinfty$.

For $\M\in \SMF$, let us  modify the $\GRR$-action on $\Acris(R)\otimes_{ S}\M$ using the differential operator $N_{\M}$ as follows:
\begin{subequations}\label{eqn:TstGK}
\begin{equation} \label{eqn:TstGK:a}
g\tim (a\otimes x) := g(a)\sum_{i=0}^\infty (t_g)^{[i]}\otimes N_{\M}^i(x), 
\end{equation}
for $g\in\GRR$, $a\in\Acris(R)$, and $ x\in\M$. Here, $(t_g)^{[i]}$ is the standard $i$th divided power; i.e., $(t_g)^{[i]}:=t_g^i/i!$ if $i>0$ and $(t_g)^{[0]}:=1$ (even when $t_g=0$).  

To see that the  sum (\ref{eqn:TstGK:a}) converges, since $N_\M(\M)\subset u\M$, it suffices to show that $(t_g)^{[i]}\to 0$ as $i\to\infty$. But this follows from  \cite[\S5.2.4]{fontaine:Asterisque223ExpII} as  $(t_g)^{[i]}\in\Acris(\Zp)$.
By the proof of \cite[Lemma~5.1.1]{Liu:StronglyDivLattice}, equation (\ref{eqn:TstGK:a}) gives a $\GRR$-action which respects $\vphi$ and the natural filtration on $\Acris(R)\otimes_{S}\M$. When $g\in\GRinfty$ we recover the natural $\GRinfty$-action on $\Acris(R)\otimes_{S}\M$.

For any $f\in T^*_{\cris}(\M)$, we $\Acris(R)$-linearly extend $f$ to $\Acris(R)\otimes_{S}\M \ra \Acris(R)$. Now the following formula clearly defines a continuous action of $g\in\GRR$ on $ T^*_{\cris}(\M)$:
\begin{equation} \label{eqn:TstGK:b}
g\tim f: x\mapsto g\tim (f(g\iv(1\otimes x))), \text{ for }x\in\M.
\end{equation}
\end{subequations}

If $\M:=\DD^*(G)(S)$ for some $p$-divisible group $G$ over $R$, then the image of the natural map $D^*(G) \hra \M[\ivtd p]$ lies in the kernel of $N_{\M}$. It then follows from Corollary~\ref{cor:BrinonTrihan} that the natural injective $\GRinfty$-map $T^*_{\cris} (\M^*(G)) \ra V_p(G)$ is indeed $\GRR$-equivariant for the $\GRR$-action on $T^*_{\cris} (\M^*(G)) $ defined as in (\ref{eqn:TstGK:a}, \ref{eqn:TstGK:b}). In particular,  the natural injective map $\rho_G:T_p(G)\hra T^*_{\cris}(\M)$ is $\GRR$-equivariant, and $\rho_G$ is a $\GRR$-isomorphism when $p>2$.

\subsection{Base change}\label{subsec:BCrhoG}
Let $R$ and $R'$ be normal domains  which satisfy the $p$-basis assumption~(\S\ref{cond:Breuil}), and consider a map
 $f:R\ra R'$ which  restricts to a $\vphi$-compatible map $R_0 \ra R'_0$ for some suitable choices. (\emph{Cf.}  \S\ref{subsec:BC}.) We also let $f:S\ra S'$ denote the map extending $f|_{R_0}$ by sending $u\mapsto u$.
 
Choose a separable closure $E'$ of $\Frac R'$ and define $\ol{R'}$ to be the union of normal finite $R'$-subalgebras of $E$ which is only ramified at $(\varpi)$, as in \S\ref{subsec:Acris}. Set $R'_\infty:=\bigcup_n R'[\varpi^{(n)}] $,  $\gal_{R'} = \Gal(\ol{R'}[\ivtd p]/R'[\ivtd p])$, and $\gal_{R'_\infty}:=\Gal(\ol{R'}[\ivtd p]/R'_\infty[\ivtd p])$. 
Let $\ol R'$, $R'_\infty$,  $\gal_{R'} $, and $\gal_{R'_\infty}$ denote the obvious objects for $R'$.
Choose $\ol f:\ol R\ra \ol{R'}$ over $f:R\ra R'$ (which is possible, and the choice is unique up to the actions by $\GRR$ and $\gal_{R'}$), and consider the map   $\gal_{R'}\ra\GRR$ of profinite groups induced by it. Under these choices, we obtain a map $\Acris(\ol f):\Acris(R)\ra\Acris(R')$  respecting all the extra structures. (In particular, $\Acris(\ol f)$ is $\gal_{R'}$-equivariant if we let $\gal_{R'}$ act on $\Acris(R)$ via the map $\gal_{R'}\ra\GRR$.)

Now one can easily see that the formation of $\rho_G$ commutes with the base change which satisfies the above assumption; in other words, for any $p$-divisible group $G$ over $R$, we have the following cartesian diagram
\begin{equation}
\xymatrix{
\Acris(R)\otimes_{S}\M^*(G)  \ar[r]^-{\rho_G} \ar[d] & \Acris(R)\otimes_{\Zp}T_p(G)^*\ar[d]^{\Acris(\ol f)\otimes1}\\
\Acris(R')\otimes_{S'}\M^*(G_{R'}) \ar[r]_-{\rho_{G_{R'}}} &\Acris(R')\otimes_{\Zp}T_p(G_{R'})^*,
}\end{equation}
where the left vertical arrow is induced by the map $\Acris(\ol f)$ and the isomorphism $S'\otimes_{S}\M^*(G)\cong\M^*(G_{R'})$ constructed in \S\ref{subsec:BC}. If both $R$ and $R'$ satisfy the ``refined almost \'etaleness'' assumption~(\S\ref{cond:RAF}), from the left vertical arrow we obtain a  $\gal_{R'_\infty}$-isomorphism $\Tcris^*(\M^*(G))\riso \Tcris^*(\M^*(G_{R'})) $, which is an isomorphism because it has a saturated image and  $\Tcris^*(\M^*(G_{R'}))$ and $\Tcris^*(\M^*(G))$ have the same $\Zp$-rank (by the diagram above). If $p>2$, this recovers the natural identification $T_p(G)\cong T_p(G_{R'})$.

\section{Kisin modules: equivalence of categories}\label{sec:CL}
The notion of Kisin modules (i.e., $\Sig$-modules of height $\leqs1$) was generalised to the case when the base is a complete regular local base with perfect residue field by Vasiu and Zink \cite{ZinkVasiu:BreuiloverRegularLocal} and Lau \cite{Lau:Frames,Lau:2010fk}, and they also constructed a natural equivalence of categories between Kisin modules and $p$-divisible groups using display theory.

We generalise the notion of Kisin modules so that it can be applied to some non-local base (\emph{cf.} Definition~\ref{def:KisMod}), and construct a natural functor from the category of $p$-divisible groups into the category of Kisin modules. The main result of this section is the construction of natural equivalence of categories between Kisin modules and Breuil modules, generalising \cite[Theorem~2.2.1]{Caruso-Liu:qst}. Combining this with Theorem~\ref{thm:BreuilClassif}, we obtain a classification theorem of $p$-divisible group when $p>2$ (Corollary~\ref{cor:RelKisin}), which works over not necessarily local base rings, but is weaker in the intersecting case (even when $p>2$). 

In this section, we  assume that $R$ satisfies the $p$-basis assumption~(\S\ref{cond:Breuil}). 
Readers are welcome to work under the normality assumption~(\S\ref{cond:BM}) though, since in the application to $p$-divisible groups and finite locally free group schemes it suffices to consider the base rings $R$ satisfying  the normality assumption~(\S\ref{cond:BM}). 

\subsection{Definitions and basic properties}\label{subset:settingBK}  
\begin{defnsub}\label{def:KisMod}
Let $R$ satisfies the $p$-basis assumption~(\S\ref{cond:Breuil}), and we use the notation from \S\ref{cond:Breuil} and \S\ref{subsec:settingBr}, such as $(\Sig, \vphi)$ and $\PP$.
A \emph{quasi-Kisin $\Sig$-module} is a pair $(\gM, \vphi_\gM)$, where
\begin{enumerate}
\item $\gM$ is a finitely generated projective $\Sig$-module;
\item $\vphi_\gM:\gM\ra\gM$ is a $\vphi$-linear map such that $\coker(1\otimes\vphi_\gM)$ is killed by $\PP$.
\end{enumerate}
Let $\SMq$  denote the category of  quasi-Kisin $\Sig$-modules.

Let $\SM$ denote the category of pairs $(\gM, \vphi_\gM)$, where $\gM$ is a  quasi-Kisin $\Sig$-module, and  $\nabla_{\M}:\M\ra\M\otimes_{\Sig}\wh\Omega_{\Sig}$ on $\M:=S\otimes_{\vphi,\Sig}\gM$  is  a topologically quasi-nilpotent integrable connection  which commutes with $\vphi_{\M} := \vphi_{S}\otimes\vphi_\gM$. 

Let $\SMqw$ denote the category of pairs  $(\gM,\nabla_{\M_0})$ where $\gM\in\SMq$ and $\nabla_{\M_0}$ is a connection on $\M_0:=R_0\otimes_{\vphi,\Sig}\gM$ which makes $\M_0$ into an object in $\PMF{R_0}$. We call an object $(\gM,\nabla_{\M_0})\in\SMqw$ a \emph{Kisin $\Sig$-module}.
\end{defnsub}
\begin{rmksub}\label{rmk:KisModSm}
The recipe in Remark~\ref{rmk:BrModAlt} defines a fully faithful functor $\SM\ra\SMqw$  by Lemma~\ref{lem:forgetfulN}, which is an equivalence of categories when $R$ satisfies the formally finite-type assumption~(\S\ref{cond:dJ}) and $p>2$, by Theorem~\ref{thm:BreuilClassif}. 
\end{rmksub}

Generalising the construction in \cite{breuil:GpSchNormField}, for any $\gM\in\SMq$ we can make $\M:=S\otimes_{\vphi,\Sig}\gM$ into an object in $\SMFq$ as follows:
\begin{gather}
\label{eqn:Filh}\Fil^1\M:=\set{x\in\M|\ 1\otimes\vphi_\gM(x) \in (\Fil^1S)\otimes_{\Sig} \gM \subset S\otimes_{\Sig} \gM } \\
\label{eqn:vphir}\vphi_1:\Fil^1\M \xra{1\otimes \vphi_\gM} (\Fil^1S)\otimes_{\Sig} \gM \xra{\vphi_1\otimes1} S\otimes_{\vphi,\Sig}\gM = \M.
\end{gather}
One can directly check that the above construction satisfies the definition of $\SMFq$. We set $\vphi_\M = \vphi_{S}\otimes\vphi_\gM$, and then $\vphi_1 = \vphi_\M/p$ defines a map $\Fil^1\M\ra \M$. Furthermore, if $\gM\in\SM$  then $S\otimes_{\vphi,\Sig}\gM$ is an object in $\SMF$.


\begin{lemsub}
Let $\gM^\bullet:=[0\ra \gM_1 \ra \gM_2 \ra \gM_3 \ra 0]$ be a sequence of maps of finite projective $\Sig$-modules. Then $\gM^\bullet $ is  exact if and only if $S\otimes_{\vphi,\Sig}\gM^\bullet$ is  exact.
\end{lemsub}
\begin{proof}
The ``only if'' direction is clear from $\Sig$-flatness of $\gM_3$. Assume that  $S\otimes_{\vphi,\Sig}\gM^\bullet$ is  exact. Since $\gM_i$ are $\Sig$-projective the natural maps $\gM_i\ra S\otimes_{\vphi,\Sig}\gM_i$ are injective for any $i=1,2,3$, so $\gM^\bullet$ is left exact. By Nakayama lemma and faithful flatness of $\vphi:R_0\ra R_0$, it suffices to show that $R_0\otimes_{\vphi,\Sig}\gM^\bullet$ is right exact, which follows since $R_0\otimes_{\vphi,\Sig}\gM^\bullet\cong R_0\otimes_{S}(S\otimes_{\vphi,\Sig}\gM^\bullet)$.
\end{proof}

\begin{defnsub}\label{def:psi}
Let $\gM\in\SMq$.
Since $1\otimes\vphi_\gM:\vphi^*\gM\ra\gM$ is injective, we have a unique injective $\Sig$-linear map 
\[\psi_\gM:\gM\ra\vphi^*\gM\] 
such that $(1\times\vphi_\gM)\circ\psi_\M = \PP\id_\gM$ and $\psi_\gM\otimes(1\otimes\vphi_\gM) = \PP\id_{\vphi^*\gM}$. 

We say that $\gM\in\SMq$ is $\vphi$-nilpotent if for some $n\gg1$ we have $\vphi_\gM^n(\gM)\subset (p,u)\gM$. We extend this definition to $\SM$ and $\SMqw$, and use the superscript $^{\vphi\nilp}$ for the full subcategories of $\vphi$-nilpotent objects (for example, $\SMq^{\vphi\nilp}$). We similarly define $\psi$-nilpotent objects in $\SMq$, $\SM$, and $\SMqw$, and use the superscript $^{\psi\nilp}$ for the full subcategories of $\psi$-nilpotent objects
\end{defnsub}

Let $\gM\in\SMq$, and we set $\M:=S\otimes_{\vphi,\Sig}\gM\in\SMFq$ and 
\[\M_0:=\Sig/(u) \otimes_{\vphi,\Sig}\gM \cong R_0\otimes_S\M \in \PMFq{R_0},\] where $\PMFq{R_0}$ is defined using the frame as in Example~\ref{exa:S}.
\begin{lemsub}\label{lem:nilpBreuilvsKisin}
Under the notation as above, the following are equivalent:
\begin{enumerate}
\item\label{lem:nilpBKvsKisin:Kisin}
$\gM$ is $\vphi$-nilpotent (respectively, $\psi$-nilpotent) in the sense of Definition~\ref{def:psi}.
\item\label{lem:nilpBKvsKisin:Breuil}
$\M$ is $\vphi$-nilpotent (respectively, $\psi$-nilpotent) in the sense of Definition~\ref{def:vphiNilp}.
\item\label{lem:nilpBKvsKisin:dJ}
$\M_0$ is $\vphi$-nilpotent (respectively, $\psi$-nilpotent) in the sense of Definition~\ref{def:vphiNilp}.
\end{enumerate}
\end{lemsub}
\begin{proof}
We clearly have (\ref{lem:nilpBKvsKisin:Kisin}) $\Leftrightarrow$ (\ref{lem:nilpBKvsKisin:dJ}) and (\ref{lem:nilpBKvsKisin:Breuil})  $\Rightarrow$ (\ref{lem:nilpBKvsKisin:dJ}). The implication  (\ref{lem:nilpBKvsKisin:dJ}) $\Rightarrow$  (\ref{lem:nilpBKvsKisin:Breuil}) follows because $\vphi^n(I_0) \subset pS$ for $n\gg1$ where $I_0 = \ker (S\thra R_0)$.
\end{proof}

\begin{lem}\label{lem:CL-FF}
Assume that $R$ satisfies the $p$-basis condition~(\S\ref{cond:Breuil}).
If $p>2$ then the functors $\SMq \ra \SMFq$ and $\SM \ra \SMF$, defined by $S\otimes_{\vphi,\Sig} (\cdot)$, is fully faithful. If $p=2$ then the full faithfulness holds up to isogeny.

The similar statement holds for  $\SM$ and $\SMqw$.
\end{lem}
\begin{proof}
It suffices to show that $\SMq \ra \SMFq$ is fully faithful when $p>2$, and $\SMq[\ivtd p] \ra \SMFq[\ivtd p]$ is fully faithful when $p=2$. When $R$ is a discrete valuation ring  with perfect residue field, this lemma is can be obtained from the classification of $p$-divisible groups (combining Theorem~2.2.7 and Proposition~A.6 in \cite{kisin:fcrys}).

To handle the general case, let $R_0'$ and $R'$ be the $p$-adic completions of $\varinjlim_\vphi R_{0(p)}$ and $\varinjlim_\vphi R_{(\varpi)}$, respectively. We accordingly define $\Sig'$ and  $S'$, etc. Then there is a natural $\vphi$-compatible map $\Sig\ra\Sig'$ lifting the natural map $R\ra R'$; \emph{cf.} \S\ref{subsec:BC} (Ex.6). Inside $S'[\ivtd p]$ we have 
\begin{equation}\label{eqn:CL-FF}
S\cap \Sig' = \Sig\text{ and } S[1/p] \cap \Sig'[1/p] = \Sig[1/p].
\end{equation}

Now, let $\gM_1,\gM_2\in \SMq$, and write $\M_i:=S\otimes_{\vphi,\Sig}\gM_i\in\SMFq$, $\gM_i':=\Sig'\otimes_\Sig\gM_i\in\PMq{\Sig'}$, and $\M_i':=S'\otimes_{\vphi,\Sig}\gM_i\in \PMFq{S'}$ for $i=1,2$. We naturally view $\gM_i$ as a submodule in $\M_i$, $\gM_i'$, and $\M_i'$.

Assume that $p>2$ and we have a morphism $f: \M_1\ra\M_2$ in $\SMFq$. We want to show that $f$ maps $\gM_1$ into $\gM_2$. Since $R'$ is a discrete valuation ring with perfect residue field, it follows that $S'\otimes_S f:\M_1'\ra\M_2'$ maps $\gM_1'$ into $\gM_2'$. Now, from (\ref{eqn:CL-FF}) we have $\gM_i = \M_i\cap \gM_i'$ for $i=1,2$, which proves  the lemma when $p>2$. The same argument works when $p=2$ by inverting $p$ everywhere.
\end{proof}

%
When $p=2$ Lemma~\ref{lem:CL-FF} can be strengthened for  $\vphi$- and $\psi$- nilpotent objects (\emph{cf.} Definition~\ref{def:psi}) as follows:
\begin{lemsub}\label{lem:CL-FF:nilp}
Assume that $p=2$ and $R$ satisfies the $p$-basis condition~(\S\ref{cond:Breuil}).
Then the functors $\SMq^{\vphi\nilp} \ra \SMFq^{\vphi\nilp}$ and $\SMq^{\psi\nilp} \ra \SMFq^{\psi\nilp}$, defined by $S\otimes_{\vphi,\Sig} (\cdot)$, are fully faithful. The same statement holds for the full subcategories of $\vphi$- and $\psi$- nilpotent objects in $\SM$ and $\SMqw$.
\end{lemsub}
\begin{proof}
Note that $\SMq$ has a dualith $\gM\rightsquigarrow \gM^\vee$, where the underlying $\Sig$-module of $\gM^\vee$ is a $\Sig$-linear dual of $\gM$, and $\vphi_{\gM^\vee}$ is induced from $\psi_\gM$. Under this duality, $\gM$ is $\vphi$-nilpotent if and only if $\gM^\vee$ is $\psi$-nilpotent. Therefore, it suffices to show that $\SMq^{\psi\nilp} \ra \SMFq^{\psi\nilp}$ is fully faithful.

The same proof of Lemma~\ref{lem:CL-FF} shows that it suffices to prove the lemma when $R$ is a $p$-adic discrete valuation ring with perfect residue field, which follows from combining Proposition~1.1.9 and Theorem~1.2.8 in \cite{Kisin:2adicBT}.
\end{proof}

\begin{prop}\label{prop:CL}
Assume that $R$ satisfies the $p$-basis condition~(\S\ref{cond:Breuil}).
Then the functors $\SMq \ra \SMFq$ and $\SM \ra \SMF$, defined by $S\otimes_{\vphi,\Sig} (\cdot)$, are essentially surjective.
In particular, they are equivalences of categories if $p>2$ or if they are restricted to $\vphi$- and $\psi$- nilpotent objects, and are equivalence of categories up to isogeny if $p=2$.
\end{prop}
We prove the proposition later in \S\ref{subsec:PfCL}. Let us record some interesting corollaries. The following is immediate from 
\begin{corsub}\label{cor:RelKisin}
Assume that $R$ satisfies the $p$-basis condition~(\S\ref{cond:Breuil}).
If $p>2$ then  there exists an exact  contravariant functor 
\[
 \gM^*: \{p\text{-divisible groups over }R\} \ra \SM \riso \SMqw
\]
 such that for any $p$-divisible group $G$ over $R$ there exists a natural isomorphism $\DD^*(G)(S) \cong S\otimes_{\vphi,\Sig}\gM^*(G)$ in $\SMF$. Furthermore, $\gM^*$ is fully faithful if $R$ satisfies the normality assumption~(\S\ref{cond:BM}), and an anti-equivalence of categories if $R$ satisfies the formally finite-type assumption~(\S\ref{cond:dJ}).
 
 If $p=2$ then we have an exact contravariant functor $\gM^*[\ivtd p]: G\mapsto \gM^*(G)[\ivtd p]$ on the isogeny categories 
\[
\{p\text{-divisible groups over }R\} [1/p] \ra \SM[1/p],
\] 
which is  fully faithful     if $R$ satisfies the formally finite-type assumption~(\S\ref{cond:dJ}).
\end{corsub}

%

When $p=2$, we have the following strengthening of Corollary~\ref{cor:RelKisin}   for $\vphi$- and $\psi$- nilpotent objects:
\begin{corsub}\label{cor:FormalRelKisin}
Let $p=2$. Then the functor $\gM^*$ is defined on the category of formal and unipotent $p$-divisible groups (without passing to the isogeny categories), and is fully faithful if $R$ satisfies the normality assumption~(\S\ref{cond:BM}). If $R$ satisfies the formally finite-type assumption~(\S\ref{cond:dJ}), then $\gM^*$ induces an anti-equivalence of categories from the category of formal (respectively, unipotent) $p$-divisible groups to $\SMqw^{\vphi\nilp}$ (respectively, $\SMqw^{\psi\nilp}$). 
\end{corsub}
\begin{proof}
This follows from Proposition~\ref{prop:CL}, Theorem~\ref{thm:FormalBreuil}, and Lemmas~\ref{lem:FormalUnip} and \ref{lem:nilpBreuilvsKisin}.
\end{proof}
There will be another strengthening of Corollary~\ref{cor:RelKisin} when $p=2$ and $R$ satisfies the formally finite-type assumption~(\S\ref{cond:dJ}); see Corollary~\ref{cor:gMFF}.

\begin{rmksub}\label{rmk:ModConnOdd}
If $R$ is a complete regular local ring with perfect residue field, then Eike~Lau \cite{Lau:2010fk} proved a stronger result than Corollary~\ref{cor:RelKisin}; namely, the classification theorem without connection, which also holds when $p=2$. Although we cannot replace $\SMqw$ in Corollary~\ref{cor:RelKisin} with $\SMq$ in general, there are cases we can ``forget the connection''; \emph{cf.} Corollary~\ref{cor:BreuilClassif}. The role of connections will be studied further in \S\ref{sec:Vasiu}.
\end{rmksub}

\begin{rmksub}[lci case]\label{rmk:lciRelKisin}
Assume that  $R/(\varpi)$  is excellent and locally complete intersection in addition to the $p$-basis assumption~(\S\ref{cond:Breuil}). Then, as mentioned in Remark \ref{rmk:lciBreuilClassif} we can still show that the functor
\[G\rightsquigarrow\M^*(G):=\DD^*(G)(S)\in\SMF\] 
is fully faithful if $p>2$ (respectively, fully faithful up to isogeny if $p=2$).  Therefore, the same holds for the functor $\gM^*$ by Proposition~\ref{prop:CL}. 
\end{rmksub}

\begin{rmksub}\label{rmk:RelKisinBC}
The functors in Proposition~\ref{prop:CL} and Corollary~\ref{cor:RelKisin} commute with the base change which satisfy the condition stated at the beginning of \S\ref{subsec:BC}. 
In particular, we can define base change for ``\'etale morphisms'' and ``completions''  as in  \S\ref{subsec:BC} (Ex1) and (Ex5). 
\end{rmksub}

\subsection{Proof of Proposition~\ref{prop:CL}}\label{subsec:PfCL}
We want to show that the functor $\SMq \ra \SMFq$ is essentially surjective. We do this by modifying the proof of \cite[Theorem~2.2.1]{Caruso-Liu:qst}. The proof also works when $p=2$ with little modification. 

Let us begin the proof of Proposition~\ref{prop:CL} with a few preliminary lemmas. From now on, we consider $(\M,\Fil^1,\vphi_1)\in \SMFq$, and set $\M_0:=R_0\otimes_{S}\M$. The following lemma is not trivial when $\M$ is not necessarily free over $S$.

\begin{lemsub}\label{lem:SigLatt}
For $\M$ as above, there exists a projective $\Sig$-module $\gN$ equipped with an $S$-isomorphism $S\otimes_{\Sig}\gN = \M$. Any two such projective $\Sig$-modules $\gN$ and $\gN'$ are (non-canonically) isomorphic.
\end{lemsub}
\begin{proof}
Since $\M/(\Fil^1S)\M$ is a projective $R$-module, we can lift it to a projective $\Sig$-module $\wt\gN$. Then by Nakayama lemma, there is an isomorphism $S\otimes_\Sig \wt\gN \cong \M$. If there are two such $\Sig$-modules $\wt\gN$ and $\wt\gN'$, then one can find an isomorphism $\wt\gN\cong\wt\gN'$ which lifts an isomorphism $\wt \gN/\PP\wt\gN \cong \M/(\Fil^1S)\M \cong \wt\gN'/\PP\wt\gN'$.
%
\end{proof}

\begin{lemsub}\label{lem:adaptee}
There exists an $S$-linear injective map $B:\M\ra\Fil^1\M$ such that $\PP\M\subseteq B(\M)$, $\Fil^1\M =  B(\M) +(\Fil^pS)\M$, and $\vphi_1(B(\M))$ generates $\M$.
\end{lemsub}
\begin{proof}
Consider the ``Hodge filtration'' 
\begin{equation*}
0\ra\Fil^1\M/(\Fil^1S)\M \ra \M/(\Fil^1S)\M \ra \M/\Fil^1\M \ra0, 
\end{equation*}
which splits as  modules over $R=S/\Fil^1S$. We lift this filtration to an $S$-direct factor $\N\subset \M$ such that $\N\subset \Fil^1\M$, and choose a splitting $\M\cong \N\oplus \M/\N$. We define $B$ to be identity on $\N$ and multiplication by $\PP$ on $\M/\N$. 

By construction, we have $\Fil^1\M =  B(\M) +(\Fil^1S)\M$ and $B(\M) $ contains $\PP\M$. Since $\Fil^1 S = \PP S + \Fil^p S$, it follows that $\Fil^1\M =  B(\M) +(\Fil^pS)\M$. Now $\vphi_1(B(\M))$ has to generate $\M$ since $\frac{\vphi}{p}(\Fil^pS)\subseteq pS$.
\end{proof}

The following is the key technical lemma for proving Proposition~\ref{prop:CL}:
\begin{lemsub}\label{lem:CL2.2.2}
There exist a projective $\Sig$-module $\gM$ 
and an injective map $B:\M\ra\Fil^1\M$ which satisfies the following:
\begin{itemize}
\item There is an isomorphism $S\otimes_{\vphi,\Sig}\gM\cong \M$. (From now on, we view, via the natural inclusions, $\gM$ as a $\vphi(\Sig)$-submodule of $\M$, and  $\vphi^*\gM$ as a $\Sig$-submodule of $\M$.)
\item We have $\PP\M\subseteq B(\M)$, $\Fil^1\M =  B(\M) +(\Fil^pS)\M$, \\and $\gM=c\iv\vphi_1(B(\vphi^*\gM))$ in $\M$, where $c=\vphi(\PP)/p$.
\item The map  $B:\M \ra \Fil^1\M \subseteq \M$ takes $ \vphi^*\gM$ into itself.  
 \end{itemize}
\end{lemsub}
\begin{proof}
(\emph{Cf.} \cite[Lemma~2.2.2]{Caruso-Liu:qst}) We let $\Fil^1(\vphi^*\M) \subseteq \vphi^*\M$ denote the image of $\vphi^*(\Fil^1\M)$ in $\vphi^*\M$. As observed in the proof of Lemma~\ref{lem:BreuilClassifLifting}, we have an isomorphism $1\otimes\vphi_1:\Fil^1(\vphi^*\M)\riso \M$. 

For any $n\geqs0$ we recursively construct a projective $\Sig$-module $\gM^{(n)}$, and $S$-linear maps $B^{(n)}, C^{(n)}, D^{(n)}: \M \ra \M$ such that 
\begin{enumerate}
\item\label{lem:CL2.2.2:Latt} $S\otimes_{\vphi,\Sig}\gM^{(n)} \cong \M$ (so we view $\vphi^*\gM^{(n)}$ as a $\Sig$-submodule of $\M$, and $\gM^{(n)}$ as a $\vphi(\Sig)$-submodule of $\M$);
\item\label{lem:CL2.2.2:img} $B^{(n)}(\M), C^{(n)}(\M)\subseteq \Fil^1\M$, and $(C^{(n)} - B^{(n)})(\M) \subseteq(p^n\Fil^{n+p}S)\M$.
\item\label{lem:CL2.2.2:adaptee} we have $\PP\M\subseteq B^{(n)}(\M)$, $\Fil^1\M = B^{(n)}(\M) + (\Fil^pS)\M$, and $\gM^{(n)}=c\iv\vphi_1\big(B^{(n)}(\vphi^*\gM^{(n)})\big)$ where the equalities take place inside $\M$;
\item\label{lem:CL2.2.2:approx} $C^{(n)}$  takes $\vphi^*\gM^{(n)}$ into itself (hence, $C^{(n)}$ is automatically injective);
\item\label{lem:CL2.2.2:RecFormula} We have a recursion formula  $\gM^{(n+1)}:= c\iv \vphi_1\Big(C^{(n)}(\vphi^*\gM^{(n)})\Big)$. 
\item\label{lem:CL2.2.2:DirSys}  $1+D^{(n)}$ is an automorphism  of $\M$ which takes $\gM^{(n)}$ onto $\gM^{(n+1)}$ and such that 
$D^{(n)}(\M)\subseteq p^{\mu_n}\M$ for some strictly increasing sequence $\mu_n\in\Z_{>0}$.
\end{enumerate}
Let us first construct $\gM^{(0)}$, $B^{(0)}$, and $C^{(0)}$. Choose a projective $\Sig$-module $\gN$  with $S\otimes_\Sig\gN\cong \gM$, which exists by Lemma~\ref{lem:SigLatt}. Choose $B':\M\ra\Fil^1\M$ as in Lemma~\ref{lem:adaptee}. Set  $\gM^{(0)}:=c\iv\vphi_1\big(B'(\gN)\big)$, which is a $\vphi(\Sig)$-submodule of $\M$. We view $\gM^{(0)}$ as a $\Sig$-module by letting $s\in\Sig$ act via the multiplication by $\vphi(s)\in S$. This $\Sig$-module structure 
makes $(c\iv\vphi_1)\circ B':\gN\ra\gM^{(0)}$ into a $\Sig$-linear surjective map.  But since $\gM^{(0)}$ spans $\M$ by assumption on $B'$ (Lemma~\ref{lem:adaptee}), the $\Sig$-linear surjective map  $(c\iv\vphi_1)\circ B':\gN\ra\gM^{(0)}$ should be an isomorphism and $\gM^{(0)}$ is projective over $\Sig$ (as the source of the map is so).
Now  $\gM^{(0)}$ satisfies (\ref{lem:CL2.2.2:Latt}).

By the uniqueness assertion in Lemma~\ref{lem:SigLatt}, we can find an automorphism  of $\M$ which takes $\vphi^*\gM^{(0)}$ onto $\gN$. Let $B^{(0)}$ denote the composition of this automorphism with $B'$, which clearly satisfies (\ref{lem:CL2.2.2:adaptee}). To find $C^{(0)}$ satisfying (\ref{lem:CL2.2.2:approx}), we first observe that $S = \Sig+\Fil^p S$, so we have $\M = \vphi^*\gM^{(0)} + (\Fil^pS)\M$. Now, choose a map $C^{(0)}:\vphi^*\gM^{(0)} \ra\vphi^*\gM^{(0)}$ which lifts 
\begin{equation*}
\vphi^*\gM^{(0)} \xra{B^{(0)}} \M \thra \M/(\Fil^pS)\M \cong \vphi^*\gM^{(0)}/\PP^p(\vphi^*\gM^{(0)}).
\end{equation*}
Then by construction the image of $B^{(0)} - C^{(0)}$ is inside $(\Fil^pS)\M$.

From now on, let us assume that we have $\gM^{(n)}$, $B^{(n)}$, and $C^{(n)}$, which satisfy (\ref{lem:CL2.2.2:Latt})--(\ref{lem:CL2.2.2:approx}).  We now define  $\gM^{(n+1)}$ using the formula given in (\ref{lem:CL2.2.2:RecFormula}) 
as a $\vphi(\Sig)$-submodule of $\M$, and we view $\gM^{(n+1)}$ as a $\Sig$-module  by letting $s\in\Sig$ act via the multiplication by $\vphi(s)\in S$. The map $c\iv\vphi_1:C^{(n)}(\vphi^*\gM^{(n)}) \ra \gM^{(n+1)}$ is a $\Sig$-linear surjection, so $\gM^{(n+1)}$ would be a projective $\Sig$-module satisfying (\ref{lem:CL2.2.2:Latt})  provided that its $S$-linear span is $\M$ (i.e., $S\otimes_{\vphi,\Sig}\gM^{(n+1)} = \M$).

Now let us construct an automorphism $1+D^{(n)}$ of $\M$ which satisfies (\ref{lem:CL2.2.2:DirSys}). From this we will deduce that $\gM^{(n+1)}$ is projective over $\Sig$ and satisfies (\ref{lem:CL2.2.2:Latt}). By induction hypothesis (\ref{lem:CL2.2.2:adaptee}) on $B^{(n)}$, the $\Sig$-linear map $(c\iv\vphi_1)\circ B^{(n)}:\vphi^*\gM^{(n)} \ra \gM^{(n)}$ is an isomorphism. We set
\begin{equation}\label{eqn:Dn}
D^{(n)}: \M \xra {((c\iv\otimes\vphi_1)\circ (\vphi^*B^{(n)}))\iv} \vphi^*\M \xra{\vphi^*(C^{(n)}-B^{(n)})} \Fil^1(\vphi^*\M) \xra{c\iv\otimes\vphi_1}\M,
\end{equation} 
where $\Fil^1(\vphi^*\M)$ is the image of $\vphi^*(\Fil^1\M)$ in $\vphi^*\M$. 
By (\ref{lem:CL2.2.2:RecFormula}), the endomorphism $1+D^{(n)}$ of $\M$ takes $\gM^{(n)}$  onto $\gM^{(n+1)}$. 

\begin{claimsub}\label{clm:lambda}
We set $\lambda_n:= n+p - \lceil \frac{n+p}{p-1} \rceil$ if $p>2$,  where $\lceil \alpha \rceil = \inf\{x\in \Z|\ x\geqs \alpha\}$; and  $\lambda_n:=1$ if $p=2$.
Then we have $D^{(n)}(\M)\subseteq p^{\lambda_n+n}\M$. 
\end{claimsub}
Granting this claim, it follows that $1+D^{(n)}$ is an automorphism of $\M$,  and  $\gM^{(n+1)}$ is projective over $\Sig$ and satisfies (\ref{lem:CL2.2.2:Latt}).  (Note that $\lambda_n\geqs 1$ for any $n$.)

Let us prove Claim~\ref{clm:lambda}.
For any $s\in \Fil^1S$ and $m\in\M$ we have $\vphi_1(sm) = c\iv \vphi_1(s)\vphi_1(\PP m)$. Since $(C^{(n)} - B^{(n)})(\M) \subseteq(p^n\Fil^{n+p}S)\M$ by assumption, it suffices to show that $\vphi_1(s) \in p^{\lambda_n}S$ for any $s\in\Fil^{n+p}S$. By writing $s=\sum_{i\geqs n+p} a_i \PP^i/i!$ with $a_i\in \Sig$, this assertion is reduced to showing the inequality $i-\ord_p(i!) \geqs\lambda_n$ for any $i\geqs n+p$. When $p>2$, Claim~\ref{clm:lambda} follows from $\ord(i!) < \frac{i}{p-1}$. When $p=2$, we can directly check $i-\ord_2(i!) \geqs 1=\lambda_n$ for any $i\geqs p=2$. (Indeed, $2^n - \ord_2 (2^n!) = 1$ for any $n\geqs1$, and this is exactly when the lower bound is achieved.) This proves Claim~\ref{clm:lambda}.

Now we set 
\begin{equation}\label{eqn:Bn}
B^{(n+1)}:=C^{(n)}(1+D^{(n)})\iv = C^{(n)}\sum_{i=0}^\infty(-D^{(n)})^i.
\end{equation}
From the hypotheses on $B^{(n)}$, $C^{(n)}$, and $D^{(n)}$, it follows that $B^{(n+1)}$ satisfies  (\ref{lem:CL2.2.2:adaptee}). It remains to construct $C^{(n+1)}$ which satisfies (\ref{lem:CL2.2.2:approx}). We first observe that $B^{(n+1)}(\gM^{(n+1)}) \subseteq \gM^{(n)} = (1+D^{(n)})\iv(\gM^{(n+1)})$, so to construct $C^{(n+1)}$ it suffices to find a decomposition $D^{(n)} = D_1^{(n)} + D_2^{(n)}$ where $D_1^{(n)}(\gM^{(n)})\subseteq \gM^{(n+1)}$ and $D_2^{(n)}(\gM^{(n)})\subseteq (p^{n+1}\Fil^{n+1+p}S)\M$. Indeed, we will show that 
\begin{equation}\label{eqn:Dn2}
D^{(n)}(\M)\subseteq p^{\lambda_n+n}\M\subseteq \gM^{(n+1)}+(p^{n+1}\Fil^{n+1+p}S)\M.
\end{equation}
(Once this is done, one can repeat the argument for $n=0$ and obtain the desired expression $D^{(n)} = D_1^{(n)} + D_2^{(n)}$.)

Now to see the last inclusion in (\ref{eqn:Dn2}) it suffices to show that $p^{\lambda_n+n}S  \subseteq \Sig + p^{n+1}\Fil^{n+1+p}S$; i.e., $p^{\lambda_n+n}\frac{\PP^i}{i!}\in\Sig$ for any $i\leqs n+p$. Indeed, when $p>2$ we have $\lambda_n+n-\ord_p(i!) \geqs \lambda_n+n - \frac{i}{p-1}\geqs 0$, and when $p=2$ we have remarked that $\ord_2(i!)\leqs i-1$. Hence, we obtain $C^{(n+1)}$ as in (\ref{lem:CL2.2.2:approx}). This concludes the ``induction step''.

Now, let $\gM$ denote the direct limit of $\gM^{(n)}$; in other words, 
\begin{equation*}
\gM:=\big(\prod_{n=0}^\infty(1+D^{(n)})\big)(\gM^{(0)}),
\end{equation*}
which makes sense thanks to (\ref{lem:CL2.2.2:DirSys}). Since $\prod_{n=0}^\infty(1+D^{(n)})$ is an automorphism of $\M$, it follows that $\gM$ is projective over $\Sig$ and $\M\cong S\otimes_{\vphi,\Sig}\gM$.

We now define $B:\M\ra\Fil^1\M$ by the following limit 
\begin{equation*}
B:=\text{``}\lim_{n\to\infty}B^{(n)}\text{''}=B^{(0)}+\sum_{n=0}^\infty(B^{(n+1)} - B^{(n)}),
\end{equation*}
which converges since $(B^{(n+1)} - B^{(n)})(\M) \subset p^n\M$ by construction (\ref{eqn:Bn}). Note  that $C^{(n)}$ also converges to $B$ by induction hypothesis (\ref{lem:CL2.2.2:img}). Now all the desired properties of $B$ can be deduced from the properties of $B^{(n)}$ and $C^{(n)}$; especially from induction hypotheses (\ref{lem:CL2.2.2:adaptee}) and (\ref{lem:CL2.2.2:approx}).
\end{proof}

\begin{proof}[Proof of Proposition~\ref{prop:CL}]
Let $\gM$ and $B$ as in Lemma~\ref{lem:CL2.2.2}.
Recall that $\vphi_\M(m) = c\iv\vphi_1(\PP m)$ for any $m\in\M$. Since $B$ takes $\vphi^*\gM$ into itself and $\PP \M\subseteq B(\M)$, we have $\PP (\vphi^*\gM)\subseteq B(\vphi^*\gM)$. So from $c\iv\vphi_1(B(\vphi^*\gM))=\gM$, we see that $\gM\subset \M$ is stable under $\vphi_{\M}$. We set $\vphi_\gM:=\vphi_\M|_{\gM}$. 

We now show that $\PP \gM \subseteq (1\otimes\vphi_{\gM})(\vphi^*\gM)$. Note that $(1\otimes\vphi_\gM) (m') = c\iv \vphi_1(\PP m')$ for any $m'\in \vphi^*\gM$. For any $m\in \gM$ let $m'\in\vphi^*\gM$ be the (unique) element such that $c\iv\vphi_1(B(m')) = m$. Clearly, we have $c\iv\vphi_1\big(B(\PP m')\big) = \vphi(\PP )m$ as elements of $\M$; i.e., $(1\otimes\vphi_\gM)(m') = \PP \tim m$ as elements of $\gM$ since we defined $\Sig$-action on $\gM\subseteq\M$ via $\vphi:\Sig\ra S$. This shows that $\gM$ is a quasi-Kisin $\Sig$-module.
 
It remains to show  $\Fil^1(S\otimes_{\vphi,\Sig}\gM) = \Fil^1\M = B(\M)+(\Fil^pS)\M$. To show $\Fil^1(S\otimes_{\vphi,\Sig}\gM) \subseteq \Fil^1\M$, it suffices to observe that for any $m'\in\vphi^*\gM$, we have $(1\otimes\vphi)(m')\in \PP\gM$ if and only if $m'\in B(\vphi^*\gM)$; indeed, we have already seen the ``if'' direction, and the ``only if'' direction easily follows from $c\iv\vphi_1(B(\vphi^*\gM))=\gM$. To show $\Fil^1(S\otimes_{\vphi,\Sig}\gM) \supseteq \Fil^1\M$, it suffices to show that $\Fil^1(S\otimes_{\vphi,\Sig}\gM)\supseteq B(\vphi^*\gM)$. Note that the image of the natural inclusion $S\otimes_{\Sig}\gM \hra S\otimes_{\vphi,\Sig}\gM=\M$ is the $\vphi(S)$-span of $\gM$ in $\M$, and for any $s\in S$ the natural multiplication by $s$ on $S\otimes_{\Sig}\gM$ is translated as multiplication by $\vphi(s)$   on $\vphi(S)\tim\langle\gM\rangle$. Then we can identify $(\Fil^1S)\otimes_{\Sig}\gM$ with $\vphi(\Fil^1S)\M$, and  $(1\otimes\vphi_\gM)(B(m'))$ with $c\iv\vphi_1(\PP B(m')) $ for any $m'\in\vphi^*\gM$.  Now the claim follows from a direct computation.
\end{proof}

\begin{rmksub}
With the suitable generalisation of the notion of relative Breuil $S$-module for the ``semi-stable'' case with weights $\leqs r$ for $r< p-1$ (\emph{cf.,} Remark~\ref{rmk:BrModAlt}), Lemma~\ref{lem:CL-FF} and Proposition~\ref{prop:CL} hold in this generalisation  with the same proofs. 
\end{rmksub}

\section{Etale $\vphi$-modules and Galois representations}\label{sec:EtPhiMod}
We generalise the theory of \'etale $\vphi$-modules to our relative setting. 
All the main ingredients are in Scholze's work on perfectoid spaces \cite{Scholze:Perfectoid}. We refer to \cite{Scholze:Perfectoid} for the basic definitions on perfectoid algebras.

Throughout this section, we suppose that $R$ is a \emph{domain} which satisfies the formally finite-type assumption~(\S\ref{cond:dJ}). Furthermore, we \emph{assume} that there exists a Cohen subring $W\subset R_0$ such that $E(u)\in W[u]$; or equivalently, there exists a complete discrete valuation subring $\fo_K\subset R$ which contains $\varpi$ as a uniformiser. (If this is the case then we have $\fo_K:=W[u]/E(u)$ and $R = \fo_K\otimes_W R_0.$) We set $K:=\fo_K[1/\varpi]$. See Remark~\ref{rmk:PerfectoidCond} for the reason for this additional assumption.%

\subsection{Review: Perfectoid algebras}\label{subsec:perfectoid}
Let $L$ be a field complete with respect to a non-discrete valuation of rank $1$ (i.e., $|\bullet | :L\starr \ra \R$ which satisfies the axioms for  non-archimedean absolute value). Recall that $L$ is called a \emph{perfectoid field} if the $p$th power map $L^\circ/(p)\ra L^\circ/(p)$ is surjective, where $L^\circ$ is the valuation ring (\emph{cf.}  \cite[Definitions~3.1, Proposition~5.9]{Scholze:Perfectoid}). For example, $\wh{\Qbar}_p$ is a perfectoid field.

A Banach $L$-algebra $A$ is called \emph{perfectoid} if the subring of powerbounded elements $A^\circ\subset A$ is open and bounded and the $p$th power map $A^\circ/(p) \ra A^\circ/(p)$ is surjective  (\emph{cf.}  \cite[Definitions~5.1, Proposition~5.9]{Scholze:Perfectoid}). By a \emph{perfectoid affinoid $L$-algebra}, we mean a pair $(A,A^+)$ where $A$ is a perfectoid $L$-algebra and $A^+\subset A^\circ$ is an open and integrally closed $L^\circ$-subalgebra. As a trivial example, $(L,L^\circ)$ is a perfectoid affinoid $L$-algebra.

For an affinoid $L$-algebra $(A,A^+)$, we define $(A^\flat, A^{\flat+})$ as follows:
\[
A^{\flat+}:= \varprojlim_{x\mapsto x^p} A^+/(p);\quad A^\flat:=L^\flat\otimes_{L^{\flat\circ}}A^{\flat+},
\]
where  $L^{\flat\circ}:=\varprojlim_{x\mapsto x^p} L^\circ/(p)$ and $L^\flat = \Frac L^{\flat\circ}$. (\emph{Cf.}  \cite[Proposition~5.17, Lemma~6.2]{Scholze:Perfectoid}.) Note that $L^\flat$ is a perfectoid field of characteristic $p$, and $A^\flat$ and $A^{\flat+}$ are respectively $L^\flat$- and $L^{\flat\circ}$- algebras. We call $(A^\flat, A^{\flat+})$ the \emph{tilt} of $(A,A^+)$.

\begin{thmsub}[Scholze, {\cite[Theorem~5.2, Lemma~6.2, Theorem~7.9]{Scholze:Perfectoid}}]\label{thm:perfectoid}
The ``tilting'' $(A, A^+)\rightsquigarrow (A^\flat,A^{\flat+})$ induces an equivalence of categories from the category of perfectoid affinoid $L$-algebras to the category of perfectoid affinoid $L^\flat$-algebras. 
Furthermore, a morphism $(A,A^+)\ra (B,B^+)$ of perfectoid affinoid $L$-algebras is finite \'etale (i.e., $A\ra B$ is finite \'etale and $B^+$ is the normalisation of $A^+$ in $B$) if and only if $(A^\flat,A^{\flat+})\ra (B^\flat,B^{\flat+})$ is finite \'etale.
\end{thmsub}

Let us now introduce the perfectoid field  which will be the base field of all perfectoid algebras in characteristic~$0$. Here, we assume that $E(u)\in W[u]$ for some Cohen subring $W\subset R_0$, and let $\fo_K = W[u]/E(u)$ and $K:=\Frac \fo_K$. We fix $\ol R$ as in \S\ref{subsec:Acris}, and let $\wh{\ol R}$ be its $p$-adic completion. Then we define $L\subset \wh{\ol R}[\ivtd p]$ to be the smallest $p$-adically closed subfield  with perfect residue field which contains $K$ and $\varpi\com n$ for all $n$, where $\set{\varpi\com n}$ are compatible $p^n$th root of $\varpi$ chosen in   \S\ref{subsec:Rinfty}. In fact, $L$ is isomorphic to the $p$-adic completion of $\bigcup_{n\geqs0} (\Frac W(k^{\perf}))(\varpi\com n)$, where  $k^{\perf}:=\varinjlim_\vphi k$. Clearly, $L$ is a perfectoid field, and $(\wh{\ol R}[\ivtd p],\wh{\ol R})$ is a perfectoid affinoid $L$-algebra

Set $\wt\varpi:=(\varpi\com n)\in L^\flat$ (as in \S\ref{subsec:Rinfty}). Then we have $L^{\flat\circ} = k^{\perf}[[\wt\varpi^{1/p^\infty}]]$; i.e., the $\wt\varpi$-adic completion of $\varinjlim_\vphi k^{\perf}[[\wt\varpi]]$.  Then  $(\ol R^\flat[1/\wt\varpi],\ol R^\flat)$ is the tilt of $(\ol R[1/p],\ol R)$, so it is a perfectoid affinoid $L^\flat$-algebra (in the obvious way).

 We set $\eE_{R_\infty}^+:=\Sig/(p) = R/(\varpi)[[u]]$, and let $\wt\eE^+_{R_\infty}$ be the $u$-adic completion of the perfect closure $\varinjlim_\vphi \eE^+_{R_\infty}$. Set $\eE_{R_\infty}:=\eE^+_{R_\infty}[1/ u]$ and  $\wt\eE_{R_\infty}:=\wt\eE^+_{R_\infty}[1/ u]$. Note that $\wt\eE^+_{R_\infty} \subset \wt\eE_{R_\infty}$ is open and bounded for the $u$-adic topology, and it coincides with the subring of power-bounded elements. So by \cite[Proposition~5.9]{Scholze:Perfectoid}, $(\wt\eE_{R_\infty}, \wt\eE^+_{R_\infty})$ is a perfectoid affinoid $L^\flat$-algebra. By sending $u$ to $\wt\varpi$, we obtain a map $(\wt\eE_{R_\infty}, \wt\eE^+_{R_\infty})\hra (\ol R^\flat[1/\wt\varpi],\ol R^\flat)$ of perfectoid affinoid $L^\flat$-algebras. Using this, we view all of these rings as subrings of $\ol R^\flat[1/\wt\varpi]$.

Let $(\wt R_\infty[\ivtd p], \wt R_\infty)$ is  a perfectoid affinoid $L$-algebra whose tilt is $(\wt\eE_{R_\infty}, \wt\eE^+_{R_\infty})$. (Here, we  give the $p$-adic topology on  $(\wt R_\infty[\ivtd p], \wt R_\infty)$.) If, for example, $R=\fo_K\langle T_i\rangle_{i=1,\cdots,d}$, then we have $\wt R_\infty =L^\circ\langle T_i^{ p^{-\infty}}\rangle=\varprojlim_n L^\circ[T_i^{ p^{-\infty}}]/(p^n)$ by \cite[Proposition~5.20]{Scholze:Perfectoid}. In general, we can construct $\wt R_\infty$ explicitly as follows (\emph{cf.,} \cite[Remark~5.19]{Scholze:Perfectoid}): 
\[\wt R_\infty\cong W(\wt\eE^+_{R_\infty})\otimes_{W(L^{\flat\circ}),\theta}L^\circ,\] 
where $\theta$ is defined as in (\ref{eqn:theta}). Note also that $\wt R_\infty$ has a natural structure of $p$-adic $\wh R_\infty$-algebra using the Cartier morphism $R_0\ra W(\wt\eE^+_{R_\infty})$ (i.e., the unique $\vphi$-equivariant morphism which lifts the natural map $R_0/(p)\ra\wt\eE^+_{R_\infty}$).

Let us show that $\wt R_\infty$ is the $p$-adic completion of an integral extension of $R$ which becomes ind-\'etale after inverting $p$. Since the assertion is Zariski local on $\Spf (R,(\varpi))$, we may assume that $R/(\varpi)$ has (globally) a finite $p$-basis.  Then for any lift $\set{T_{i}}\subset R_0$ of a $p$-basis of $R/(\varpi)$ we have
\begin{equation}\label{eqn:wtRinfty}
\wt R_{\infty} \cong \wh R_{\infty} \wh\otimes_{\Zp}  \Zp\langle T_{i}^{p^{-\infty}}\rangle,
\end{equation}
where   the completed tensor product is taken with respect to the $p$-adic topology. 
Indeed, we have $\wh R_{\infty}/(\varpi) \otimes_{\Zp}  \Zp\langle T_{i}^{p^{-\infty}}\rangle \cong \wt\eE^+_{R_\infty}/(u)$ by mapping $T_i$ to the corresponding $p$-basis of $R/(\varpi)$ contained in  $\wt\eE^+_{R_\infty}/(u)$, so by the characterisation of the tilting (\emph{cf.,} Theorem~\ref{thm:perfectoid}) we have the isomorphism (\ref{eqn:wtRinfty}).
Now, the right hand side of (\ref{eqn:wtRinfty}) is the $p$-adic completion of  $ R_{\infty}\otimes_{\Zp}  \Zp[T_{i}^{p^{-\infty}}]$, which is an integral extension of $R$ that becomes ind-\'etale after inverting $p$, as desired.

Let $(\wt R_\infty[1/p],\wt R_\infty) \hra (\wh{\ol R}[1/p],\wh{\ol R})$ be the injective morphism whose tilt is $(\wt\eE_{R_\infty}, \wt\eE^+_{R_\infty})\hra (\ol R^\flat[1/\wt\varpi],\ol R^\flat)$ (defined by $u\mapsto\wt\varpi$). So we have a continuous map $\GRtinfty \ra \GRR$, where $\GRtinfty$ is the \'etale fundamental group of $\Spec \wt R_\infty[1/ p]$ (with respect to the common geometric generic point as a base point). It follows from \cite[Proposition~5.4.54]{GabberRamero} that this map is a closed embedding as $\wt R_\infty$ is the $\varpi$-adic completion of an integral $R$-subalgebra of $\ol R$ which is henselian along $(\varpi)$. (Note that a $\varpi$-adic ring is henselian along $(\varpi)$, and the property of being henselian along $(\varpi)$ is preserved under taking direct limits.) 

The following is a direct consequence of Scholze's ``almost purity theorem'' (stated in Theorem~\ref{thm:perfectoid}):
\begin{thmsub}\label{thm:NormRings}
For any fixed $\ol R$, $\ol R^\flat[1/\wt\varpi]$ can be canonically identified with the $\wt\varpi$-adic completion of the affine ring of a pro-universal covering of $\Spec \wt\eE_{R_\infty}$. Defining $\gal_{\wt\eE_{R_\infty}}$ using this pro-universal covering, there is a canonical isomorphism $\gal_{\wt\eE_{R_\infty}}\cong\GRtinfty$.
\end{thmsub}

\begin{rmksub}\label{rmk:PerfectoidCond}
Let us elaborate more on the assumption that $\varpi\in R$ is a uniformiser of  some discrete valuation subring $\fo_K\subset R$. We use this assumption to obtain the perfectoid base field $L$ such that $\wt\varpi\in L^\flat$. Without this assumption, it is unclear whether there exists a perfectoid subfield $L\subset \wh{\ol R}[\ivtd p]$ so that the ``untilt'' $\wt R_\infty$ of $\wt \eE_{R_\infty}$ contains $\varpi\com n$ for each $n\geqs 0$ and $u\in \wt \eE_{R_\infty}$ corresponds to $\wt\varpi$. 
\end{rmksub}

\subsection{\'Etale $\vphi$-modules and Galois representations}\label{subsec:EtPhi}
Let $\fo_{\Eps}$ be the $p$-adic completion of $\Sig[\ivtd u]$, and set $\Eps:=\fo_{\Eps}[\ivtd p]$. As $R$ satisfies the formally finite-type assumption~(\S\ref{cond:dJ}), the rings $\Sig$, $\fo_\Eps$, and $\Eps$ are regular (in particular, normal).
The endomorphism $\vphi:\Sig\ra\Sig$ extends to  $\fo_{\Eps}$ and  $\Eps$, which we also denote by $\vphi$. Note that $\fo_\Eps$ is a Cohen ring with residue field $\eE_{R_\infty}$, and $\vphi$ lifts the $p$th power map on the residue field.

\begin{defnsub}\label{def:EtPhiMod}
An \emph{\'etale $(\vphi,\fo_{\Eps})$-module} is a finitely generated $\fo_{\Eps}$-module $M$ equipped with a $\vphi$-linear endomorphism $\vphi_M:M \ra M$ such that the linearisation $1\otimes\vphi_M:\vphi^*M\ra M$ is an isomorphism. We say that an \'etale $(\vphi,\fo_{\Eps})$-module $M$ is \emph{projective} (respectively, \emph{torsion}) if the underlying $\fo_{\Eps}$-module $M$ is projective (respective $p$-power torsion).

Let $\EtPhiR$  denote the category of \'etale $(\vphi,\fo_{\Eps})$-modules. Let $\EtPhiPrR$ and  $\EtPhiTorR$ respectively  denote the full subcategories of projective  and torsion objects.
\end{defnsub}
Any \'etale $(\vphi,\fo_{\Eps})$-module annihilated by $p$ is automatically projective over $\eE_{R_\infty}:=\fo_{\Eps}/(p)$, which follows from the same proof as \cite[Lemma~7.10]{Andreatta:GenNormRings}. 

For any quasi-Kisin $\Sig$-module $\gM$, the scalar extension $M:=\fo_{\Eps}\otimes_{\Sig}\gM$ together with $\vphi_M:=\vphi_{\fo_{\Eps}}\otimes\vphi_\gM$ is a projective \'etale $(\vphi,\fo_{\Eps})$-module, since $\PP (u)$ is a unit in $\fo_{\Eps}$.

There exists a natural notion of subquotient, direct sum, $\otimes$-product for \'etale $\vphi$-modules. Duality is only defined for projective and torsion objects. For a projective \'etale $(\vphi,\fo_{\Eps})$-module $M$, we define a dual \'etale $(\vphi,\fo_{\Eps})$-module $M^*$ to be  the $\fo_{\Eps}$-linear dual of $M$ where $\vphi_{M^*}$ is defined so that $1\otimes\vphi_{M^*} = ((1\otimes\vphi_M)\iv)^*$. We can similarly define duality for torsion  \'etale $(\vphi,\fo_{\Eps})$-modules using Pontragin duality $M\rightsquigarrow \Hom_{\fo_{\Eps}}(M,\fo_{\Eps}\otimes_{\Zp}\Qp/\Zp)$.

The natural inclusion $\eE_{R_\infty}\hra \wt\eE_{R_\infty}$ has a unique lift $\fo_\Eps \hra W(\wt\eE_{R_\infty})$ with the property that the Witt vector Frobenius restricts to $\vphi$ on $\fo_{\Eps}$. Indeed,  the $p$-adic completion of $\varinjlim_\vphi\fo_\Eps$ can be naturally identified with $W(\eE_{R_\infty}^{\perf})$, so the desired morphism is obtained as follows:
\begin{equation}\label{eqn:EpsEmb}
\fo_\Eps \hra \varinjlim_\vphi\fo_\Eps \hra W(\eE_{R_\infty}^{\perf}) \hra W(\wt\eE_{R_\infty}).
\end{equation}
 We will view $ W(\ol R^\flat[1/\wt\varpi])$ as an $\fo_\Eps$-algebra via $u\mapsto [\wt\varpi]$. 

Let $\fo^{\ur}_\Eps$ be the integral closure of $\fo_\Eps$ in  $ W(\ol R^\flat[\ivtd u])$ and $\wh\fo^{\ur}_\Eps$ the $p$-adic closure of $\fo^{\ur}_\Eps$. 
%
Let $\Sig^{\ur}$ is the integral closure of $\Sig$ in $\wh\fo^{\ur}_\Eps$, and $\wh\Sig^{\ur}$ be its $p$-adic closure. Note that $\wh\Sig^{\ur}$ is indeed contained in  $ W(\ol R^\flat)$, not just in   $W(\ol R^\flat[1/\wt\varpi])$.

We now define the $\GRtinfty$-action on these rings. As $\fo_\Eps$ is normal, it follows from Theorem~\ref{thm:NormRings} that $\fo_\Eps^{\ur}$ is the union of finite \'etale $\fo_\Eps$-subalgebras  in $W(\ol R^\flat[1/\wt\varpi])$, and $\Aut_{\fo_\Eps}(\fo^{\ur}_{\Eps})$ is isomorphic to the fundamental group of $\Spec \fo_\Eps$ (with suitable base point).
Now, not that we have the following natural equivalences of categories induced by base change
\begin{multline}
\{\text{finite \'etale covers of }\Spec \fo_\Eps\} \riso \{\text{finite \'etale covers of }\Spec \eE_{R_\infty}\} \\
\riso \{\text{finite \'etale covers of }\Spec \wt\eE_{R_\infty} \}.
\end{multline}
Indeed, the first arrow is an equivalence because $\fo_\Eps$ is $p$-adic and finite \'etale morphisms uniquely lift under infinitesimal thickenings. The second arrow is an equivalence because $\wt\eE_{R_\infty}$ is the $u$-adic completion of a radicial extension of $\eE_{R_\infty}$ (which is henselian along $(u)$), so we may apply \cite[Proposition~5.4.54]{GabberRamero}. Now, the isomorphism $\GRtinfty\cong \gal_{\wt\eE_{R_\infty}}$ in Theorem~\ref{thm:NormRings} and the above equivalences of categories produce the following isomorphisms:
\begin{equation}\label{eqn:FundGpIsoms}
\Aut_{\fo_\Eps}(\fo^{\ur}_{\Eps})\cong\gal_{\eE_{R_\infty}}\cong\gal_{\wt\eE_{R_\infty}}\cong\GRtinfty.
\end{equation}
Since $\GRtinfty$-action on $\fo_\Eps^{\ur}$ fixes $\Sig$, it  stabilises the subring $\Sig^{\ur}\subset\fo_\Eps^{\ur}$.
Since the $\GRtinfty$-action on $\fo^{\ur}_{\Eps}$ is $p$-adically continuous, it extends to a continuous $\GRtinfty$-action on $\wh \fo^{\ur}_{\Eps}$ and $\wh\Sig^{\ur}_R$, respectively. In particular,  we obtain a $\GRtinfty$-equivariant embedding 
\begin{equation}\label{eqn:SigurToAcris}
\wh\Sig^{\ur}\hra  W(\ol R^\flat) \subseteq \Acris^\nabla(R).
\end{equation}

\begin{lemsub}
We have $\fo_\Eps = (\wh\fo^{\ur}_\Eps)^{\GRtinfty}$ and $\Sig = (\wh\Sig^{\ur})^{\GRtinfty}$. In particular, the same statement holds modulo $p^n$ as well.
\end{lemsub}
\begin{proof}
It suffices to prove $\fo_\Eps \supseteq (\wh\fo^{\ur}_\Eps)^{\GRtinfty}$. Since  $(\wh\fo^{\ur}_\Eps)^{\GRtinfty}\subset \wh\fo^{\ur}_\Eps$ is a closed subspace under the $p$-adic topology (by continuity of $\GRtinfty$-action), this claim follows from $\eE_{R_\infty} = \fo_\Eps/(p) \cong  (\fo^{\ur}_\Eps/(p))^{\GRtinfty}$.
\end{proof}

\begin{lemsub}\label{lem:sig}
There exists a unique $\GRtinfty$-equivariant ring endomorphism $\vphi:\wh\fo^{\ur}_{\Eps} \ra \wh\fo^{\ur}_{\Eps}$ which lifts the $p$th power map $\vphi:\fo^{\ur}_{\Eps}/(p) \ra \fo^{\ur}_{\Eps}/(p)$ and extends $\vphi:\fo_{\Eps}\ra \fo_{\Eps}$. Furthermore, this map restricts to $\vphi:\wh\Sig^{\ur}\ra\wh\Sig^{\ur}$, and the natural inclusion $\wh\fo^{\ur}_{\Eps} \hra W(\ol R^\flat[1/\wt\varpi])$ is $\vphi$-equivariant.
\end{lemsub}
\begin{proof}
Let us first show that for any finite \'etale $\fo_{\Eps}$-algbra $A$, there exists a unique $\fo_{\Eps}$-algbra isomorphism $\vphi^*A\ra A$ which lifts the relative Frobenius (iso)morphism $\vphi^*A/(p)\ra A/(p)$ over $\fo_{\Eps}/(p)$.  This follows from the same argument as the proof of Lemma~\ref{lem:lifting} (the existence of $\vphi:R_0\ra R_0$). The uniqueness follows from $\fo_{\Eps}$-\'etaleness of $A$, and clearly $\vphi$ sends any element integral over $\Sig$ to an element integral over $\Sig$. 

Using the unique lift of Frobenius $\vphi:A\ra A$, we can now construct a $\vphi$-equivariant embedding 
\[A\hra W(A\otimes_{\fo_\Eps}\wt\eE_{R_\infty})\]
as in (\ref{eqn:EpsEmb}). By choosing an $\wt\eE_{R_\infty}$-embedding $A\otimes_{\fo_\Eps}\wt\eE_{R_\infty} \hra \ol R^\flat[1/\wt\varpi]$, we now obtain the unique $\vphi$-equivariant lift $A\hra W(\ol R^\flat[1/\wt\varpi])$. Note that $\fo^{\ur}_\Eps$ is the union of the image of such $A$. This gives a unique lift of Frobenius $\vphi:\fo_\Eps^{\ur}\ra\fo_\Eps^{\ur}$ over $\fo_\Eps$ with respect to which the natural inclusion $\fo^{\ur}_\Eps\hra W(\ol R^\flat[1/\wt\varpi])$ is $\vphi$-equivariant. The desired lift of Frobenius $\vphi:\wh\fo_\Eps^{\ur}\ra\wh\fo_\Eps^{\ur}$ is obtained by $p$-adically extending this map. The uniqueness follows from the uniqueness of $\vphi$ on $\fo_\Eps^{\ur}$ and the density of  $\fo_\Eps^{\ur}\subset \wh\fo_\Eps^{\ur}$.

Finally, the $\GRtinfty$-equivariance claim follows from the uniqueness of $\vphi$. Indeed, for any $g\in \GRtinfty$ the uniqueness implies $g\vphi g\iv=\vphi$. 
\end{proof}

For any $M\in \EtPhiR$ and $T\in\prep(\GRtinfty)$, we define:
\begin{align}
\cT(M) &:= (\wh\fo^{\ur}_{\Eps}\otimes_{\fo_{\Eps}}M)^{\vphi=1}\\
\cD(T) &:= (\wh\fo^{\ur}_{\Eps}\otimes_{\Zp}T)^{\GRtinfty} .
\end{align}
Note that $\GRtinfty$ continuously acts on $\cT(M)$ induced by its natural action on $\wh\fo^{\ur}_{\Eps}$, and there is a natural $\vphi$-linear endomorphism on $\cD(T)$ induced by $\vphi$ on $\wh\fo^{\ur}_{\Eps}$.

For any profinite group $\gal$, let $\prep(\gal)$ denote the category of finitely generated $\Zp$-modules equipped with continuous $\gal$-action. Let $\freeprep(\gal)$ and $\torprep(\gal)$ respectively denote the full subcategories of free and torsion objects. 
\begin{prop}[{\cite[Lemma~4.1.1]{katz:antwerp350}}]\label{prop:Katz}
The constructions $\cT$ and $\cD$ are exact quasi-inverse equivalences of $\otimes$-categories between $\EtPhiR$ and $\prep(\GRtinfty)$. Furthermore, $\cT$ and $\cD$ restrict to rank-preserving equivalences of categories between $\EtPhiPrR$ and $\freeprep(\GRtinfty)$, and length-preserving equivalences of categories between $\EtPhiTorR$ and $\torprep(\GRtinfty)$. In both cases, $\cT$ and $\cD$ commute with  duality.
\end{prop}
\begin{proof} (\emph{Cf.} \cite[Theorem~7.11]{Andreatta:GenNormRings})
By d\'evissage, it suffices to prove the proposition for the objects killed by $p$, which is done in \cite[Lemma~4.1.1]{katz:antwerp350}; note that  $\eE_{R_\infty} $ is a normal domain (which is required for \cite[Lemma~4.1.1]{katz:antwerp350}), and we have identified $\GRtinfty$ with a fundamental group of $\Spec \eE_{R_\infty}$ in (\ref{eqn:FundGpIsoms}).
\end{proof}

When $M$ is a projective \'etale $(\vphi,\fo_{\Eps})$-module, we can define the following contravariant  functor:
\begin{equation}\label{eqn:cTstar}
\cT^*(M):=\cT(M^*) = \Hom_{\fo_{\Eps},\vphi}(M, \wh\fo^{\ur}_{\Eps}).
\end{equation}
By Proposition~\ref{prop:Katz}, we have a natural isomorphism $\cT^*(M)\cong (\cT(M))^*$. One can define $\cT^*$ for torsion \'etale $(\vphi,\fo_{\Eps})$-modules using Pontragin duality:
\begin{equation}
\cT^*(M): = \Hom_{\fo_{\Eps},\vphi}(M, \wh\fo^{\ur}_{\Eps}\otimes_{\Zp}\Qp/\Zp).
\end{equation}

Clearly, $\cT^*$ preserves exact sequences of projective or torsion objects. Furthermore, if we have a short exact sequence $0\ra \wt M' \ra\wt M \ra M \ra 0$ of \'etale $(\vphi,\fo_\Eps)$-module, where $\wt M$ and $\wt M'$ are $\fo_\Eps$-projective, and  $M$ is $p$-power torsion, then we have a natural $\GRtinfty$-equivariant short exact sequence
\begin{equation}\label{eqn:RednEtPhiMod}
0\ra\cT^*(\wt M) \ra \cT^*(\wt M') \ra\cT^*(M)\ra 0,
\end{equation} 
where the second map is defined by viewing $f:\wt M' \ra \wh\fo^{\ur}_\Eps$ as $f:\wt M \ra \wh\fo^{\ur}_\Eps\otimes_{\Zp}\Qp$ using the isomorphism $\wt M'\otimes_{\Zp}\Qp\riso\wt M\otimes_{\Zp}\Qp$.

\subsection{Base change}\label{subsec:BCGal}
Let $R$, $W\subset R_0$, and $E(u)$ be as in the beginning of \S\ref{sec:EtPhiMod}, and consider a topological adic $R_0$-algebra $R_0'$ such that $R_0'/(p)$ locally admits a finite $p$-basis and is formally finitely generated over some field $k'$. Set $R':= R_0'[u]/E(u)$, and let  $f:R\ra R'$ denote the structure morphism. (\emph{Cf.} \S\ref{cond:dJ}, \S\ref{subsec:BC}.)
Choose a map $\ol f:\ol R \ra \ol R'$ that extends $f:R\ra R'$ (as in \S\ref{subsec:BCrhoG}). We use the superscript $\prime$ to denote the construction for $R'$ (such as $\Sig'$ and $\fo_{\Eps'}$). Note that $\ol f$ induces a continuous group homomorphism $\gal_{R'}\ra \GRR$, and a continuous ring homomorphism $\ol f^\flat:\ol R^\flat\ra \ol R^{\prime\flat}$ which respects the Galois action in a suitable sense. 

As in (\ref{eqn:SigurToAcris}), we identify $\wh\Sig^{\ur}$ and $\wh\Sig^{\prime\ur}$ as subrings  of $W(\ol R^\flat)$ and $W(\ol R^{\prime\flat})$, respectively. Since the map $W(\ol f^\flat):W(\ol R^\flat)\ra W(\ol R^{\prime\flat})$ takes $\Sig$ into $\Sig'$, $W(\ol f^\flat)$ restricts to a map $\wh\Sig^{\ur} \ra \wh\Sig^{\prime\ur}$ which commutes with $\vphi$'s and $\gal_{R'_\infty}$-action if we let $\gal_{R'_\infty}$ act on $\wh\Sig^{\ur} $ via the map $\gal_{R'_\infty}\ra \GRinfty$. We also obtain a map $\wh\fo^{\ur}_{\Eps} \ra \wh\fo^{\ur}_{\Eps'}$, which respects $\vphi$ and Galois actions. 

Now for any \'etale $(\vphi,\fo_{\Eps})$-module $M$, we obtain a natural $\gal_{R'_\infty}$-equivariant map $\cT(M)\ra\cT(\fo_{\Eps'}\otimes_{\fo_{\Eps}}M)$, where $\gal_{R'_\infty}$ acts on $\cT(M)$ via the natural map  $\gal_{R'_\infty}\ra \GRinfty$. This map $\cT(M)\ra\cT(\fo_{\Eps'}\otimes_{\fo_{\Eps}}M)$ is indeed an isomorphism, as it is an injective map of finite free $\Zp$-modules of same rank with saturated image. (To see the image is saturated, note that $\Eps\cap\fo_{\Eps'} = \fo_\Eps$ inside $\Eps'$.)
Similarly, we  obtain $\cT^*(M)\riso\cT^*(\fo_{\Eps'}\otimes_{\fo_{\Eps}}M)$.

\begin{lem}\label{lem:FontaineB183}
For any quasi-Kisin $\Sig$-module $\gM$, the natural morphism 
\begin{equation*}
\Hom_{\Sig,\vphi}(\gM, \wh\Sig^{\ur}) \ra \cT^*(\gM) = \Hom_{\Sig,\vphi}(\gM, \wh\fo^{\ur}_{\Eps}),
\end{equation*}
induced by the natural inclusion $\wh\Sig^{\ur} \hra \wh\fo^{\ur}_{\Eps}$, is an isomorphism.
\end{lem}
\begin{proof}
When $R = \fo_K$ with perfect residue field the proposition is proved in \cite[\S{B}, Proposition~1.8.3]{fontaine:grothfest}. For the general case, let  $k':=\varinjlim_\vphi \Frac(R/\varpi)$ be the perfect closure of $\Frac(R/\varpi)$. As discussed in \S\ref{subsec:BC} (Ex6), there is an $\fo_K$-algebra homomorphism $R\ra \fo_{K'}=W(k')\otimes_{W}\fo_K$ which satisfies the assumption in \S\ref{subsec:BC}.  
We apply the discussion in \S\ref{subsec:BCGal} to this setting.

Let us write 
$\Sig'$, $\fo_{\Eps'}$, $\wh\Sig^{\prime\ur}$, and  $\wh\fo^{\ur}_{\Eps'}$ for the rings constructed from $\fo_{K'}$. 
Now, note that for any quasi-Kisin $\Sig$-module $\gM$, the scalar extension $\Sig'\otimes_{\Sig}\gM$ is a quasi-Kisin $\Sig'$-module. In particular, we obtain a natural $\gal_{K'_\infty}$-equivariant isomorphism $\cT^*(\gM)\riso \cT^*(\Sig'\otimes_{\Sig}\gM)$ (\emph{cf.,} \S\ref{subsec:BCGal}). Now, applying the result for $\fo_{K'}$ (which is known by  \cite[\S{B}, Proposition~1.8.3]{fontaine:grothfest}), it follows that any $\vphi$-compatible map $x:\gM\ra \wh\fo^{\ur}_{\Eps'}$ has image in $\wh\Sig^{\prime\ur}\cap \wh\fo^{\ur}_{\Eps}$. To conclude it suffices to show that $\wh\Sig^{\ur} = \wh\Sig^{\prime\ur}\cap \wh\fo^{\ur}_{\Eps}$. In turn, it suffices to show that for any $s\in \Sig^{\prime\ur}\cap \fo^{\ur}_{\Eps}$ the subalgebra $\Sig[s]\subseteq \fo^{\ur}_{\Eps'}$ is finite over $\Sig$. But $\{\Spec \Sig', \Spec \fo_{\Eps}\}$ is an fpqc covering of $\Spec \Sig$, and the assumption on $s$ implies that $\Sig'[s]$ is finite over $\Sig'$ and $\fo_{\Eps}[s]$ is finite over $\fo_{\Eps}$. Now the claim follows from fpqc descent theory.
\end{proof}
%

\section{Comparison between Galois-stable lattices}\label{sec:FaltingsKisin}
Assume that $R$ satisfies the formally finite-type assumption~(\S\ref{cond:dJ}). If $p>2$ then we have the following  exact functor
\begin{equation*}
 \gM^*: \{p\text{-divisible groups over }R\} \xra{\M^*} \SMF \liso \SM,
 \end{equation*}
 which is an  anti-equivalence of categories if $p>2$; \emph{cf.} Corollary~\ref{cor:RelKisin}.
We also associated, to a $p$-divisible group $G$ over $R$, a $\GRtinfty$-representation $\cT^*(\gM^*(G))$ under some additional assumption on $R$ (as stated in \S\ref{sec:EtPhiMod}); \emph{cf.,} (\ref{eqn:cTstar}). The main goal of this section is to prove the following proposition:

\begin{prop}\label{prop:GRinftyRes}
Assume that $R$ satisfies the ``refined almost \'etaleness'' assumption~(\S\ref{cond:RAF}), and for some Cohen subring $W\subset R_0$ we have $E(u)\in W[u]$ (\emph{cf.} \S\ref{sec:EtPhiMod}).
Let $\M\in\SMF$, and suppose that $\gM\in\SMq$ is a quasi-Kisin $\Sig$-module such that $\M\cong S\otimes_{\vphi,\Sig}\gM$ as quasi-Breuil $S$-modules (which exists by Proposition~\ref{prop:CL}).
Then there is a natural $\GRtinfty$-equivariant injective map 
\begin{equation*}
 \cT^*(\gM) \hra T^*_{\cris}(\M),
\end{equation*}
which is an isomorphism if $p>2$. 
Furthermore, if $\M=\DD^*(G)(S)$ for some $p$-divisible group $G$, then (for any $p$) we have a natural $\GRtinfty$-equivariant injective map $\cT^*(\gM)\hra V_p(G)$, whose image is $T_p(G)$ .
\end{prop}
If $p>2$ then the second assertion  follows from the first due to Theorem~\ref{thm:Faltings}, while it requires more work if $p=2$. When $p=2$, we can use Proposition~\ref{prop:GRinftyRes} to define the functor $\gM^*:G\rightsquigarrow \gM$ (without passing to the isogeny categories), and obtain some  full faithfulness result; \emph{cf.,} Corollary~\ref{cor:gMFF}.

Before we begin the proof, let us record an immediate corollary of Proposition~\ref{prop:GRinftyRes}:
\begin{cor}\label{cor:GRinftyRes}
Assume that $R$ satisfies the formally finite-type assumption~(\S\ref{cond:dJ}), and for some Cohen subring $W\subset R_0$ we have $E(u)\in W[u]$ (\emph{cf.} \S\ref{sec:EtPhiMod}). Let $G$ be a $p$-divisible group over $R$, and let $\gM\in\SMq$ be such that $S\otimes_{\vphi,\Sig}\gM\cong \DD^*(G)(S)$. 
Then (for any $p$) there is a natural $\GRtinfty$-equivariant isomorphism $\cT^*(\gM)\cong T_p(G)$, which recovers the isomorphism in Proposition~\ref{prop:GRinftyRes} when $R$ additionally satisfies the ``refined almost \'etaleness'' assumption~(\S\ref{cond:RAF}).
\end{cor}
\begin{proof}
By Lemma~\ref{lem:etloc} and the discussion in \S\ref{subsec:BCGal}, for any open immersion of formal schemes $\Spf (R',J_RR')\hra \Spf(R,J_R)$ we have a natural $\Zp$-isomorphism 
\[\cT^*(\gM)\riso \cT^*(\Sig_{R'}\otimes_\Sig\gM).\]
For any $(R,J_R)$ as in the statement, there exists a Zariski open cover $\set{\Spf (R_\alpha, J_{R_\alpha})}$ of $\Spf(R,J_R)$ such that each $R_\alpha$ satisfies the ``refined almost \'etaleness'' assumption~(\S\ref{cond:RAF}). (Note that $R[\ivtd p] = K\otimes_W R_0$ is  \'etale over $R_0[\ivtd p]$.) To prove the corollary, it suffices to prove it for each $R_\alpha$, for which we can apply Proposition~\ref{prop:GRinftyRes}.
\end{proof}

Recall that $\Tcris^*$ is defined via a  $\GRinfty$-invariant embedding $S\hra \Acris(R)$ which respects the suitable connections on both sides (\emph{cf.,} (\ref{eqn:Tcris})). On the other hand, $\cT^*$ is defined using the period ring $\wh\Sig^{\ur}$, which admits a $\GRtinfty$-equivariant embedding into $W(\ol R^\flat)$, hence into $\Acris^\nabla(R)$. To ``interpolate'' these two constructions, we introduce  auxiliary semilinear objects over $\wt S$ as below. 
\subsection{More divided power algebras}\label{subsec:STilde}
As in Proposition~\ref{prop:GRinftyRes}, we assume that $R$ satisfies the ``refined almost \'etaleness'' assumption~(\S\ref{cond:RAF}),  and for some Cohen subring $W\subset R_0$ we have $E(u)\in W[u]$ (\emph{cf.} \S\ref{sec:EtPhiMod}), although some constructions work in more generality.
Consider $R_0\wh\otimes_{\Zp}\Sig \thra R$ defined by $R_0$-linearly extending the natural projection $\Sig\thra \Sig/(\PP(u) )\cong R$.  (Here, $\wh\otimes$ is the completed $\otimes$-product with respect to the $p$-adic topology.) Let $\wt S$ denote the $p$-adically completed divided power envelope of $R_0\wh\otimes_{\Zp}\Sig$ with respect to $\ker (R_0\wh\otimes_{\Zp}\Sig \thra R)$, which is generated by $\PP(u) $ and $a\otimes1-1\otimes a$ for $a\in R_0$.  

Let $\Fil^1\wt S$ denote the divided power ideal with $\wt S/\Fil^1\wt S  \cong R$. Note that $\vphi\otimes\vphi$ on $R_0\otimes_{\Zp}\Sig$ extend to $\vphi$ on $\wt S$, and we have $\langle\vphi(\Fil^1\wt S)\rangle =  p\wt S$ as $\vphi(\PP(u))\in p\wt S\starr$. We set $\vphi_1:=\frac{\vphi}{p}:\Fil^1\wt S \ra \wt S$. 

We define 
a connection $d_{\wt S}:\wt S \ra \wt S\otimes_{R_0\wh\otimes \Sig}\wh\Omega_{R_0\wh\otimes \Sig}$ by extending the universal continuous derivation for $R_0\wh\otimes_{\Zp} \Sig$ in a usual way, and define $d^u_{\wt S}:\wt S \ra \wt S\otimes_{R_0}\wh\Omega_{R_0}$ by composing $d_{\wt S}$ with the natural projection $\wt S\otimes_{R_0\wh\otimes \Sig}\wh\Omega_{R_0\wh\otimes \Sig}\thra \wt S\otimes_{R_0}\wh\Omega_{R_0}$. 

One can embed $\Sig=R_0[[u]]$ into $\wt S$ in two different ways, as follows: for any $\sum_{n=0}^\infty a_nu^n\in \Sig$ with $a_n\in R_0$
\begin{align}
\imath_1:\sum_{n=0}^\infty a_nu^n &\mapsto \sum_{n=0}^\infty a_n\otimes u^n \in \wt S\\
\imath_2:\sum_{n=0}^\infty a_nu^n &\mapsto 1\otimes\left(\sum_{n=0}^\infty a_nu^n\right) \in \wt S.
\end{align}
Both maps extend to $\imath_1,\imath_2:S\hra \wt S$, which respect $\vphi$ and the divided power structures. Note that $\imath_1$ is horizontal when $S$ is given a connection $d^u_{S}$, and $\imath_2(S) = (\wt S)^{d^u_{\wt S}=0}$. Let us also consider the map
\begin{equation}\label{eqn:diagonal}
\delta:\wt S \ra S;\ a\otimes s \mapsto a\tim s \text{ for any } a\in R_0,\ s\in\Sig,
\end{equation}
which respects $\vphi$, the divided power structures, and the connections when $S$ is endowed with $d^u_{S}$. Also note that $\delta\circ\imath_1 = \delta\circ\imath_2 = \id_{S}$.

We now define a $\vphi$-compatible $\GRtinfty$-invariant $R_0$-algebra map  $\tilde\jmath:\wt S\hra \Acris(R)$ as follows. First, recall that we have a $\vphi$-compatible $\GRtinfty$-invariant map $\Sig\hra W(\ol R^\flat)$. (\emph{Cf.} (\ref{eqn:SigurToAcris}).) We may $R_0$-linearly extend this map to $R_0\otimes_{\Zp}\Sig\hra R_0\otimes_{\Zp}W(\ol R^\flat)$, which extends to $\wt S\hra \Acris(R)$.  This map, which we denote by $\tilde\jmath$, is clearly compatible with $\vphi$ and $\nabla$, and the image of $\Fil^1\wt S$ is precisely the intersection of the image of $\wt S$ and $\Fil^1\Acris(R)$. 

We remark that the embedding $S \xra{\imath_1} \wt S \xra{\tilde\jmath}\Acris(R)$ is precisely the embedding we use to define $\Tcris^*$ (\ref{eqn:Tcris}). The embedding $\Sig  \xra{\imath_2} \wt S \xra{\tilde\jmath}\Acris( R_0)$ has image in $W(\ol R^\flat)\subset \Acris^\nabla(R)$ and coincides with the map (\ref{eqn:SigurToAcris}).

\subsection{More filtered modules with connection}\label{subsec:Mtilde}
By applying Proposition~\ref{prop:FiltMod} to $\wh D = \wt S$, we obtain an equivalence of categories $\E\rightsquigarrow\E(\wt S)$ from the category of Dieudonn\'e crystals over $\Spf(R,(\varpi))$ to $\wtSMF$ (following the notation of Definition~\ref{def:DieudonneModwHF}). Since $\E\rightsquigarrow\E(\wt S)$ also defines an equivalence of categories into $\SMF$, we obtain an equivalence of categories $\SMF \riso \wtSMF$. 

We can indeed make the functor $\SMF \ra \wtSMF$ explicit as follows. The maps $\imath_1,\imath_2:S\ra \wt S$  and $\delta:\wt S\ra S$ are divided power morphisms over $R$, so they induce a horizontal $\vphi$-compatible filtered morphisms $\E(S)\underset{\delta}{\overset{\imath_j}{\rightleftharpoons}}\E(\wt S)$ (with $j=1,2$) for any Dieudonn\'e crystal $\E$ over $R$. Since $\delta\circ\imath_1 = \id_{S}$, we obtain natural isomorphisms 
\begin{equation}\label{eqn:Mtilde}
\E(\wt S) \cong \wt S\otimes_{\imath_j,S}\E(S),\quad \E(S)\cong S\otimes_{\delta,\wt S}\E(\wt S), \text{ for } j=1,2.
 \end{equation}
Set $\M:=\E(S)$ and $\wt\M:=\wt S\otimes_{\imath_1,S}\E(S)$, using $\imath_1$. Then one can check that the following coincides with the natural extra structure on $\E(\wt S)$ as in Proposition~\ref{prop:FiltMod}:
\begin{enumerate}
\item $\vphi_{\wt\M} = \vphi_{\wt S}\otimes \vphi_\M$;
\item $\nabla_{\wt\M}$ is the connection over $d_{\wt S}$, such that if for $m\in\M$ we have $\nabla_\M(m)=\sum_i m_i \otimes ds_i$ then $\nabla_{\wt\M}(1\otimes m) = \sum_i m_i \otimes d(\iota_1(s_i))$;
\item $\Fil^1\wt\M$ is the preimage of the Hodge Filtration on $\E(R)$  by the natural projection $\E(\wt S)\thra \E(R)$.  
\end{enumerate}
\begin{rmksub}\label{rmk:imaths}
Note that we can use $\imath_2$ instead of $\imath_1$ and repeat the above construction, but it follows from Proposition~\ref{prop:FiltMod} that there is a natural isomorphism $\wt\M:=\wt S\otimes_{\imath_1,S}\E(S)\cong \wt S\otimes_{\imath_2,S}\E(S)$ in $\wtSMF$.
\end{rmksub}


For any $\wt\M\in\wtSMF$, we define a connection $\nabla^u_{\wt\M}:\wt\M \ra \wt\M\otimes_{R_0} \wh\Omega_{R_0}$ compatible with $d^u_{\wt S}$ on $\wt S$ by 
\begin{equation}
\nabla^u_{\wt\M}:\wt\M \xra{\nabla_{\wt \M}} \wt\M\otimes_{R_0\wh\otimes\Sig} (\wh\Omega_{R_0\wh\otimes\Sig}) \thra \wt\M\otimes_{R_0} \wh\Omega_{R_0},
\end{equation}
where the second map is the natural projection.

The relation between $\nabla^u_\M$ and $\nabla^u_{\wt\M}$ is as follows:
\begin{equation}\label{eqn:ResNabla}
\nabla^u_{\wt \M}(\wt s\otimes m ) := \wt s\otimes \nabla^u_{\M}(m)+ m \otimes d^u_{\wt S}(\wt s),
\end{equation} 
for any  $m\in\M$ and $\wt s\in \wt  S$. Here we use $\imath_1$ to get $\wt\M\cong\wt S\otimes_{\imath_1,S}\M$. 
We will often work with $\nabla^u_{\wt\M}$ instead of  $\nabla_{\wt\M}$. 
\begin{rmksub}\label{rmk:HorSec} 
By abuse of notation, let $\imath_1:\M \ra \wt S\otimes_{\imath_1,S}\M=:\wt\M$  be the natural map, and $\imath_2:\M \ra \wt S\otimes_{\imath_2, S} \M \cong \wt \M$ the composition of the natural map and the natural isomorphism discussed in Remark~\ref{rmk:imaths}. Then while $\iota_1$ is horizontal with respect to $\nabla^u_\M$ and $\nabla^u_{\wt \M}$, one can directly see that $\imath_2(\M) = \wt\M^{\nabla^u=0}$. This observation will be used later in \S\ref{subsec:GRinftyResPf}.
\end{rmksub}

Using the embedding $\tilde\jmath:\wt S\hra \Acris(R)$  constructed in \S\ref{subsec:STilde}, we define
\begin{equation}
\wt T_{\cris}^*(\wt\M):=\Hom_{\wt S, \Fil^1,\vphi,\nabla^u}(\wt\M,\Acris(R))
\end{equation}
for $\wt\M\in\wtSMF$. Here, the morphisms are required to respect the connections $\nabla^u_{\wt\M}$ on $\wt \M$ and $\nabla$ on $\Acris(R)$ defined in \S\ref{subsec:Acris}. Note that $\wt T_{\cris}^*(\wt\M)$ is $p$-adic, and has a continuous $\GRtinfty$-action (as $\wt S\hra \Acris(R)$ is $\GRtinfty$-invariant).

For any $\M\in\SMF$, there is a $p$-divisible group $G'$ over $R$ such that $\M[\ivtd p]\cong \M^*(G')[\ivtd p]:=\DD^*(G')(S)[\ivtd p]$ by Theorem~\ref{thm:BreuilClassif}. Set $\wt\M:=\wt S\otimes_{\imath_1,S}\M\in\wtSMF$. Then we have a natural isomorphism $D^*(G'):=\DD^*(G')(R_0)[\ivtd p]\cong R_0[\ivtd p]\otimes_{\wt S}\wt\M$ as objects of $\RMF$ (\emph{cf.,} Example~\ref{exa:BT}), where we view $R_0$ as an $\wt S$-algebra via $\wt S\xra\delta S \thra R_0$, and the Hodge Filtration on $R_0[\ivtd p]\otimes_{\wt S}\wt\M$ is given by the image of $\Fil^1\wt\M$. Note that the following defines a $\vphi$-compatible section of the natural projection $\wt\M[\ivtd p]\thra D^*(G')$:
\begin{equation}\label{eqn:wts}
\wt s: D^*(G') \xra{s} \M[1/p] \xra{\imath_1} \wt\M[1/p],
\end{equation}
where $s$ is as in Lemma~\ref{lem:Reduction}, and $\iota_1$ is as in Remark~\ref{rmk:HorSec}. Clearly,  $\wt s$ is a unique $\vphi$-compatible section.
\begin{lemsub}\label{lem:GalCompTilde}
We use the notation as above. There exists a natural $\GRtinfty$-equivariant isomorphism $V_p(G') \cong \wt T_{\cris}^*(\wt\M)[\ivtd p]$, and the image of $\wt T_{\cris}^*(\wt\M)$ in $V_p(G')$ is same as the image of $\Tcris^*(\M)$ in $V_p(G')$. 
\end{lemsub}
\begin{proof}
By Theorem~\ref{thm:Faltings} we have a natural $\GRR$-isomorphism $V_p(G')\cong V^*_{\cris}(D^*(G'))$. Since the section $\wt s$ (\ref{eqn:wts}) induces an isomorphism \[1\otimes\wt s:\Bcris(R)\otimes_{R_0[1/ p]}D^*(G') \riso \Bcris(R)\otimes_{\wt S} \wt\M[\ivtd p],\] one obtains a $\GRtinfty$-equivariant isomorphism $\wt T_{\cris}^*(\wt\M)[\ivtd p]\riso V_{\cris}^*(D^*(G'))\cong V_p(G')$. 
For the second assertion, $\imath_1:\M \ra \wt\M$ (\emph{cf.,} Remark~\ref{rmk:HorSec}) induces an $\GRtinfty$-equivariant map $\wt T_{\cris}^*(\wt\M) \ra T_{\cris}^*(\M)$, which is surjective (as $\delta:\wt\M\ra\M$ defines a section) and respects the embeddings into $V_p(G')$. The lemma now follows.
\end{proof}

\subsection{Proof of Proposition~\ref{prop:GRinftyRes}}\label{subsec:GRinftyResPf} 
Consider $\M\in\SMF$ and $\gM\in\SMq$  such that $S\otimes_{\vphi,\Sig}\gM\cong \M$. Set $\wt\M:=\wt S\otimes_{\imath_1,S}\M$, as before. By Lemma~\ref{lem:GalCompTilde}, it suffices to construct a natural $\GRtinfty$-equivariant injective map $\cT^*(\gM)\ra \wt T_{\cris}^*(\wt\M)$ and show that it is an isomorphism when $p>2$.

The map $\imath_2:\M \ra \wt\M$ (\emph{cf.,} Remark~\ref{rmk:HorSec}) induces  an $\GRinfty$-isomorphism
\begin{equation}\label{eqn:GalCompNabla}
\Hom_{S,\Fil^1,\vphi,\vphi_1}(\M,\Acris^\nabla(R)) \cong \Hom_{\wt S, \Fil^1, \vphi, \nabla^u}(\wt\M,\Acris(R)) = \wt T_{\cris}^*(\wt\M).
\end{equation}
This is an isomorphism because $\iota_2(\M) = (\wt\M)^{\nabla^u=0}$. Since we have $S\otimes_{\vphi,\Sig}\gM = \M$, we obtain a $\GRtinfty$-equivariant map 
\begin{equation}\label{eqn:GalCompSig}
\cT^*(\gM)=\Hom_{\Sig,\vphi}(\gM,\wh\Sig^{\ur})\ra \Hom_{S,\Fil^1,\vphi,\vphi_1}(\M,\Acris^\nabla(R)),
\end{equation}
where the arrow is induced by $\vphi:\wh\Sig^{\ur}\ra\Acris^\nabla(R)$. This map is clearly injective. Combining (\ref{eqn:GalCompNabla}) and (\ref{eqn:GalCompSig}), we obtain a $\GRtinfty$-equivariant injective map $\cT^*(\gM)\ra  \wt T_{\cris}^*(\M)$.

Let us show that the map (\ref{eqn:GalCompSig}) is an isomorphism when $p>2$. We follow the same strategy as in \cite[Theorem~2.2.7]{kisin:fcrys}. It suffices to show that the following map, obtained by reducing the map (\ref{eqn:GalCompSig}) modulo~$p$, is an isomorphism when $p>2$:
 \begin{equation}\label{eqn:GalCompModp}
\Hom_{\Sig,\vphi}(\gM/(p),\wh\Sig^{\ur}/(p))\ra \Hom_{S,\Fil^1,\vphi,\vphi_1}(\M/(p),\Acris^\nabla(R)/(p)).
\end{equation}
Since both sides are finite-dimensional $\Fp$-vector spaces of same dimension, it suffices to prove injectivity. Note that $\Acris^\nabla(R)/(p)$ is naturally isomorphic to the divided power envelope of $\ol R^\flat$ with respect to $\wt\varpi^e$, so it easily follows that the kernel of $\vphi:\wh\Sig^{\ur}_R/(p) \ra \Acris^\nabla(R)$ is principally generated by $u^e$.  Therefore, $x:\gM/(p) \ra \wh\Sig^{\ur}/(p)$ is in the kernel of (\ref{eqn:GalCompModp}) if and only if the image of $x$ is contained in $u^e\wh\Sig^{\ur}$. 

Recall that $(1\otimes\vphi_\gM)(\gM/(p))$ contains $u^e(\gM/(p)) = \PP(u) (\gM/(p))$.
Since $x$ commutes with $\vphi$'s, we have 
\begin{equation}
u^e\tim x(m)= x(u^e m) = x((1\otimes\vphi_\gM)(m')) \in \vphi(u^e)\tim\wh\Sig^{\ur}/(p) = u^{pe}\tim\wh\Sig^{\ur}/(p) ,
\end{equation}
where $m\in\gM/(p)$ and $m'\in \vphi^*(\gM/(p))$ such that $(1\otimes\vphi_\gM)(m') = u^em$. So we have $x(\gM/(p)) \subseteq u^{(p-1)e}\tim\wh\Sig^{\ur}_R/(p)$. If $p>2$, then we can iterate this process to get $x=0$, which implies that the map (\ref{eqn:GalCompModp}) is injective. This concludes the proof of Proposition~\ref{prop:GRinftyRes} when $p>2$.

\begin{rmksub}
The map (\ref{eqn:GalCompSig}) is not necessarily an isomorphism when $p=2$. For example, set $\gM:=\gM^*(\Gmhat)$ and $\M:=\M^*(\Gmhat)$. Note that we can find a $\Sig$-basis $\e\in\gM$ such that $\vphi_\gM(\e) = \PP(u) \e$. Then one can show that the image of $\cT^*(\gM) $ in $\Tcris^*(\M)$ is  $2\tim\Tcris^*(\M)$ by the same proof as \cite[Proposition~5.4(2)]{Kim:ClassifFFGpSchOver2AdicDVR}. On the other hand, it is not difficult to show that $T_p(\Gmhat) = 2\tim \Tcris^*(\M)$, so we nonetheless have $T_p(\Gmhat) = \cT^*(\gM)$. We can show the same statement when $G'$ is a multiplicative-type $p$-divisible group.
\end{rmksub}

It remains to show that the $\Zp$-lattices $T_p(G)$ and $\cT^*(\gM^*(G))$ in $V_p(G)$ coincide when $p=2$. Recall that there exists a (unique) $\vphi$-equivariant map $R_0\ra W(\Frac(R/\varpi)^{\perf})$ (\emph{cf.,} \S\ref{subsec:BC} (Ex6)). Set $\fo_{K'}:= W(\Frac(R/\varpi)^{\perf})\otimes_{W}\fo_K$, which is an $p$-adic discrete valuation ring with perfetct residue field. As discussed in \S\ref{subsec:BCGal}, applying the base change $R\ra \fo_{K'}$ does not change the $\Qp$-vector space $V_0(G)$ and the $\Zp$-lattices $T_p(G)$ and $\cT^*(\gM^*(G))$. So in order to show $T_p(G)=\cT^*(\gM^*(G))$, we may assume that $R=\fo_{K'}$ with perfect residue field. But this is known already, independently in \cite{Lau:GalRep} and in \cite{Kim:ClassifFFGpSchOver2AdicDVR}. This concludes the proof of Proposition~\ref{prop:GRinftyRes} when $p=2$.

\subsection{Full faithfulness in the  $p=2$ case}\label{subsec:p2}
Let $p=2$. Then we have a functor
\[ \{p\text{-divisible groups}/R\}[1/p] \ra \SM[1/p];\quad G\rightsquigarrow \gM^*(G)[1/p]\]
by Corollary~\ref{cor:RelKisin}.

Using Proposition~\ref{prop:GRinftyRes}, we can improve  Corollary~\ref{cor:RelKisin} in the the following case:
\begin{corsub}\label{cor:gMFF}
Let $p=2$, and assume that $R$ satisfies the formally finite-type assumption~(\S\ref{cond:dJ}), and for some Cohen subring $W\subset R_0$ we have $E(u)\in W[u]$ (\emph{cf.} \S\ref{sec:EtPhiMod}).
Then for a $p$-divisible group $G$ over $R$, there exists a \emph{unique} $\vphi(\Sig)$-lattice $\gM\subset \DD^*(G)(S)$ with $\gM\in\SM$. Furthermore, the assignment $G\rightsquigarrow \gM$ defines a fully faithful functor:
\[
\gM^*:\{p\text{-divisible groups}/R\} \ra \SM.
\]
\end{corsub}
\begin{proof}
For a $p$-divisible group $G$ over $R$, we write $\M:=\DD^*(G)(S)$ and llet  $\gM,\gM'\subset \M$
be $\vphi(\Sig)$-lattices  which are in $\SM$, which exist by Proposition~\ref{prop:CL}.   We want to show that  $\gM = \gM'$. Since this  assertion can be proved Zariski-locally, we may assume that $R$ satisfies the ``refined almost \'etaleness'' assumption~(\S\ref{cond:RAF}). Note that $\gM[\ivtd p] = \gM'[\ivtd p]$ in $\M[\ivtd p]$ by Lemma~\ref{lem:CL-FF}, and we have $T_p(G) =\cT^*(\gM) = \cT^*(\gM')$ inside $T^*_{\cris}(\M)$ by Proposition~\ref{prop:GRinftyRes}. Therefore we have $\fo_\Eps\otimes_\Sig\gM = \fo_\Eps\otimes_\Sig\gM'$ inside $\fo_\Eps\otimes_\Sig\gM[\ivtd p]$ by Proposition~\ref{prop:Katz}, which forces $\gM = \gM '$.

Now let $\gM^*(G)\in \SM$ be the unique $\vphi(\Sig)$-lattice in $\M = \DD^*(G)(S)$, and we show that the assignment $\gM:G\rightsquigarrow \gM^*(G)$ is functorial; in other words,  for any morphism $f:G\ra G'$ of $p$-divisible groups over $R$, the map $\DD^*(f)(S)$ restricts to $\gM^*(G')\ra\gM^*(G)$. Since this  assertion can be proved Zariski-locally, we may assume that $R$ satisfies the ``refined almost \'etaleness'' assumption~(\S\ref{cond:RAF}). 
By Lemma~\ref{lem:CL-FF}, there is a map $\alpha:\gM^*(G')\ra\gM^*(G)$ in $\SM$ which is the restriction of $p^n\DD^*(f)(S)$. Then, the image of $\cT^*(\alpha) :\cT^*(\gM^*(G))\ra\cT^*(\gM^*(G'))$ is divisible by $p^n$ since $\cT^*(\alpha) = p^nT_p(f)$ by Proposition~\ref{prop:GRinftyRes}. This shows that $\ivtd{p^n}\alpha:\gM^*(G')\ra\gM^*(G)$ is well-defined by Proposition~\ref{prop:Katz}. We set $\gM^*(f):=\ivtd{p^n}\alpha$, which is the restriction of $\DD^*(f)(S)$.

Let us now show the full faithfulness of $\gM$.
For $\gM:=\gM^*(G)$ and $\gM':=\gM^*(G')$ where $G$ and $G'$ are $p$-divisible groups over $R$, consider a map $\alpha:\gM'\ra \gM$ in $\SMqw$. Since we already showed that $\gM^*$ is fully faithful up to isogeny (as $\M^*$ is so by the proof of Theorem~\ref{thm:BreuilClassif}), there is a homomorphism $f':G\ra G'$ with $\gM^*(f') = p^n\alpha$ for some $n$. Applying the base change $R\ra \wh R_{(\varpi)}$ (\emph{cf.,} \S\ref{subsec:BC} (Ex5)), it follows from Proposition~\ref{prop:GRinftyRes} that $f'$ kills $G[p^n]_{R_{(\varpi)}}$. (We can apply Proposition~\ref{prop:GRinftyRes} for the base ring $\wh R_{(\varpi)}$ as it is a $p$-adic discrete valuation ring.) So the kernel of $f':G[p^n]\ra G'[p^n]$ is some finite (not necessarily flat) subgroup of $G[p^n]$ containing $G[p^n]_{R_{(\varpi)}}$, so it should be equal to $G[p^n]$.  
It now follows that there is a homomorphism $f:G\ra G'$ such that $p^nf = f'$. Clearly, $\gM^*(f)=\alpha$. Faithfulness of $\gM^*$ is proved very similarly.
\end{proof}
%

\section{Finite locally free group schemes}\label{sec:finfl}
In this section, we state some results on classifications  of finite locally free\footnote{Note that we always assume our finite locally free group schemes are commutative. We use ``finite locally free group schemes'' instead of more popular terminology ``finite flat group schemes'' as it is essential to deal with possibly non-noetherian base rings in the proof.} group schemes via torsion versions of Breuil $S$-modules and Kisin $\Sig$-modules; \emph{cf.} Theorem~\ref{thm:RelKisinFF}. The proof will be given in \S\ref{subsec:RelKisinFF}. 

Throughout this section, we let $R$ be a $p$-adic ring satisfying the $p$-basis assumption~(\S\ref{cond:Breuil}) (possibly with stronger assumptions on $R$, which will be specified), with the choice of $(\Sig,\vphi)$, $E(u)$, and $S$ as in \S\ref{subsec:settingBr}. 
\subsection{Torsion filtered $S$-modules}
The following definition is the torsion version of Breuil $S$-modules. \begin{defnsub}\label{def:torBrMod}
Let $\SMFqtor$ denote the category of tuples $(\M, \Fil^1\M, \vphi_1)$ where 
\begin{enumerate}
\item\label{def:torBrMod:M} $\M$ is a finitely generated $S$-module which is killed by some power of $p$, 
and $\Fil^1\M\subset \M$ is an $S$-submodule containing $(\Fil^1S)\M$;
\item\label{def:torBrMod:phi} $\vphi_1:\Fil^1\M\ra M$ is a $\vphi$-linear endomorphism of $\M$ such that $\vphi_1(\M)$ generates $\M$ and for any $s\in\Fil^1S$ and $m\in\M$ we have 
\[\vphi_1(sm) = \frac{\vphi_1(s)}{\vphi_1(\PP)} \vphi_1(\PP m).\]
\end{enumerate}
Note that $\vphi_1(\PP)\in S\starr$. We define $\vphi:\M\ra \M$ by 
\[\vphi_\M(m):=\vphi_1(\PP)\iv \vphi_1(\PP m)\] 
for any $m\in\M$. Then we have $\vphi_\M = p\vphi_1$ on $\Fil^1\M$.

Let  $\SMFtor$ denote the category of tuples $(\M, \Fil^1\M, \vphi_1,\nabla_\M)$ where $(\M, \Fil^1\M, \vphi_1)\in\SMFqtor$ and $\nabla_\M:\M\ra\M\otimes_{\Sig} \wh\Omega_{\Sig}$ is a quasi-nilpotent integrable connection which commutes with $\vphi_\M$.

Let $\SMFqwtor$ denote the category of tuples $(\M, \Fil^1\M, \vphi_1,\nabla_{\M_0})$ where $(\M, \Fil^1\M, \vphi_1)\in\SMFqtor$ and $\nabla_{\M_0}:\M_0\ra\M_0\otimes_{R_0} \wh\Omega_{R_0}$ is a quasi-nilpotent integrable connection on $\M_0:=R_0\otimes_S\M$ which commutes with $\vphi_{\M_0}$. 
\end{defnsub}

The proof of Lemma~\ref{lem:forgetfulN} shows that  the natural functor $\SMFtor\ra\SMFqwtor$ (defined in the same way as in Remark~\ref{rmk:BrModAlt}) is fully faithful. 

Not all objects  $\M\in\SMFtor$ are expected to come from a finite locally free group schemes (by the recipe as in the proof of Proposition~\ref{prop:RelKisinFF}), as the definition of $\SMFtor$ does not prevent $\M$ from being a $R_0$-module. For the purpose of studying the $p$-power order finite locally free group schemes, it is much more convenient to use torsion Kisin modules, which we introduce in Definition~\ref{def:torKisMod}.

\subsection{Classification via torsion Kisin modules}
The following definition is the torsion version of  Kisin $\Sig$-modules. 

\begin{defnsub}\label{def:torKisMod}
Let $\SMqtor$  denote the category of  pairs $(\gM, \vphi_\gM)$, where
\begin{enumerate}
\item $\gM$ is a finitely presented  $\Sig$-module killed by some power of $p$, and of $\Sig$-projective dimension $\leqs1$;
\item $\vphi_\gM:\gM\ra\gM$ is a $\vphi$-linear map such that $\coker(1\otimes\vphi_\gM)$ is killed by $\PP$.
\end{enumerate}

Let $\SMtor$ denote the category of pairs $(\gM,\nabla_\M)$ where $\gM\in\SMqtor$ and $\nabla_{\M}:\M\ra\M\otimes_{\Sig}\wh\Omega_{\Sig}$ on $\M:=S\otimes_{\vphi,\Sig}\gM$ is a quasi-nilpotent integrable connection  which commutes with $\vphi_{\M} := \vphi_{S}\otimes\vphi_\gM$. 

Let $\SMqwtor$ denote the category of pairs $(\gM,\nabla_{\M_0})$ where $\gM\in\SMqtor$ and $\nabla_{\M_0}:\M_0\ra\M_0\otimes_{R_0} \wh\Omega_{R_0}$ is  a quasi-nilpotent integrable  connection on $\M_0:=R_0\otimes_{\vphi,\Sig}\gM$ which commutes with $\vphi_{\M_0}$.   We call an object in $\SMqwtor$  a \emph{torsion Kisin $\Sig$-module}.

We  let  $\SMFI$, $\SMFIqw$, and $\SMFIq$ denote the full subcategories of $\gM$ such that $\gM\cong \bigoplus\gM_i$ as a $\Sig$-module where $\gM_i$ are projective over $\Sig/(p^{n_i})$.
\end{defnsub}

\begin{rmksub}\label{rmk:torKisModSm}
As before, one can define a natural functor $\SMtor \ra \SMqwtor$ by the same way as in Remark~\ref{rmk:BrModAlt}, and the same proof of Lemma~\ref{lem:forgetfulN} shows that this is fully faithful. We will see later that this functor is an equivalence of categories if $R$ satisfies the formally finite-type assumption~(\S\ref{cond:dJ}) for any $p$ (including $p=2$) -- indeed, we show that both categories are naturally anti-equivalent to the category of $p$-power order finite locally free group schemes over $R$; \emph{cf.,} Theorem~\ref{thm:RelKisinFF}.
\end{rmksub}

Let $\gM$ be an object of either $\SMqtor$ or $\SMtor$, and set $\M:=S\otimes_{\vphi,\Sig}\gM$.
One can define $\Fil^1\M\subset \M$ and $\vphi_1:\Fil^1\M\ra\M$ using the same formula for  (quasi-)Kisin modules (\emph{cf.,} (\ref{eqn:Filh}), (\ref{eqn:vphir})). 
Clearly, $\M$ with this structure is an object in $\SMFqtor$ (or $\SMFtor$ if $\gM\in\SMtor$).  

\begin{rmksub}\label{rmk:VerFF}
For $\gM\in\SMqtor$ there exists a unique (injective) $\Sig$-linear map $\psi_\gM:\gM\ra\vphi^*\gM$ such that $\psi_\gM\circ(1\otimes\vphi_\gM) = E(u)\id_{\vphi^*\gM}$ and $(1\otimes\vphi_\gM)\circ\psi_\gM = E(u)\id_\gM$. Indeed, the multiplications  by $E(u)$ on $\gM$ and $\vphi^*\gM$ are injective, and the image of $1\otimes\vphi_\gM$ contains $E(u)\gM$. 

Similarly, if $\M=S\otimes_{\vphi,\Sig}\gM$ for some $\gM\in\SMqtor$, then we set 
\[\psi_\M = \vphi_1(E(u))\iv(1\otimes\psi_\gM):\M\ra\vphi^*\M.\] 
Then we have $\psi_\M\circ(1\otimes\vphi_\M) = p\id_{\vphi^*\M}$ and $(1\otimes\vphi_\M)\circ\psi_\M = p\id_\M$. 
\end{rmksub}

\begin{prop}\label{prop:RelKisinFF}
Assume that the functor $\gM^*$  on $p$-divisible groups over $R$ is fully faithful. (\emph{cf.,} Corollaries~\ref{cor:RelKisin} and \ref{cor:gMFF}). 
Then there exists an exact fully faithful functor $\gM^*$ from the category of $p$-power order finite locally free group schemes over $R$ to $\SMtor$ with the following properties:
\begin{enumerate}
\item For any $p$-power order finite locally free group scheme $H$ over $\Spf(R,(\varpi))$, there exists a natural horizontal isomorphism $S\otimes_{\vphi,\Sig}\gM^*(H)\cong\DD^*(H)(S)$ which matches   the quotient $e_H^*\Omega_{H/R}$ of $\DD^*(H)(S)$ with 
\[
\coker (1\otimes\vphi_{\gM^*(H)}) \cong  (S\otimes_{\vphi,\Sig}\gM^*(H))/\Fil^1 (S\otimes_{\vphi,\Sig}\gM^*(H)),
\]
 and  $(F,V)$ on $\DD^*(H)(S)$ with $(1\otimes\vphi, \psi)$ on $\M^*(H)$ (\emph{cf.,} Definition~\ref{def:torBrMod}, Remark~\ref{rmk:VerFF}). 
\item If $H=\ker[d:G^0\ra G^1]$ for some isogeny $d$, then there exists a natural isomorphism $\gM^*(H)\cong \coker[\gM^*(d)]$.
\item The formation of $\gM^*$ commutes with a base change by $R\ra R'$ that lifts to some $\vphi$-equivariant map $\Sig\ra\Sig'$, where $\Sig' = R_0'[[u]]$ and $\vphi$ are chosen as in \S\ref{cond:Breuil} and \S\ref{subsec:settingBr}; \emph{cf.} \S\ref{subsec:BC}.
\end{enumerate}
\end{prop}
As a preparation for the proof, we recall the following theorem of Raynaud  \cite[Th\'eor\`eme~3.1.1]{Berthelot-Breen-Messing:DieudonneII}:
\begin{thmsub}\label{thm:Raynaud}
Let  $H$ be a commutative finite locally free group scheme over some scheme $X$. Then there exists a Zariski covering $\{U_\alpha\}_\alpha$ of $X$ and an abelian scheme $\mathcal{A}_\alpha$ over each $U_\alpha$ such that $H_{U_{\alpha}}$ can be embedded in $\mathcal{A}_\alpha$ as a closed subgroup.
\end{thmsub}
Now assume that $R$ is $p$-adic, and let $H$ be a finite locally free group scheme over $R$ of $p$-power order. (Note that there is no difference in viewing $H$ either over $\Spf (R,(p))$ or over $\Spec R$.) We apply Raynaud's theorem for $H$ over  $\Spec R$ to obtain an open affine cover $\{U_\alpha=\Spec R_{\alpha}\}$ and an abelian scheme $\mathcal{A}_\alpha$ over each $U_\alpha$. So we obtain a $p$-divisible group $G_\alpha:=\mathcal{A}_\alpha[p^\infty]$ over the $p$-adic completion $\wh R_\alpha$ of $R_\alpha$ into which $H_{\wh R_\alpha}$ can be embedded. (Note that $\{\Spf(\wh R_\alpha, (p))\}$ is a Zariski open covering of $\Spf(R,(p))$.) The same discussion holds if we work with $J_R$-adic topology if $R$ is $J_R$-adically separated and complete.

We need another lemma for  the proof of Proposition~\ref{prop:RelKisinFF}. Let us consider a  triangulated full subcategory $\Isog^\bullet_{R}$ of the derived category of fppf sheaves of abelian groups over $R$, generated by two-term complexes $G^\bullet:=[G^0\xra{d} G^1]$ at degree $[0,1]$ such that $G^i$ are $p$-divisible groups and $d$ is an isogeny. Note that the only non-zero cohomology of $G^\bullet\in\Isog^\bullet_R$ is $\coh 0(G^\bullet)$, which is a finite locally free group scheme over $R$ with $p$-power order.

\begin{lemsub}\label{lem:glueing}
The functor $\coh{0}:[G^0\xra{d} G^1]\rightsquigarrow \ker(d)$ from $\Isog^\bullet_R$ to the category of $p$-power order finite locally free group schemes over $R$ is fully faithful.
\end{lemsub}
\begin{proof}
Let $G^\bullet:=[G^0\xra{d}G^1]$ and $G^{\prime\bullet}:=[G^{\prime0}\xra{d'}G^{\prime1}]$ be objects in $\Isog^\bullet_R$, and set $H:=\coh{0}(G^\bullet)$ and $H':=\coh{0}(G^{\prime\bullet})$. 

Let $f:G^\bullet\ra G^{\prime\bullet}$ be a morphism such that $\coh 0(f)$ is a zero map, and we seek to show that $f$ is a zero map. This can be checked after composing with any isomorphism, so we may assume, without loss of generality, that $f$ is given by morphisms $f^i:G^i\ra G^{\prime i}$ of $p$-divisible groups.  If $\coh 0(f):H\ra H'$ is a zero map, then $f^0$ factors through $G^1$ giving a null homotopy for $f$. Now it easily follows that $\coh 0$ is faithful.

If $f:H\ra H'$ be any homomorphism, then we can factor $f$ into $H\xra{(1,f)} H\times H' \thra H'$. Since $H\times H':=\ker[(d,d'):G^0\times G^{\prime0} \ra G^1\times G^{\prime1}]$, we may assume that $f$ is a closed immersion by replacing $f$ by $(1,f)$ if necessary. Set $G^{\prime\prime\bullet}:= [ G^{\prime0} \ra G^{\prime0}/f(H)]$. Then we have $H = \coh{0} [ G^{\prime\prime\bullet}]$ and $f$ comes from $g:G^{\prime\prime\bullet}\ra G^{\prime\bullet}$ where $g^0=\id$ and $g^1$ is the natural projection.
\end{proof}

\begin{proof}[{Proof of Proposition~\ref{prop:RelKisinFF}/Construction}]
Let us first construct $\gM^*$. For any isogeny $d:G^0\ra G^1$ of $p$-divisible groups over $R$, $\gM:=\coker [\gM^*(d):\gM^*(G^1)\ra \gM^*(G^0)]$ is clearly a torsion Kisin $\Sig$-module. By Corollary~\ref{cor:RelKisin} (Corollary~\ref{cor:gMFF} when $p=2$) and Lemma~\ref{lem:glueing}, $\coh{0}\circ \gM^*$ induces a fully faithful functor $\Isog^\bullet_R\ra \SMtor$ which commutes with \'etale base change $\Spf(R',(\varpi))\ra \Spf(R,(\varpi))$. If $H:=\ker(d)$, then we set $\gM^*(H):=\coker [\gM^*(d)]$.

Let $H$ be a finite locally free group scheme over $R$ of $p$-power order. Then by Raynaud's theorem (Theorem~\ref{thm:Raynaud} and the subsequent discussion) there exists a Zariski covering $\{\Spf (R_\alpha, (p))\}_\alpha$ of $\Spf(R,(p))$ and $p$-divisible groups $G_\alpha$ over $R_\alpha$ such that $H_{R_\alpha}\subset G_\alpha$. 

Recall that for each $R_\alpha$ there is a natural $\Sig$-algebra $\Sig_\alpha$ together with a unique choice of $\vphi_{\Sig_\alpha}$ over $\vphi_\Sig$, such that $\Sig_\alpha/(\PP)=R_\alpha$. (\emph{Cf.} Lemma~\ref{lem:etloc}, Remark~\ref{rmk:RelKisinBC}.) Note also that $\{\Spf(\Sig_\alpha,(p,\PP))\}$ is a Zariski covering of $\Spf(\Sig,(p,\PP))$. Now, $\gM^*(H_{R_\alpha})\in\PMtor{\Sig_\alpha}$ makes sense, and the natural glueing datum on $\{H_{R_\alpha}\}$ induces a glueing datum on  $\{\gM^*(H_{R_\alpha})\}$. 
We define $\gM^*(H)\in\SMtor$ by glueing $\{\gM^*(H_{R_\alpha})\}$. 
%
\end{proof}

For $\gM\in\SMq$, we say $\gM$ is \emph{$\vphi$-nilpotent} if $\vphi_\gM^n=0$ for some $n$, and similarly define $\psi$-nilpotent objects. (\emph{Cf.} Definition~\ref{def:vphiNilp}.) 
We can improve Proposition~\ref{prop:RelKisinFF} in the following case: 
\begin{thm}\label{thm:RelKisinFF}
Assume that $R$ satisfies the formally finite-type assumption~(\S\ref{cond:dJ}). If $p>2$ then the functor $\SMFI \ra \SMFIqw$, defined in Remark~\ref{rmk:torKisModSm}, is an equivalence of categories, and the  following functors is an anti-equivalence of categories:
\[
\left\{\begin{aligned}
&p\text{-power order finite locally}\\
&\text{free group schemes }H\text{ over }R\\
&\text{such that }H[p^n]\text{ is locally free }\forall n
\end{aligned}\right\} 
\xra{\gM^*}\SMFI \riso \SMFIqw.\]

Let  $R_0=W(k)[[T_1,\cdots, T_d]]$ with perfect  field $k$ of characteristic $p>2$, and define $\vphi:R_0\ra R_0$ by extending the Witt vector Frobenius by $\vphi(T_i)=T_i^p$. Set $R = R_0[[u]]/E(u)$. Then $\gM^*$ composed with the forgetful functor $\SMtor\ra\SMqtor$ is an anti-equivalence of categories.

Allowing any prime $p$, each of the following functors is an anti-equivalence of categories:
\begin{gather*}
\left\{\begin{aligned}
&\text{infinitesimal finite locally}\\
&\text{free group schemes over }R
\end{aligned}\right\} 
\xra{\gM^*}\SM^{\vphi\nilp} \riso \SMq^{\vphi\nilp};\\
\left\{\begin{aligned}
&\text{unipotent finite locally}\\
&\text{free group schemes over }R
\end{aligned}\right\} 
\xra{\gM^*}\SM^{\psi\nilp} \riso \SMq^{\psi\nilp},
\end{gather*}
where $\gM^*$ is as in Proposition~\ref{prop:RelKisinFF}, the other arrows are suitable forgetful functors, and the superscripts $^{\vphi\nilp}$ and $^{\psi\nilp}$ respectively denote the full subcategory of $\vphi$-nilpotent and $\psi$-nilpotent objects.
\end{thm}

We prove this theorem later in \S\ref{subsec:RelKisinFF}. 
\begin{rmksub}
Let us review previous results on Theorem~\ref{thm:RelKisinFF}. When $R$ is a $p$-adic discrete valuation ring with perfect residue field and $p>2$, then Theorem~\ref{thm:RelKisinFF} is a theorem of Kisin \cite[Theorem~2.3.5]{kisin:fcrys}. (Note that in this case, $\wh\Omega_{R_0} = 0$ so $\nabla^0$ is forced to be zero.) The strategy of Kisin  to deduce Theorem~\ref{thm:RelKisinFF} from the classification of $p$-divisible groups can be applied in more general setting. Most general result of this type is the theorem of Eike~Lau \cite[Corollary~7.7, Proposition~8.1]{Lau:2010fk} in the case when $R$ is a complete regular local ring with perfect residue field. (Note that his result does not require $R/(\varpi)$ to be a formal power series.)

Unfortunately, there are difficulties in using the strategy of Kisin to deduce Theorem~\ref{thm:RelKisinFF} from Corollary~\ref{cor:RelKisin}. The most serious problem is the presence of connections $\nabla^0$, which cannot be removed in general (unless it is asserted in Theorem~\ref{thm:RelKisinFF}). In order to handle the connections, we use the theory of moduli of connections \cite[\S3]{Vasiu:ys}. We review (and slightly generalise) the theory in \S\ref{sec:Vasiu}.
\end{rmksub}

\subsection{Galois representations}
We continue to assume that $R$  satisfies the formally finite-type assumption~(\S\ref{cond:dJ}). Let $H$ be a $p$-power order finite locally free group scheme over $R$. Then $H(\ol R)$ is a finite torsion $\Zp$-module equipped with continuous (i.e., discrete) action of $\GRR$.  
On the other hand, under the additional assumption (stated in Proposition~\ref{prop:GalRepFF}) we obtain a torsion $\GRtinfty$-representation $\cT^*(\gM^*(H))$ since $\fo_\Eps\otimes_\Sig\gM^*(H)$ is an \'etale $(\vphi,\fo_\Eps)$-module.

\begin{propsub}\label{prop:GalRepFF}
Assume that $R$ satisfies the formally finite-type assumption~(\S\ref{cond:dJ}), and for some Cohen subring $W\subset R_0$ we have $E(u)\in W[u]$ (\emph{cf.} \S\ref{sec:EtPhiMod}). 
Then there exists a natural $\GRtinfty$-equivariant isomorphism $H(\ol R)\cong \cT^*(\gM^*(H))$ with the following properties:
\begin{enumerate}
\item If $H=\ker[d:G^0\ra G^1]$ for some isogeny $d$, then the isomorphism $H(\ol R)\cong \cT^*(\gM^*(H))$ is compatible with the isomorphism $T_p(G^i)\cong \cT^*(\gM^*(G^i))$ as in Corollary~\ref{cor:GRinftyRes}.
\item The formation of the isomorphism $H(\ol R)\cong \cT^*(\gM^*(H))$  commutes (in the obvious sense)  with a base change by $R\ra R'$ that lifts to some $\vphi$-equivariant map $\Sig\ra\Sig'$, where $\Sig' = R_0'[[u]]$ and $\vphi$ are chosen as in \S\ref{cond:Breuil} and \S\ref{subsec:settingBr}; \emph{cf.} \S\ref{subsec:BC}.
\end{enumerate}
\end{propsub}
\begin{proof}
If $H=\ker[d:G^0\ra G^1]$, then the $\GRtinfty$-isomorphism $T_p(G^i)\cong \cT^*(\gM^*(G^i))$ as in Corollary~\ref{cor:GRinftyRes} induces a $\GRtinfty$-isomorphism $H(\ol R)\cong \cT^*(\gM^*(H))$ with all the desired properties. (Here we use the exact sequence (\ref{eqn:RednEtPhiMod}).) But since the formation of the isomorphism commutes with Zariski localisation, it suffices to define the isomorphism on some Zariski covering where $H$ can be embedded in some $p$-divisible group.
\end{proof}

\section{Moduli of connections}\label{sec:Vasiu}
In this section we study the role of connections using Vasiu's construction of ``moduli of connections'' \cite[\S3]{Vasiu:ys}. This construction will play an important role in the classification of finite locally free group schemes in \S\ref{sec:finfl}. The main results in this section are Corollaries~\ref{cor:BreuilClassif} and \ref{cor:dJgen}.

Throughout this section, we  assume that $R$ satisfies the formally fintie-type assumption~(\S\ref{cond:dJ}) and choose $(R_0,\vphi)$ as usual. Set $R_{0,k}:= R_0/(p)(\cong R/(\varpi))$. Let  $J_R$ and $J_0(=J_{R_0})$ respectively denote  the Jacobson radicals of $R$ and $R_0$. Note that $R$ and $R_0$ are respectively $J$-adic and $J_0$-adic.

\subsection{Preliminaries}\label{subsec:IndEtBC}
Let $I\subset R$ be a $J_R$-adically closed ideal such that $(\varpi)\subseteq I \subseteq J_R$, and set $I_0:=\ker(R_0\thra R/I)$. (In practice, $I$ is either $(\varpi)$ or $J_R$.) 
Let $A$ be a $I$-adic formally \'etale $R$-algebra such that $A/IA$ is finitely generated over $R/I$ (i.e., 
$\Spec A/I^nA\ra \Spec R/I^n$ is \'etale for any $n$). 
Then applying \S\ref{subsec:BC} (Ex1), there exists a unique $I_0$-adic formally \'etale lift $R_0\ra A_0$ of $R/I\ra A/IA$, and $A_0$ can naturally be viewed as a $R_0$-subalgebra of $A$. We set $\Sig_A:=A_0[[u]]$, and define $S_A$, etc., correspondingly. More generally, if $A$ is ``ind-\'etale'' (i.e., $A$ is a $I$-adic completion of $\varinjlim_i A\com i$ for some $I$-adic formally \'etale $R$-algebras $A\com i$) then we define $A_0$ by $I_0$-adically completing $\varinjlim A_{0}\com i$, where $A_{0}\com i$ is the $R_0$-lift of $A\com i/IA\com i$, and define $\Sig_A$, $S_A$, etc., accordingly. Note that $A_0$ is flat over $W$ and $A/(\varpi)$ locally admits a finite $p$-basis, so it satisfies the requirement for $R_0$ as in \S\ref{cond:Breuil} and $A$ satisfies the $p$-basis assumption (\S\ref{cond:Breuil}) using this $A_0$ and $E(u)\in R_0[u]\subset A_0[u]$.

\subsection{Review of the construction of moduli of connections}\label{subsec:Vasiu}
For any $\M\in\SMFq$ (respectively, $\M_0\in\PMFq{R_0}$), we have $S_A\otimes_S\M\in\PMFq{S_A}$ (respectively, $A_0\otimes_{R_0}\M\in\PMFq{A_0}$). 

Let us recall (and slightly extend)  Vasiu's construction of moduli of connections \cite[Theorem~3.2]{Vasiu:ys}. Let $\M_0,\M_0'\in \PMFq{R_0}$ and a morphism $f:\M_0\ra\M_0'$ be given. Fix an $R_0$-direct factor $(\M_0)^1\subseteq\M_0$ which lifts $\Fil^1\M_0/p\M_0\subseteq \M_0/p\M_0$, and similarly choose $(\M_0')^1\subseteq\M_0'$. Set
 \[\wt\M_0:=\vphi^*\left(\M_0+ \ivtd p (\M_0)^1\right)\subset\M_0[\ivtd p].\] 
 Note that $1\otimes\vphi_{\M_0}$ induces an $R_0$-linear isomorphism $\wt\M_0\riso \M_0$, and one can recover $\M_0\in\PMFq{R_0}$ from $(\M_0, (\M_0)^1, \wt\M_0 \underset{1\otimes\vphi_{\M_0}}{\riso}\M_0)$. We similarly define $\wt\M_0'$ so that $f$ induces $\wt\M_0\ra\wt\M_0'$.
  
By passing to a Zariski covering of $\Spf (R_0, J_0)$,  we may assume that   $(\M_0)^1$, $\M_0/(\M_0)^1$, $(\M_0')^1$, $\M_0'/(\M_0')^1$, and $\wh\Omega_{R_0}$ are all free over $R_0$. Let us fix a $R_0$-basis of $\wh\Omega_{R_0}$ and an $R_0$-basis of  $\M_0$ adapted to the direct factor $(\M_0)^1$, and same for $\M_0'$. 

Consider a functor $\cQ_n$ which associates to any \'etale map $\Spf(A_0,(p))\ra\Spf (R_0,(p))$ the set of (naturally defined) equivalence classes of additive morphisms \[\nabla_{A_0,n}:A_0/(p^n)\otimes_{R_0}\M_0 \ra A_0/(p^n)\otimes_{R_0}\M_0'\otimes_{R_0}\wh\Omega_{R_0}\] which satisfies the Leibnitz rule (i.e., $\nabla_{A_0,n}(ax) = a\nabla_{A_0,n}(x) + x\otimes da$ for any $a\in A_0/(p^n)$ and $x\in\M_0$, where we identify $\wh\Omega_{A_0}\cong A_0\wh\otimes_{R_0}\wh\Omega_{R_0}$ by \'etaleness) and makes the following diagram commute:
\begin{equation}\label{eqn:ASsys1}
\xymatrix@C=50pt{
A_0/(p^n)\otimes_{R_0}\wt\M_0 \ar[r]^-{\vphi^*(\nabla_{A_0,n})} \ar[d]_-{1\otimes\vphi_{\M_0}}& 
A_0/(p^n)\otimes_{R_0}\wt\M_0'\otimes_{R_0}\wh\Omega_{R_0} \ar[d]^-{(1\otimes\vphi_{\M_0'})\otimes\id_{\wh\Omega_{R_0}}}\\
A_0/(p^n)\otimes_{R_0}\M_0\ar[r]_-{\nabla_{A_0,n}} & 
A_0/(p^n)\otimes_{R_0}\M_0'\otimes_{R_0}\wh\Omega_{R_0}
}.\end{equation}
Here we define $\vphi^*(\nabla_{A_0,n})$ by choosing an arbitrary lift of $\nabla_{A_0,n}$ over $A_0/(p^{n+1})$, and one can check without difficulty that  $\vphi^*(\nabla_{A_0,n})$  does not depend on the choice of lift. (\emph{Cf.} equation (9) in \cite[\S3.1.1]{Vasiu:ys}.)

If $\M_0 = \M_0'$ then a connection $\nabla_{A_0,n}$ on $A_0/(p^n)\otimes_{R_0}\M_0$ commutes with $\vphi_{A_0}\otimes\vphi_{\M_0}$ if and only it satisfies (\ref{eqn:ASsys1}) modulo any $p^n$. More generally, if $\M_0, \M_0'\in\PMF{R_0}$ and $f:\M_0 \ra \M_0'$ respects $\vphi$ and takes $(\M_0)^1$ into $(\M_0')^1$, then both $\nabla_{M_0'}\circ f$ and $(f\otimes\id)\circ\nabla_{\M_0}$ satisfy (\ref{eqn:ASsys1}) modulo any $p^n$.

We claim that the functor $\cQ_n$ is representable by a $p$-adic formally \'etale $R_0$-algebra (again denoted by $\cQ_n$).  Fixing $R_0$-bases as above, the condition for $\nabla_{A_0,1}$ 
to satisfy (\ref{eqn:ASsys1}) for any $A_0$ we consider is given by a system of equations of the form $\nf x = B\vphi(\nf x) + C_0$ for some matrices $B,C_0$ with entries in $R_{0,k}=R_0/(p)$, where the variables $x$ are the ``matrix entries'' of $\nabla_{A_0,1}$ with respect to the fixed bases (\emph{Cf.} the proof of \cite[Theorem~3.2]{Vasiu:ys}.) Such a system of equations  is an example of  \emph{Artin-Schreier system of equation} (defined in \cite[\S2.4]{Vasiu:ys}), and defines a faithfully flat \'etale algebra $\cQ_{1,k}$ over $R_{0,k}$ by \cite[Theorem~2.4.1]{Vasiu:ys}, and we take $\cQ_1$ to be the unique $p$-adic formally \'etale $R_0$-algebra lifting $\cQ_{1,k}$.

Now, assume that $\cQ_n$ is defined, and let $\nabla_{\cQ_n,n}$ denote the universal object. Pick an arbitrary lift $\wt\nabla_{\cQ_n,n+1}:\cQ_n/(p^{n+1})\otimes_{R_0}\M_0 \ra \cQ_n/(p^{n+1})\otimes_{R_0}\M_0\otimes_{R_0} \wh\Omega_{R_0}$ of $\nabla_{\cQ_n,n}$. For any \'etale map $\jmath:\Spf(A_0,(p))\ra \Spf(\cQ_{n},(p))$, the condition for some $\nabla_{A_0,n+1}$ to satisfy (\ref{eqn:ASsys1}) is given an system of equations of the form $\nf x = B\vphi(\nf x) + C_n$ over $\cQ_n/(p)$, where the variables are the matrix entries of $\delta_{A_0,n+1}$ where $\nabla_{A_0,n+1} = \jmath\otimes \wt\nabla_{\cQ_n,n+1}+p^n\delta_{A_0,n+1}$. (See the proof of \cite[Theorem~3.2]{Vasiu:ys} for the details.) We take $\cQ_{n+1}$ to be  the unique $p$-adic formally \'etale $\cQ_n$-algebra lifting the faithfully flat \'etale $\cQ_{n,k}$-algebra $\cQ_{n+1,k}$  defined by these equations. The system of \'etale algebras $\{\cQ_{n,k}\}$ is an example of \emph{Artin-Schreier tower} defined in \cite[\S2.4]{Vasiu:ys}.

\begin{defnsub}\label{def:ModConn}
The system $\{\cQ_n\}$ constructed above is called the \emph{moduli of connections} for $\M_0$ and $\M_0'$.
\end{defnsub}

The following lemma states that any ``connected component'' of $\Spec \cQ_\infty/(p)$ enjoys nice properties, where $\cQ_\infty:=\varinjlim_n\cQ_n$.
\begin{lemsub}\label{lem:Artin-Schreier}
Let $\{\cQ_{n}\}$ be the $p$-adic $R_0$-formally \'etale lift of some ``Artin-Schreier tower'' over $R_{0,k}:=R_0/(p)$ in the sense of  \cite[\S2.4]{Vasiu:ys}; for example, $\{\cQ_n\}$ can be a ``moduli of connections''. Let $\wh\cQ_\infty$ denote the $p$-adic completion of $\cQ_{\infty}:=\varinjlim_n\cQ_{n}$, and choose a direct summand $A_0$ of $\wh\cQ_{\infty}$ such that $\Spec A_0$ has finitely many connected components. 
Set $A_{0,k}:=A_0/(p)$. Then the following properties hold:
\begin{enumerate}
\item\label{lem:Artin-Schreier:Sm} 
Let $\p\subset A_0$ be a prime ideal containing $p$, and set $\ol \p:=\p/(p)\subset A_{0,k}$.  Then the localisation $(A_{0,k})_{\ol\p}$ is regular and excellent,\footnote{Note that this does not necessarily imply that $A_{0,k}$ is noetherian. The author does not know whether or when to expect that  $A_{0,k}$ is noetherian.} and $A_{0,k}$ locally admit a finite $p$-basis. 
\item\label{lem:Artin-Schreier:Stratif} The image of any open subset of $\Spec A_{0,k}$ in $\Spec R_{0,k}$ is open.
\end{enumerate}
\end{lemsub}
\begin{proof}
Note that $R_0/(p)$ is regular and excellent by the theory of excellence; \emph{cf.,} \cite[Theorem~4]{Valabrega:ExcellencePS},  \cite[Lemma~1.3.3]{dejong:crysdieubyformalrigid}, and \cite[\S33.I, Theorem~79]{matsumura:CommAlg}. Set $\ol\q:=\ol\p\cap R_{0,k}$. Since $\Spec (A_{0,k})_{\ol\p}$ is an projective limit of  \'etale neighbourhoods of $\ol\q\in\Spec(R_{0,k})_{\ol\q}$, we obtain a (faithfully) flat local morphism $(A_{0,k})_{\ol\p} \ra (R_{0,k})_{\ol\q}^{\sh}$, where $(R_{0,k})_{\ol\q}^{\sh}$ is a strict henselisation at $\ol\q$. Now it follows that $(A_{0,k})_{\ol\p}$ satisfies ascending chain condition. (Recall that the strict henselisation of a noetherian local ring is noetherian \cite[IV{$_4$, Proposition~18.8.8(iv)}]{EGA}.)

Since the strict henselisation of $(A_{0,k})_{\ol\p}$ is naturally isomorphic to $(R_{0,k})_{\ol\q}^{\sh}$, $A_{0,k}$  is  regular by \cite[IV{$_4$}, Corollaire~18.8.13]{EGA}. Excellence of $A_{0,k}$ follows from \cite[Corollary~5.6(iv)]{Greco:ThmExcellence} since $R_{0,k}$ is excellent and normal.

Recall that $R_{0,k}$ locally admits a finite $p$-basis by \cite[Lemmas~1.1.3, 1.3.3]{dejong:crysdieubyformalrigid}. So  $A_{0,k}$ locally admits a finite $p$-basis since it is a direct limit of \'etale $R_{0,k}$-algebras.
 This proves (\ref{lem:Artin-Schreier:Sm}).

The assertion (\ref{lem:Artin-Schreier:Stratif}) is a consequence of  \cite[Theorem~2.4.1(c)]{Vasiu:ys}, which asserts that there exists a ``stratification'' of $\Spec R_{0,k}$ such that over each stratum $\Spec \cQ_{n,k}$ is finite \'etale cover. Therefore, the map $\Spec A_{0,k} \ra \Spec R_{0,k}$ is open and closed over each stratum (unless the fibre at the stratum is empty) by \cite[Theorem~9.6]{matsumura:crt}, from which (\ref{lem:Artin-Schreier:Stratif})  follows.
%
\end{proof}

Let us specialise to the moduli of connections $\{\cQ_n\}$ for $\M_0=\M_0'$. Under the notation as in Lemma~\ref{lem:Artin-Schreier}, let $\wh\cQ_\infty$ denote the $p$-adic completion of $\cQ_\infty:=\varinjlim\cQ_n$, and choose a quotient $A_0$ of $\wh \cQ_\infty$ which is $p$-adic and formally \'etale over $R_0$ such that $A_{0,k}:=A_0/(p)$ is constructed as in Lemma~\ref{lem:Artin-Schreier}.
\begin{propsub}\label{prop:Vasiu}
In the above setting, $A_0\otimes_{R_0}\M_0$ together with the universal connection is an object in $\PMF{A_0}$. 
\end{propsub}
\begin{proof}
The universal connection is integrable and topologically quasi-nilpotent, which is shown in the last three paragraphs in the proof of \cite[Theorem~3.2]{Vasiu:ys}. 
\end{proof}

\subsection{Special cases}
There are some special cases where the moduli of connection $\{\cQ_\infty\}$ turns out to be very close to $R_0$. We begin with some definitions:

Let $\{\cQ_n\}$ be the moduli of connections for $\M_0, \M_0'\in\PMFq{R_0}$, and consider one of the following cases.
\begin{description}
\item[Formal/Unipotent case]  $\M_0$ and $\M_0'$ are either both $\vphi$-nilpotent or both $\psi$-nilpotent (in the sense of Definition~\ref{def:vphiNilp}).
\item[Local case] Assume that $k$ is perfect, $R_0=W(k)[[T_1,\cdots,T_d]]$, and $\vphi$ satisfies $d\vphi(T_i)/p\in \m_{R_0}\wh\Omega_{R_0}$ where $\m_{R_0}$ is the maximal ideal (e.g.,  $\vphi(T_i) = T_i^p$). We allow $\M_0$ and $\M_0'$ to be arbitrary.
\end{description}
The proof of \cite[Theorem~3.3.1]{Vasiu:ys} shows that in the above cases there exists a \emph{unique} geometric point of $\Spec \cQ_n/(p)$ over any geometric closed point of $\Spec R_{0,k}$. Set $J_0:=\ker (R_0\thra R_{\red})$ as before. 
From this, one can check that 
there exists a unique connected component $\Spf (A_0\com n,(p))\subset \Spf(\cQ_n,(p))$ for each $n$ which intersects with the fibre over some closed point in $\Spec R_{0,k}$, and furthermore we have $R_0\riso A_0\com n$ for all $n$. (Indeed,  all closed points lie in $\Spec R/J_R$, so if we set $J_0:=\ker(R_0\thra R_{\red})$ then $R_0/J_0 \ra \cQ_nJ_0\cQ_n$ is an isomorphism for all $n$ because it is \'etale and radicial covering. Now exploit the $J_0$-adic completeness of $R_0$.) 
In particular, it follows that for any $\M_0, \M_0'\in\PMFq{R_0}$ in one of the above cases, one can uniquely define a connection on each of $\M_0$ and $\M_0'$ so that they become an object of $\PMF{R_0}$ by Proposition~\ref{prop:Vasiu}, and any $\vphi$-equivariant map $\M_0\ra\M_0'$ is horizontal. 
Combining this with Theorems~\ref{thm:BreuilClassif} and \ref{thm:FormalBreuil}, we obtain:
\begin{corsub}\label{cor:BreuilClassif}
In the \textbf{Local case} above with $p>2$, the functors $\gM^*$ and $\M^*$ below
\[
\xymatrix@C=40pt{
\{p\text{-divisible groups}/R\}  \ar@/^1.5pc/[rr]^-{\gM^*} \ar[r]^-{\M^*} &\SMFq &\SMq \ar[l]^{S\otimes_{\vphi,\Sig}(-)}_-{\cong}
},
\]
defined by forgetting the connections,  are  anti-equivalences of categories.

Let $p$ be any prime. Then $\gM^*$ and $\M^*$ induce anti-equivalences of categories from the category of formal $p$-divisible groups over $R$ to $\SMq^{\vphi\nilp}$ and $\SMFq^{\vphi\nilp}$, and similarly for unipotent $p$-divisible groups.  
\end{corsub}

\begin{rmksub}
Corollary~\ref{cor:BreuilClassif} can be alternatively obtained using display theory for formal $p$-divisible groups and Dieudonn\'e display theory (even in the \textbf{Local case}  when $p=2$); \emph{cf.} \cite{Zink:DisplayFormalGpAsterisq278, Lau:DisplayFormalBT} and \cite{Lau:2010fk}. Indeed, the ``display'' corresponding to a $p$-divisible group can be recovered from $\SMFq$ (without connection).
 \end{rmksub}

Let us now state the main result of this section:
\begin{thm}\label{thm:dJgen}
Let $A_0$ be a finite product of $p$-adic $\Zp$-flat domains which satisfies the conclusions of Lemma~\ref{lem:Artin-Schreier}(\ref{lem:Artin-Schreier:Sm}); namely, $A_{0,k}:=A_0/(p)$ locally admits a finite $p$-basis, and for any prime ideal $\p\subset A_0$ containing $p$, $(A_0)_\p/(p)$ is regular and excellent. We also fix a lift of Frobenius $\vphi=\vphi_{A_0}:A_0\ra A_0$ over $\vphi_W$, which exists by Lemma~\ref{lem:lifting}.

Then, the contravariant functor  
\begin{equation*}
\M_0^*:\big\{p\text{-divisible groups over }A_{0,k}:=A_0/(p) \big\} \ra \PMF{A_0},
\end{equation*}
given by $G_{A_{0,k}}\rightsquigarrow\DD^*(G_{A_{0,k}})(A_0)$ is an anti-equivalence of categories.
\end{thm}
%

Let $\{\cQ_n\}$ denote a moduli of connections for $\M_0=\M_0'\in\PMF{R_0}$, and let $\wh\cQ_\infty$ denote the $p$-adic completion of $\cQ_\infty:=\varinjlim\cQ_n$, as usual.
By Lemma~\ref{lem:Artin-Schreier}(\ref{lem:Artin-Schreier:Sm}), Theorem~\ref{thm:dJgen} can be applied to any direct factor $A_0$ of $\wh\cQ_\infty$ which is a finite product of domains.

Let us first record a consequence. 
Let $A_0$ be as in Theorem~\ref{thm:dJgen}, and set $A:=A_0\otimes_{R_0}R \cong A_0[u]/E(u)$, which is a $\varpi$-adic formally \'etale $R$-algebra. Set $\Sig_A:=A_0[[u]]$ and construct $S_A$, etc., accordingly.
\begin{corsub}\label{cor:dJgen} 
In the above setting, the contravariant functor
\begin{equation*}
\M^*:\big\{p\text{-divisible groups over }A \big\} \ra \PMFqw{S_A},
\end{equation*}
given by $G\rightsquigarrow\DD^*(G)(S_A)$ is an exact anti-equivalence of categories.
\end{corsub}
\begin{proof}
The full assertion for $p>2$ and the essential surjectivity claim for $p=2$ follows from Theorems~\ref{thm:BreuilClassif} and \ref{thm:dJgen}. When $p=2$, full faithfulness of $\M^*$ follows from Corollary~\ref{cor:gMFF}, which can be applied thanks to Lemma~\ref{lem:Artin-Schreier}(\ref{lem:Artin-Schreier:Sm}).
\end{proof}

We first begin with the following standard commutative algebra lemma, which will be used in the proof of Theorem~\ref{thm:dJgen}:
\begin{lemsub}\label{lem:BourbakiComplGrade}
Let $B$ be a (not necessarily noetherian) ring, and $I\subset B$ a \emph{finitely generated} ideal. Let  $\wh B:=\varprojlim B/I^n$ denote the $I$-adic completion, equipped with the inverse limit topology. Then the topological closure of $I^n$ in $\wh B$  is $I^n\wh B$, and the natural topology of $\wh B$ coincides with the  $I\wh B$-adic topology. Furthermore, if $B/I$ is noetherian, then so is $\wh B$. 
\end{lemsub}
We will apply this lemma when $I=(p)$.
\begin{proof}
The first assertion (on the completion of $I^n$ and the topology of $\wh B$) is exactly Corollaire~2 in \cite[Ch. III, \S2, no.~12, page~228]{Bourbaki:AlgComm1-4}. Since  $I$ is finitely generated, the associated graded ring $\bigoplus_{n\geqs0} I^n/I^{n+1}$ for $\wh B$ is finitely generated over $B/I$, so it is noetherian if $B/I$ is noetherian. Now the noetherian-ness claim follows from Corollaire~1 in \cite[Ch. III, \S2, no.~9, page~217]{Bourbaki:AlgComm1-4}.
\end{proof}

\begin{proof}[Proof of Theorem~\ref{thm:dJgen}]
Full faithfulness of $\M_0^*$ follows from  \cite[Th\'eor\`eme~4.1.1]{Berthelot-Messing:DieudonneIII} and Proposition~\ref{prop:FiltMod}. We prove essential surjectivity by generalising the proof  of \cite[Theorem~3.4.1]{Vasiu:ys}: the main idea is fpqc  descent from the completion at each point of $\Spec A_{0,k}$. (The strategy of descending from completion using the theory of excellent rings appeared in the proof of \cite[Theorem~103]{Zink:DisplayFormalGpAsterisq278} as well.)

Let us set up the notation. Let $\p\subset A_0$ be a prime ideal containing $p$, and set $\ol\p:=\p/(p)$ and $k(\p):=\Frac(A_0/\p) = \Frac(A_{0,k}/\ol\p)$. Let $B_{0,\p}$ be the $p$-adic completion of  $(A_0)_\p$. Note that $B_{0,\p}/(p) = (A_{0,k})_{\ol\p}$, 
which is noetherian by assumption. Applying Lemma~\ref{lem:BourbakiComplGrade}, if follows that $B_{0,\p}$ is $p$-adic and noetherian. 

Let $\wh{(A_0)}_\p$ be the $\p$-adic completion of $(A_0)_\p$. One can easily check that  $\wh{(A_0)}_\p$ is faithfully flat over $B_{0,\p}$, and $\wh{(A_0)}_\p/(p)$ is naturally isomorphic to the $\ol\p$-adic completion of $(A_{0,k})_{\ol\p}$. Since  $(A_{0,k})_{\ol\p}$ is regular, $\wh{(A_0)}_\p/(p)$ is isomorphic to a formal power series ring over $k(\p)$ with finitely many variables, and $\wh{(A_0)}_\p$ is isomorphic to a formal power series ring over some Cohen ring of $k(\p)$ with finitely many variables.

Note that $(A_{0,k})_{\ol\p}$ has a finite $p$-basis because the relative Frobenius morphism  $(A_{0,k})_{\ol\p}\ra A_{0,k}\otimes_{\vphi,A_{0,k}} (A_{0,k})_{\ol\p}$ is an isomorphism. (Indeed, $(A_{0,k})_{\ol\p}=\varinjlim_{f\notin\ol\p} A_{0,k}[\ivtd f]$, and the relative Frobenius morphism is clearly an isomorphism for the $A_{0,k}$-algebra $A_{0,k}[\ivtd f]$.) Hence, $\wh{(A_0)}_\p/(p)$, being the $\ol\p$-adic completion of $(A_{0,k})_{\ol\p}$, has a finite $p$-basis by \cite[Lemma~1.1.3]{dejong:crysdieubyformalrigid}.

As observed in \S\ref{subsec:BC} (Ex5), $\vphi_{A_0}$ uniquely extends to lifts of Frobenius on $B_{0,\p}$ and $\wh{(A_0)}_\p$. Therefore, for any $\M_0\in\PMF{A_0}$, the scalar extensions $\wh{(A_0)}_\p\otimes_{A_0}\M_0$ and $B_{0,p} \otimes_{A_0}\M_0$ can respectively be viewed as objects in $\PMF{\wh{(A_0)}_\p}$ and $\PMF{B_{0,\p}}$. 

Now, by \cite[Main~Theorem~1]{dejong:crysdieubyformalrigid}, there exists a $p$-divisible group $\wh G_{\ol\p}$ over $\wh{(A_0)}_\p/(p)$ equipped with an isomorphism $\DD^*(\wh G_{\ol\p})(\wh{(A_0)}_\p)\cong\wh{(A_0)}_\p\otimes_{A_0}\M_0$ as objects in $\PMF{\wh{(A_0)}_\p}$. We will proceed by descending  $\wh G_{\ol\p}$ to a $p$-divisible group over $(A_{0,k})_{\ol\p}$, and then to a $p$-divisible group over $A_{0,k}$. To produce descent data for $p$-divisible groups, we will first produce descent data for ``Dieudonn\'e crystals'', and show full faithfulness of  crystalline Dieudonn\'e functors over relevant base schemes.

 Viewing $B_{0,\p}\otimes_{A_0}\M_0$ and $\wh{(A_0)}_\p\otimes_{A_0}\M_0$, respectively,  as vector bundles on $\Spf(B_{0,\p},(p))$ and $\Spf(\wh{(A_0)}_\p,(p))$, the ``vector bundle'' $\wh{(A_0)}_\p\otimes_{A_0}\M_0$ carries a descent datum for the fpqc covering $\Spf(\wh{(A_0)}_\p,(p))\ra\Spf(B_{0,\p},(p))$, as it is a pull back of the ``vector bundle'' $B_{0,\p}\otimes_{A_0}\M_0$. To make it more precise, we obtain a projective system of descent data on $\{\wh{(A_0)}_\p/(p^n)\otimes_{A_0}\M_0\}_n$ for faithfully flat maps $\{B_{0,p}/(p^n) \ra \wh{(A_0)}_\p/(p^n)\}_n$. It is now clear that the standard results on full faithfulness and effectivity of descent can be extended to our setting. In particular, we can consider ``descent data with Frobenius structure and connections'', and the  natural Frobenius structure and connection on $B_{0,\p}\otimes_{A_0}\M_0$ can be uniquely recovered from the descent datum on $\wh{(A_0)}_\p\otimes_{A_0}\M_0$.

For the rest of the proof, we let $\wh\otimes$ denote the $p$-adic completion of the $\otimes$-product. 
We respectively define  $i_1,i_2: \wh{(A_0)}_\p\ra\wh{(A_0)}_\p\wh\otimes_{B_{0,\p}}\wh{(A_0)}_\p$ by $i_1:a\mapsto a\wh\otimes1$ and $i_2:a\mapsto 1\wh\otimes a$ for any $a\in \wh{(A_0)}_\p$. Then the descent datum on $\wh{(A_0)}_\p\otimes_{A_0}\M_0$ can be interpreted as an $\wh{(A_0)}_\p\wh\otimes_{B_{0,\p}}\wh{(A_0)}_\p$-linear isomorphism 
\begin{equation}\label{eqn:descent}
\left(\wh{(A_0)}_\p\wh\otimes_{B_{0,\p}}\wh{(A_0)}_\p\right)\otimes_{i_1,A_0}\M_0 \riso \left(\wh{(A_0)}_\p\wh\otimes_{B_{0,\p}}\wh{(A_0)}_\p\right)\otimes_{i_2,A_0}\M_0 
\end{equation}
which satisfies the ``standard'' cocycle condition that involves scalar extensions by different choices of maps  $\wh{(A_0)}_\p\wh\otimes_{B_{0,\p}}\wh{(A_0)}_\p\ra \wh{(A_0)}_\p\wh\otimes_{B_{0,\p}}\wh{(A_0)}_\p\wh\otimes_{B_{0,\p}}\wh{(A_0)}_\p$; actually,  the cocycle condition 
coincides with the usual cocycle condition for the descent datum for modules, except that $\otimes$ is replaced by $\wh\otimes$. 
Since the Frobenius structure and connection on  $\wh{(A_0)}_\p\otimes_{A_0}\M_0$ are induced from $B_{0,p}\otimes_{A_0}\M_0$, the isomorphism (\ref{eqn:descent}) is an isomorphism in $\PMF{\wh{(A_0)}_\p\wh\otimes_{B_{0,\p}}\wh{(A_0)}_\p}$. Note also that the descent datum (\ref{eqn:descent}) can be interpreted as an isomorphism
\[\DD^*(i_1^*\wh G_{\ol\p}) \riso \DD^*(i_2^*\wh G_{\ol\p}) \]
of Dieudonn\'e crystals over $\Spec(\wh{(A_{0,k})}_{\ol\p}\otimes_{(A_{0,k})_{\ol\p}}\wh{(A_{0,k})}_{\ol\p})$, which satisfies some natural cocycle condition (\emph{cf.,} Proposition~\ref{prop:FiltMod}).

In order to show that this descent datum induces a descent datum on $\wh G_{\ol\p}$ for the fpqc covering $\Spec \wh{(A_{0,k})}_{\ol\p} \ra \Spec(A_{0,k})_{\ol\p}$, we need the following claim:
 \begin{claimsub}\label{clm:ExcNormal}
The rings  $\wh{(A_{0,k})}_{\ol\p}\otimes_{(A_{0,k})_{\ol\p}}\wh{(A_{0,k})}_{\ol\p}$ and $\wh{(A_{0,k})}_{\ol\p}\otimes_{(A_{0,k})_{\ol\p}}\wh{(A_{0,k})}_{\ol\p}\otimes_{(A_{0,k})_{\ol\p}}\wh{(A_{0,k})}_{\ol\p}$ are normal and admit a finite $p$-basis.
\end{claimsub}
The $p$-basis assertion is clear.
To show normality, let us write $\wh{(A_{0,k})}_{\ol\p}=\varinjlim_i B\com i$ where $B\com i$'s are finitely generated \emph{normal} $(A_{0,k})_{\ol\p}$-subalgebras in $\wh{(A_{0,k})}_{\ol\p}$; indeed, this is possible as the normalisation of any finitely generated $(A_{0,k})_{\ol\p}$-algebra is finite by excellence of $(A_{0,k})_{\ol\p}$. 

Recall that the natural map $(A_{0,k})_{\ol\p}\ra\wh{(A_{0,k})}_{\ol\p}$ is  flat with geometrically regular fibres by  excellence of $(A_{0,k})_{\ol\p}$. In particular, $B\com i \ra \wh{(A_{0,k})}_{\ol\p}\otimes_{(A_{0,k})_{\ol\p}}B\com i$ is flat with geometrically regular fibres. Since the source of the map is normal, it follows from (Corollary of) \cite[Theorem~23.9]{matsumura:crt} that $ \wh{(A_{0,k})}_{\ol\p}\otimes_{(A_{0,k})_{\ol\p}}B\com i$ is normal. Normality of $\wh{(A_{0,k})}_{\ol\p}\otimes_{(A_{0,k})_{\ol\p}}\wh{(A_{0,k})}_{\ol\p} = \varinjlim_i\left( \wh{(A_{0,k})}_{\ol\p}\otimes_{(A_{0,k})_{\ol\p}}B\com i\right) $ is now clear. We similarly show that $\wh{(A_{0,k})}_{\ol\p}\otimes_{(A_{0,k})_{\ol\p}}\wh{(A_{0,k})}_{\ol\p}\otimes_{(A_{0,k})_{\ol\p}}\wh{(A_{0,k})}_{\ol\p}$ is normal. This proves Claim~\ref{clm:ExcNormal}.

It now follows from Claim~\ref{clm:ExcNormal} that the functors $\DD^*(-)(\wh{(A_0)}_\p\wh\otimes_{B_{0,\p}}\wh{(A_0)}_\p)$ and $\DD^*(-)(\wh{(A_0)}_\p\wh\otimes_{B_{0,\p}}\wh{(A_0)}_\p\wh\otimes_{B_{0,\p}}\wh{(A_0)}_\p)$ on $p$-divisible groups (over suitable bases) are fully faithful. (\emph{Cf.,} \cite[Th\'eor\`eme~4.1.1]{Berthelot-Messing:DieudonneIII}, Proposition~\ref{prop:FiltMod}. Recall that   $(\wh{(A_0)}_\p\wh\otimes_{B_{0,\p}}\wh{(A_0)}_\p)/(p)=\wh{(A_{0,k})}_{\ol\p}\otimes_{(A_{0,k})_{\ol\p}}\wh{(A_{0,k})}_{\ol\p}$, and similarly for the triple tensor product.) Therefore, the descent datum on  $\wh{(A_0)}_\p\otimes_{A_0}\M_0$ induces a descent datum on $\wh G_{\ol\p}$. By descending the finite flat group scheme $\wh G_{\ol\p}[p^n]$ for each $n$, we obtain a $p$-divisible group $G_{\ol\p}$ over $(A_{0,k})_{\ol\p}$ such that $\DD^*(G_{\ol\p})(B_{0,\p})\cong B_{0,\p}\otimes_{A_0}\M_0$. 

Now, a similar (but much easier) consideration produces a descent datum on $\{G_{\ol p}\}_{\ol\p\in\Spec A_{0,k}}$ with respect to the fpqc covering $\{\Spec ( A_{0,k})_{\ol\p}\}$ of $\Spec A_{0,k}$. (Note that Claim~\ref{clm:ExcNormal} becomes rather straightforward  for $(A_{0,k})_{\ol\p}\otimes_{A_{0,k}}(A_{0,k})_{\ol\p'}$ and $(A_{0,k})_{\ol\p}\otimes_{A_{0,k}}(A_{0,k})_{\ol\p'}\otimes_{A_{0,k}}(A_{0,k})_{\ol\p''}$, where $\ol\p,\ol\p',\ol\p''\in\Spec A_{0,k}$.) The resulting $p$-divisible group $G_{A_{0,k}}$ comes equipped with a natural isomorphism $\DD^*(G_{A_{0,k}})(A_0)\cong \M_0$, as desired. This concludes the proof of Theorem~\ref{thm:dJgen}.
\end{proof}
\subsection{Proof of Theorem~\ref{thm:RelKisinFF}}\label{subsec:RelKisinFF}
Assume that $R$ satisfies the formally finite-type assumption~(\S\ref{cond:dJ}), and let $I\subset R$ be an ideal with $(\varpi)\subset I \subset J_R$. We set $I_0:=\ker (R_0\thra R/I)$.
Recall from \S\ref{subsec:IndEtBC} that for any $I$-adic formally \'etale $R$-algebra $A$ such that $A/IA$ is an inductive limit of \'etale $R/I$-algebras, we have a canonical choice of $R_0$-subalgebra $A_0\subset A$ and a lift of Frobenius $\vphi_{A_0}$ over $\vphi_{R_0}$. We set $\Sig_A:=A_0[[u]]$ (with $\vphi_{\Sig_A}$ extended by $u\mapsto u^p$), and  $I_{\Sig}:=\ker(\Sig\thra R/I)$. Clearly $\Sig$ is $I_{\Sig}$-adically separated and complete. If $I=J_R$, then we write $J_\Sig:=I_\Sig$. (If $I=(\varpi)$ for example,  then $I_{\Sig} = (p,\PP(u))$.)

If  $\{\Spf(A_\alpha,I A_\alpha)\}$ is an ind-\'etale fpqc covering of $\Spf (R, I)$ in the formal scheme sense (i.e., for each $n$, $\Spec A_\alpha/I^nA_\alpha$ are ind-\'etale over $\Spec R/I^n$ and form a fpqc covering), then the corresponding $\{\Spf(\Sig_{A_\alpha}, I_\Sig\Sig_{A_\alpha})\}$ is an ind-\'etale fpqc  covering of $\Spf(\Sig, I_\Sig)$. (Here, ind-\'etale fpqc coverings are defined in the formal scheme sense.) For any $\gM$ either in $\SMqwtor$ or $\SMqtor$, we let $\gM_{A_\alpha}:=\Sig_{A_\alpha}\otimes_\Sig\gM$ denote the scalar extension. 
Note that we will actually work with ind-\'etale fpqc coverings that may not be \'etale coverings; \emph{cf.,}  Proposition~\ref{prop:ConnRaynaud}.

\begin{lemsub}\label{lem:FrobRaynaud}
For any $\gM\in\SMqtor$, there exists a Zariski open covering  $\{\Spf(R_\alpha,J_RR_\alpha)\}$ of  $\Spf (R,J_R)$ and $\gN_\alpha, \gN_\alpha'\in\PMq{\Sig_{R_\alpha}}$ such that for each $\alpha$ there is a $\vphi$-equivariant exact sequence \[0\ra\gN_\alpha' \ra \gN_\alpha\ra \gM_\alpha:=\Sig_{R_\alpha}\otimes_\Sig\gM \ra 0.\] 
%
\end{lemsub}
\begin{proof}
We prove this lemma by adapting the proof of  \cite[Lemma~2.3.4]{kisin:fcrys} as follows. For $\gM\in\SMqtor$,  set $L:=\coker(1\otimes\vphi_\gM)$ and $\gF:=\im(1\otimes\vphi_\gM)$. Choose a finite free $R$-module $\wt L$ which surjects onto $L$ and a finite-rank free $\Sig$-module $\gN$ which surjects onto both $\gM$ and $\wt L$. Set $\wt\gF:=\ker(\gN\thra\wt L)$. We may assume that $\wt \gF$ surjects onto $\gF$ under the natural projection $\gN\thra \gM$ by replacing, if necessary, $\gN$ by $\gN\oplus\Sig^r$ where $\Sig^r$ surjects onto $\gF$ and maps to zero in $\wt L$. To summarise, we have the following commutative diagram with exact rows:
\[\xymatrix{
0 \ar[r] & \wt\gF \ar[r] \ar@{->>}[d] & \gN \ar[r] \ar@{->>}[d]& \wt L \ar[r] \ar@{->>}[d]& 0\\
0 \ar[r] & \gF \ar[r] & \gM \ar[r] &  L \ar[r] & 0}.\]
Since $\wt L$ is free over $R$ and $R$ is of $\Sig$-projective dimension $1$, it follows that $\wt\gF$ is projective over $\Sig$. We choose $\set{R_\alpha}$ so that $\wt\gF_\alpha:=\Sig_{R_\alpha}\otimes_\Sig\wt\gF$ is free over $\Sig_{R_\alpha}$ for each $\alpha$. (Note that since $\Sig_{R_\alpha}$ is $J_\Sig$-adic, a finitely generated projective $\Sig_{R_\alpha}$-module $\wt\gF_\alpha$ is free if and only if $\wt\gF_\alpha/J_\Sig\wt\gF_\alpha$ is free over $\Sig_{R_\alpha}/J_\Sig\Sig_{R_\alpha} = R_\alpha/J_RR_\alpha$. So we take $R_\alpha$ so that $\wt\gF_\alpha/J_\Sig\wt\gF_\alpha$ is free over $R_\alpha/J_RR_\alpha$.) Since both $\wt\gF_\alpha$ and $\gN_\alpha$ are free of the same rank, we may choose an isomorphism $\vphi^*\gN_\alpha\cong\wt\gF_\alpha$, and define $\vphi_{\gN_\alpha}$  so that   we have:
 \[1\otimes\vphi_{\gN_\alpha}:\vphi^*\gN_\alpha\cong\wt\gF_\alpha\hra\gN_\alpha.\]
It is clear that $(\wt\gM_\alpha,\vphi_{\wt\gM_\alpha})\in\PMqtor{\Sig_{R_\alpha}}$ and the natural projection $\wt\gM_\alpha\thra\gM_\alpha$ is $\vphi$-equivariant.

Set $\wt\gM_\alpha':=\ker(\wt\gM_\alpha\thra\gM_\alpha)$, which is a $\vphi$-stable submodule of $\wt\gM_\alpha$. We claim that $\wt\gM_\alpha'\in\PMqtor{\Sig_{R_\alpha}}$. 
Set $\wt L_\alpha':=\coker(1\otimes\vphi_{\wt\gM_\alpha'})$. Then  $\wt L_\alpha'$ naturally embeds into $\wt L_\alpha:=R_\alpha\otimes_R \wt L$ since $\wt L_\alpha'[\ivtd p] = \wt L_\alpha[\ivtd p]$ and $R$ is $\Zp$-flat. It follows that $\wt L_\alpha'$ is annihilated by $\PP$.
\end{proof}
%
\begin{propsub}\label{prop:ConnRaynaud}
For any $\gM\in\SMFIqw$, there exists a Zariski open covering $\{\Spf(R_\alpha,J_RR_\alpha)\}$  of $\Spf (R,J_R)$, such that  for each $\alpha $ there exist  a faithfully flat ind-\'etale map $\Spf(A_{\alpha},(\varpi))\ra\Spf(R_\alpha,(\varpi))$ and $\gN_{A_\alpha}\in\PMqw{\Sig_{A_\alpha}}$ with the following properties:
\begin{enumerate}
\item\label{prop:ConnRaynaud:base} for each  $\alpha$,  $A_{\alpha,0}$ satisfies the assumption for $A_0$ over $R_{\alpha,0}$ in Theorem~\ref{thm:dJgen};
\item\label{prop:ConnRaynaud:connection} for each $\alpha$, there is a $\vphi$-equivariant surjective map $\gN_{A_\alpha}  \thra \gM_{A_\alpha}:=\Sig_{A_\alpha}\otimes_\Sig\gM$ such that 
the induced map $\N_{A_\alpha,0}:=A_{\alpha,0}\otimes_{\vphi,\Sig_{A_0}}\gN_{A_0}\thra \M_{A_\alpha,0}:=A_{\alpha,0}\otimes_{\vphi,\Sig}\gM$ respects connections.
\end{enumerate} 
\end{propsub}
\begin{proof}
In order to prove the proposition, we may replace $R$ by one of $R_\alpha$'s as in Lemma~\ref{lem:FrobRaynaud}, so that we have $\gN\in\SMq$ with $\vphi$-equivariant surjective map $\gN\thra\gM$. (\emph{Cf.,} Lemma~\ref{lem:FrobRaynaud}.) The main idea is to find an ind-\'etale fpqc covering over which $\N_0:=R_0\otimes_{\vphi,\Sig}\gN$ admits a connection as in the statement. 
We use some slight variant of Vasiu's construction of ``moduli of connections'' (\emph{cf.,} \cite[Theorem~3.2]{Vasiu:ys}, and also \S\ref{subsec:Vasiu} of this paper).

Let us set up the notation. Set $\M_0:=R_0\otimes_{\vphi,\Sig}\gM\cong R_0\otimes_S(S\otimes_{\vphi,\Sig}\gM)$. Since $S\otimes_{\vphi,\Sig}\gM \in\SMFtor$, one can view $\M_0\in\PMFtor{R_0}$ where all the extra structure is induced from $S\otimes_{\vphi,\Sig}\gM$. Similarly, $\N_0:=R_0\otimes_{\vphi,\Sig}\gN$ can naturally be viewed as an object in $\PMFq{R_0}$.  We also fix a $R_0$-direct factor  $(\N_0)^1\subseteq \N_0$ which lifts $(\Fil^1\N_0)/p\N_0\subseteq \N_0/p\N_0$.

Assume that $p^n\gM=0$. 
Let $\{\Spf(R_\alpha, J_RR_\alpha)\}$ be a Zariski covering of $\Spf (R,J_R)$ with the following properties: 
\begin{itemize}
\item $(\N_0)^1$, $\N_0/ (\N_0)^1$, and $\wh\Omega_{R_0}$ become free over $R_{\alpha,0}$; 
\item the underlying $R_0$-modules of both $\M_0$ and the kernel of $\N_0/p^n\N_0\thra \M_0$ become isomorphic to $\bigoplus (R_{\alpha,0}/p^i)^{d_i}$ after the base change. 
\end{itemize}
We may (and do) replace $R$ by one of $R_\alpha$, and drop the subscript $_\alpha$ from now on. 

We now construct  an ind-\'etale fpqc covering $\Spf(A_0,(p))\ra \Spf (R_0,(p))$ as follows.
Recall that for each $n\geqs1$ we have a faithfully flat $p$-adic formally \'etale $R$-algebra $\cQ_n$ such that for any $p$-adic formally \'etale $R_0$-algebra $B_0$, $\Hom_{R_0}(\cQ_n,B_0)$ is naturally in bijection with the set of connections on $B_0/(p^n)\otimes_{R_0}\N_0$ which satisfy the commutative diagram (\ref{eqn:ASsys1}). One can see without difficulty that there is a universal quotient $ \cQ_{1,k}'$ of $\cQ_1/(p)$ such that the natural projection $\cQ_{1,k}'\otimes_{R_0}\N_0\thra\cQ_{1,k}'\otimes_{R_0}\M_0$ is horizontal when we give a universal connection on the source. Furthermore, by writing down this condition in terms of the ``matrix entries'' of a universal connection, one sees that $\cQ_{1,k}'$ is defined by equations of the form $x' = B'\vphi(x')+C'$ where $B'$ and $C'$ are matrices over $R_0/(p)$ and $x'$ is a column vector of variables. One can check (\emph{cf.,} \cite[Theorem~2.4.1]{Vasiu:ys}) that $\cQ_{1,k}'$  is faithfully flat \'etale over $R_0/(p)$, hence $\Spec \cQ_{1,k}'$ is a union of connected components in $\Spec \cQ_1/(p)$. Let $\cQ_1'$ denote  the unique $p$-adic formally \'etale lift of $\cQ_{1,k}'$, which can naturally be viewed as a quotient of $\cQ_1$. 

We can repeat this construction to obtain a union of connected components $\Spf (\cQ_{n}',(p)) \subset \Spf(\cQ_n,(p))$ where $p^n\gM=0$, with the property that $\cQ_n'/(p^n)\otimes_{R_0}\N_0 \thra \cQ_n'\otimes_{R_0}\M_0 $ is horizontal. (Note that both $p^i\N_0/p^{i+1}\N_0$ and $p^i\M_0/p^{i+1}\M_0$ are free over $R_0/(p)$ by assumption.) Pick a formally \'etale lift $\cQ_n'$ of $\cQ_{n,k}'$ and view it as a quotient of $\cQ_n$.  For any $i\geqs n$, we set $\cQ_i':=\cQ_i\otimes_{\cQ_n}\cQ_n'$, and let $\wh \cQ'_\infty$ denote the $p$-adic completion of $\cQ_\infty':=\varinjlim_i\cQ_i'$. 

Note that each $\Spf(\cQ_n',(p))$ is faithfully flat and \'etale over $\Spf(R_0,(p))$, so we can choose a direct summand $A_0$ of $ \wh\cQ_\infty'$, such that $A_0$ is a finite product of domains and $\Spec A_0/(p)$ surjects onto $\Spec R_0/(p)$ (which is possible by Lemma~\ref{lem:Artin-Schreier}(\ref{lem:Artin-Schreier:Stratif})). As $\wh\cQ'_\infty$ is a direct factor of $\wh\cQ_\infty$, $A_0$ satisfies (\ref{prop:ConnRaynaud:base}) in the statement by Lemma~\ref{lem:Artin-Schreier}(\ref{lem:Artin-Schreier:Sm}). The universal connection on $\wh\cQ'_\infty$ induces a  connection on $A_0\otimes_{R_0}\N_0$, which makes it an object in $\PMF{A_0}$ by  Proposition~\ref{prop:Vasiu}, and satisfies (\ref{prop:ConnRaynaud:connection}) in the statement by construction. Now  we set $A:=A_0\otimes_W\fo_K$ and $\gN_{A}:= \Sig_{A}\otimes_{\Sig}\gN$. 
 %
%
%
%
\end{proof}

Now let us prove Theorem~\ref{thm:RelKisinFF}. By Proposition~\ref{prop:RelKisinFF}, we only need to show essential surjectivity.
From Lemma~\ref{lem:FrobRaynaud} and Corollary~\ref{cor:BreuilClassif}, one easily obtain Theorem~\ref{thm:RelKisinFF} except the assertions involving $\SMFIqw$. We now show the remaining part of Theorem~\ref{thm:RelKisinFF}.

By local flatness criterion (for modules over $\Sig/(p^i)$), for a $p$-power order finite locally free group scheme $H$ over $R$, $H[p^i]$ is locally free over $R$ for each $i$ if and only if $\gM^*(H)\in\SMFIqw$. So by Proposition~\ref{prop:RelKisinFF} and Remark~\ref{rmk:torKisModSm} it remains to show that for any $\gM\in\SMFIqw$ there exists a finite locally free group scheme $H$ over $R$ such that $\gM^*(H)\cong \gM$. 

For a given $\gM\in\SMFIqw$ we choose $\{R_\alpha\}$, $\{A_\alpha\}$ and  $\gN_{A_\alpha} \in\PMqw{\Sig_\beta}$ as in Proposition~\ref{prop:ConnRaynaud}, and set $\gN'_{A_\alpha}:=\ker(\gN_{A_\alpha}\thra \gM_{A_\alpha})$, which is an object in $\PMqw{\Sig_{A_\alpha}}$. By Corollary~\ref{cor:dJgen}, we obtain an isogeny of $p$-divisible groups $G_{A_\alpha} \ra G_{A_\alpha}'$ over $\Spf(A_\alpha,(\varpi))$ corresponding to the natural inclusion $\gN_{A_\alpha}'\hra\gN_{A_\alpha}$, and set $H_{A_\alpha}:=\ker(G_{A_\alpha}\ra G_{A_\alpha}')$. Clearly, we have $\gM^*(H_{A_\alpha}) \cong \Sig_{A_\alpha}\otimes_\Sig\gM$. We will produce $H$ over $R$ by first descending $H_{A_\alpha}$ over $R_\alpha$, and glue them together.

We write $\wh\otimes_{R_\alpha}$ for the $\varpi$-adically completed tensor product, and $\wh\otimes_{\Sig_\alpha}$ for the $(p,\PP(u))$-adically completed tensor product (or equivalently, $(p,u)$-adically completed tensor product). Consider maps  $i_1,i_2:A_\alpha \ra A_\alpha\wh\otimes_{R_\alpha}A_\alpha$ defined by $i_1:a\mapsto a\wh\otimes1$ and $i_2:a\mapsto 1\wh\otimes a$ for $a\in A_\alpha$, and use the same notation to denote the maps $i_1,i_2:\Sig_{A_\alpha}\ra\Sig_{A_\alpha}\wh\otimes_{\Sig_{R_\alpha}}\Sig_{A_\alpha}$ defined similarly. 

Since $\Spf(A_\alpha,(\varpi))\ra \Spf(R_\alpha,(\varpi))$ is faithfully flat map of formal schemes (i.e., $\Spec A_\alpha/(\varpi^n)\ra \Spec R_\alpha/(\varpi^n)$ is faithfully flat for each $n$), we can apply fpqc descent theory. 
A descent datum on $H_{A_\alpha}$ is an isomorphism
\begin{equation}\label{eqn:DescentFinFl}
i_1^*H_{A_\alpha} \liso i_2^*H_{A_\alpha}
\end{equation}
of finite locally free group schemes over $A_\alpha\wh\otimes_{R_\alpha}A_\alpha$ which satisfies the natural cocycle conditions. By setting $A_{\alpha,k}:=A_\alpha/(\varpi)$ and $R_{\alpha,k}:=R_\alpha/(\varpi)$, the tensor products $A_{\alpha,k}\otimes_{R_{\alpha,k}}A_{\alpha,k}$ and $A_{\alpha,k}\otimes_{R_{\alpha,k}}A_{\alpha,k}\otimes_{R_{\alpha,k}}A_{\alpha,k}$ are normal and locally admit finite $p$-basis. (For normality, note that these rings are filtered direct limits of \'etale $R_{\alpha,k}$-algebras, which are normal.) Therefore, by the base change property and full faithfulness of $\gM^*$ as stated in Proposition~\ref{prop:RelKisinFF}
, we obtain a descent datum on $H_{A_\alpha}$ as in  (\ref{eqn:DescentFinFl}) from the ``$(p,u)$-adically continuous descent datum''
\begin{equation}\label{eqn:DescentTorKis}
(\Sig_{A_\alpha}\wh\otimes_{\Sig_{R_\alpha}}\Sig_{A_\alpha})\otimes_{i_1,\Sig}\gM \riso(\Sig_{A_\alpha}\wh\otimes_{\Sig_{R_\alpha}}\Sig_{A_\alpha})\otimes_{i_2,\Sig}\gM
\end{equation}
obtained from the fact that $\Sig_{A_\alpha}\otimes_\Sig\gM$ is the scalar extension of $\Sig_{R_\alpha}\otimes_\Sig\gM$. Therefore, we obtain $H_{R_\alpha}$ over $\Spf(R_\alpha,(\varpi))$ such that $\gM^*(H_{R_\alpha})\cong \Sig_{R_\alpha}\otimes_\Sig\gM$. 

A similar consideration produces a glueing datum on $\{H_{R_\alpha}\}_\alpha$, so we obtain a finite locally group scheme $H$ over $R$ with $\gM^*(H)\cong \gM$.
This concludes the proof of Theorem~\ref{thm:RelKisinFF}.
\bibliography{bib}

\def\cprime{$'$}
\providecommand{\bysame}{\leavevmode\hbox to3em{\hrulefill}\thinspace}
\providecommand{\MR}{\relax\ifhmode\unskip\space\fi MR }
\providecommand{\MRhref}[2]{%
  \href{http://www.ams.org/mathscinet-getitem?mr=#1}{#2}
}
\providecommand{\href}[2]{#2}
\begin{thebibliography}{BBM82}

\bibitem[And06]{Andreatta:GenNormRings}
Fabrizio Andreatta, \emph{Generalized ring of norms and generalized
  {$(\phi,\Gamma)$}-modules}, Ann. Sci. {\'E}cole Norm. Sup. (4) \textbf{39}
  (2006), no.~4, 599--647. \MR{2290139 (2007k:12006)}

\bibitem[BBM82]{Berthelot-Breen-Messing:DieudonneII}
Pierre Berthelot, Lawrence Breen, and William Messing, \emph{Th{\'e}orie de
  {D}ieudonn{\'e} cristalline. {II}}, Lecture Notes in Mathematics, vol. 930,
  Springer-Verlag, Berlin, 1982. \MR{667344 (85k:14023)}

\bibitem[BK]{BlockKato:Dieudonne}
Spencer Bloch and Kazuya Kato, \emph{{$p$}-divisible groups and {Dieudonn\'e}
  crystals}, unpublished. {\tt
  http://staff.science.uva.nl/\~{}bmoonen/pDivDieudCryst.pdf}, 46 pages.

\bibitem[BM90]{Berthelot-Messing:DieudonneIII}
Pierre Berthelot and William Messing, \emph{Th{\'e}orie de {D}ieudonn{\'e}
  cristalline. {III}. {T}h{\'e}or{\`e}mes d'{\'e}quivalence et de pleine
  fid{\'e}lit{\'e}}, The {G}rothendieck {F}estschrift, {V}ol.\ {I}, Progr.
  Math., vol.~86, Birkh{\"a}user Boston, Boston, MA, 1990, pp.~173--247.
  \MR{1086886 (92h:14012)}

\bibitem[BO78]{Berthelot-Ogus}
Pierre Berthelot and Arthur Ogus, \emph{Notes on crystalline cohomology},
  Princeton University Press, Princeton, N.J., 1978. \MR{MR0491705 (58
  \#10908)}

\bibitem[BO83]{MR700767}
P.~Berthelot and A.~Ogus, \emph{{$F$}-isocrystals and de {R}ham cohomology.
  {I}}, Invent. Math. \textbf{72} (1983), no.~2, 159--199. \MR{700767
  (85e:14025)}

\bibitem[Bou06]{Bourbaki:AlgComm1-4}
N.~Bourbaki, \emph{\'{E}l\'ements de math\'ematique. {A}lg\`ebre commutative.
  {C}hapitres 1 \`a 4}, Springer, Berlin, 2006, Reprint of the 1985 original.

\bibitem[Bre98]{breuil:GpSchNormField}
Christophe Breuil, \emph{Sch{\'e}mas en groupes et corps des normes},
  Preprint,\hfill\\ {\tt www.ihes.fr/\~{}breuil/PUBLICATIONS/groupesnormes.pdf}
  (1998).

\bibitem[Bre00]{Breuil:GrPDivGrFiniModFil}
\bysame, \emph{Groupes {$p$}-divisibles, groupes finis et modules filtr\'es},
  Ann. of Math. (2) \textbf{152} (2000), no.~2, 489--549. \MR{MR1804530
  (2001k:14087)}

\bibitem[Bre02]{Breuil:IntegralPAdicHodgeThy}
\bysame, \emph{Integral {$p$}-adic {H}odge theory}, Algebraic geometry 2000,
  {A}zumino ({H}otaka), Adv. Stud. Pure Math., vol.~36, Math. Soc. Japan,
  Tokyo, 2002, pp.~51--80. \MR{MR1971512 (2004e:11135)}

\bibitem[Bri06]{Brinon:imperfect}
Olivier Brinon, \emph{Repr{\'e}sentations cristallines dans le cas d'un corps
  r{\'e}siduel imparfait}, Ann. Inst. Fourier (Grenoble) \textbf{56} (2006),
  no.~4, 919--999. \MR{2266883 (2007h:11131)}

\bibitem[Bri08]{Brinon:CrisDR}
\bysame, \emph{Repr{\'e}sentations {$p$}-adiques cristallines et de de {R}ham
  dans le cas relatif}, M{\'e}m. Soc. Math. Fr. (N.S.) (2008), no.~112, vi+159.
  \MR{2484979 (2010a:14034)}

\bibitem[Bri10]{Brinon:Habilitation}
\bysame, \emph{Th\'eorie de {H}odge {$p$}-adique et
  {$(\varphi,\Gamma)$}-modules r\'elatifs}, preprint (2010).

\bibitem[BT08]{BrinonTrihan:CrisImperf}
Olivier Brinon and Fabien Trihan, \emph{Repr{\'e}sentations cristallines et
  {$F$}-cristaux: le cas d'un corps r{\'e}siduel imparfait}, Rend. Semin. Mat.
  Univ. Padova \textbf{119} (2008), 141--171. \MR{2431507 (2009f:14034)}

\bibitem[CL09]{Caruso-Liu:qst}
Xavier Caruso and Tong Liu, \emph{Quasi-semi-stable representations}, Bull.
  Soc. Math. France \textbf{137} (2009), no.~2, 185--223. \MR{MR2543474}

\bibitem[dJ95]{dejong:crysdieubyformalrigid}
A.~J. de~Jong, \emph{Crystalline {D}ieudonn{\'e} module theory via formal and
  rigid geometry}, Inst. Hautes {\'E}tudes Sci. Publ. Math. \textbf{82} (1995),
  5--96. \MR{1383213 (97f:14047)}

\bibitem[dJM99]{deJong-Messing}
A.~J. de~Jong and W.~Messing, \emph{Crystalline {D}ieudonn{\'e} theory over
  excellent schemes}, Bull. Soc. Math. France \textbf{127} (1999), no.~2,
  333--348. \MR{1708635 (2001b:14075)}

\bibitem[EGA]{EGA}
\emph{\'{E}l\'ements de g\'eom\'etrie alg\'ebrique, {I}--{IV}}, Inst. Hautes
  \'Etudes Sci. Publ. Math. (1960-1967), no.~4, 8, 11, 17, 20, 24, 28, and 32.

\bibitem[Fal89]{Faltings:Ccris}
Gerd Faltings, \emph{Crystalline cohomology and {$p$}-adic
  {G}alois-representations}, Algebraic analysis, geometry, and number theory
  ({B}altimore, {MD}, 1988), Johns Hopkins Univ. Press, Baltimore, MD, 1989,
  pp.~25--80. \MR{1463696 (98k:14025)}

\bibitem[Fal99]{Faltings:IntegralCrysCohoVeryRamBase}
\bysame, \emph{Integral crystalline cohomology over very ramified valuation
  rings}, J. Amer. Math. Soc. \textbf{12} (1999), no.~1, 117--144.
  \MR{MR1618483 (99e:14022)}

\bibitem[Fal02]{Faltings:AlmostEtaleExt}
\bysame, \emph{Almost {\'e}tale extensions}, Ast{\'e}risque (2002), no.~279,
  185--270, Cohomologies $p$-adiques et applications arithm{{\'e}}tiques, II.
  \MR{1922831 (2003m:14031)}

\bibitem[Fon82]{Fontaine:BT}
Jean-Marc Fontaine, \emph{Sur certains types de repr{\'e}sentations
  {$p$}-adiques du groupe de {G}alois d'un corps local;\ construction d'un
  anneau de {B}arsotti-{T}ate}, Ann. of Math. (2) \textbf{115} (1982), no.~3,
  529--577. \MR{MR657238 (84d:14010)}

\bibitem[Fon90]{fontaine:grothfest}
\bysame, \emph{Repr\'esentations {$p$}-adiques des corps locaux. {I}}, The
  Grothendieck Festschrift, Vol.\ II, Progr. Math., vol.~87, Birkh{\"a}user
  Boston, Boston, MA, 1990, pp.~249--309. \MR{MR1106901 (92i:11125)}

\bibitem[Fon94]{fontaine:Asterisque223ExpII}
\bysame, \emph{Le corps des p\'eriodes {$p$}-adiques}, Ast\'erisque (1994),
  no.~223, 59--111, With an appendix by Pierre Colmez, P\'eriodes $p$-adiques
  (Bures-sur-Yvette, 1988). \MR{MR1293971 (95k:11086)}

\bibitem[GR03]{GabberRamero}
Ofer Gabber and Lorenzo Ramero, \emph{Almost ring theory}, Lecture Notes in
  Mathematics, vol. 1800, Springer-Verlag, Berlin, 2003. \MR{2004652
  (2004k:13027)}

\bibitem[Gre76]{Greco:ThmExcellence}
Silvio Greco, \emph{Two theorems on excellent rings}, Nagoya Math. J.
  \textbf{60} (1976), 139--149. \MR{0409452 (53 \#13207)}

\bibitem[Ill71]{Illusie:CplxCotI}
Luc Illusie, \emph{Complexe cotangent et d{\'e}formations. {I}}, Lecture Notes
  in Mathematics, Vol. 239, Springer-Verlag, Berlin, 1971. \MR{0491680 (58
  \#10886a)}

\bibitem[Kat73]{katz:antwerp350}
N.\thinspace{}M. Katz, \emph{$p$-adic properties of modular schemes and modular
  forms}, Modular functions of one variable, III (Proc. Internat. Summer
  School, Univ. Antwerp, Antwerp, 1972) (Berlin), Springer, 1973, pp.~69--190.
  Lecture Notes in Mathematics, Vol. 350.

\bibitem[Kat81]{Katz:SerreTate}
N.~Katz, \emph{Serre-{T}ate local moduli}, Algebraic surfaces ({O}rsay,
  1976--78), Lecture Notes in Math., vol. 868, Springer, Berlin, 1981,
  pp.~138--202. \MR{638600 (83k:14039b)}

\bibitem[Kim12]{Kim:ClassifFFGpSchOver2AdicDVR}
Wausu Kim, \emph{The classification of $p$-divisible groups over $2$-adic
  discrete valuation rings}, Math. Res. Lett. \textbf{19} (2012), no.~01,
  121--141.

\bibitem[Kis06]{kisin:fcrys}
Mark Kisin, \emph{Crystalline representations and {$F$}-crystals}, Algebraic
  geometry and number theory, Progr. Math., vol. 253, Birkh{\"a}user Boston,
  Boston, MA, 2006, pp.~459--496. \MR{MR2263197 (2007j:11163)}

\bibitem[Kis09]{Kisin:2adicBT}
\bysame, \emph{Modularity of 2-adic {B}arsotti-{T}ate representations}, Invent.
  Math. \textbf{178} (2009), no.~3, 587--634. \MR{MR2551765}

\bibitem[Lau08]{Lau:DisplayFormalBT}
Eike Lau, \emph{Displays and formal {$p$}-divisible groups}, Invent. Math.
  \textbf{171} (2008), no.~3, 617--628. \MR{2372808 (2009j:14058)}

\bibitem[Lau10a]{Lau:GalRep}
Eike Lau, \emph{Displayed equations for {Galois} representations}, Preprint,
  {arXiv:1012.4436} (2010).

\bibitem[Lau10b]{Lau:Frames}
Eike Lau, \emph{Frames and finite group schemes over complete regular local
  rings}, Doc. Math. \textbf{15} (2010), 545--569. \MR{2679066 (2011g:14107)}

\bibitem[Lau10c]{Lau:2010fk}
Eike Lau, \emph{A relation between {D}ieudonn\'e displays and crystalline
  {D}ieudonn\'e theory}, Preprint, {arXiv:1006.2720} (2010).

\bibitem[Liu08]{Liu:StronglyDivLattice}
Tong Liu, \emph{On lattices in semi-stable representations: a proof of a
  conjecture of {B}reuil}, Compos. Math. \textbf{144} (2008), no.~1, 61--88.
  \MR{MR2388556}

\bibitem[Mat80]{matsumura:CommAlg}
Hideyuki Matsumura, \emph{Commutative algebra}, second ed., Mathematics Lecture
  Note Series, vol.~56, Benjamin/Cummings Publishing Co., Inc., Reading, Mass.,
  1980. \MR{575344 (82i:13003)}

\bibitem[Mat86]{matsumura:crt}
\bysame, \emph{Commutative ring theory}, Cambridge Studies in Advanced
  Mathematics, vol.~8, Cambridge University Press, Cambridge, 1986, Translated
  from the Japanese by M. Reid. \MR{879273 (88h:13001)}

\bibitem[Mes72]{messingthesis}
William Messing, \emph{The crystals associated to {B}arsotti-{T}ate groups:
  with applications to abelian schemes}, Lecture Notes in Mathematics, Vol.
  264, Springer-Verlag, Berlin, 1972. \MR{0347836 (50 \#337)}

\bibitem[MM74]{Mazur-Messing}
B.~Mazur and William Messing, \emph{Universal extensions and one dimensional
  crystalline cohomology}, Lecture Notes in Mathematics, Vol. 370,
  Springer-Verlag, Berlin, 1974. \MR{0374150 (51 \#10350)}

\bibitem[Sch11]{Scholze:Perfectoid}
Peter Scholze, \emph{Perfectoid spaces}, Preprint, {arXiv:1111.4914} (2011).

\bibitem[Sch12]{Scholze:CdR}
\bysame, \emph{p-adic {Hodge} theory for rigid-analytic varieties}, Preprint,
  {arXiv:1205.3463} (2012).

\bibitem[SGA]{SGA}
\emph{S\'eminaire de {G\'eom\'etrie Alg\'ebrique} du {B}ois {M}arie,
  {I}--{VII}}, Lecture Notes in Mathematics (1970-1973), Vols. 151, 152, 153,
  225, 269, 270, 288, 305, 340, 569, and 589, Springer-Verlag, Berlin;
  Documents Math\'ematiques (Paris) [Mathematical Documents (Paris)], 3 and 4,
  Soci\'et\'e Math\'ematique de France, Paris.

\bibitem[Val75]{Valabrega:ExcellencePS}
Paolo Valabrega, \emph{On the excellent property for power series rings over
  polynomial rings}, J. Math. Kyoto Univ. \textbf{15} (1975), no.~2, 387--395.
  \MR{0376677 (51 \#12852)}

\bibitem[Vas12]{Vasiu:ys}
Adrian Vasiu, \emph{A motivic conjecture of {Milne}}, to appear in J. Reine
  Angew. Math. (2012).

\bibitem[VZ10]{ZinkVasiu:BreuiloverRegularLocal}
Adrian Vasiu and Thomas Zink, \emph{Breuil's classification of {$p$}-divisible
  groups over regular local rings of arbitrary dimension}, Algebraic and
  arithmetic structures of moduli spaces ({S}apporo 2007), Adv. Stud. Pure
  Math., vol.~58, Math. Soc. Japan, Tokyo, 2010, pp.~461--479. \MR{2676165}

\bibitem[Zin02]{Zink:DisplayFormalGpAsterisq278}
Thomas Zink, \emph{The display of a formal {$p$}-divisible group},
  Ast{\'e}risque (2002), no.~278, 127--248, Cohomologies $p$-adiques et
  applications arithm{{\'e}}tiques, I. \MR{MR1922825 (2004b:14083)}

\end{thebibliography}
\bibliographystyle{amsalpha}

\end{document}